\documentclass[12pt]{article}
\usepackage[margin =1in]{geometry}
\usepackage{amssymb,amsthm}
\usepackage{amsmath}
\usepackage{enumerate}
\usepackage{enumitem}
\usepackage{mathtools}

\usepackage{hyperref}
\usepackage{subcaption}
\usepackage{cite}
\usepackage{color}

\usepackage{graphicx}
\numberwithin{equation}{section}

\newcommand{\range}{\mbox{\rm range}\,}
\newcommand{\epi}{\mathrm{epi}\,}
\newcommand{\cH}{\mathcal{H}\,}

\newcommand{\inclu}[0] {\ar@{^{(}->}}

\newcommand{\spann}{\text{span}}

\newcommand{\gph}{{\rm gph}\,}

\newcommand{\dist}{{\rm dist}}
\newcommand{\R}{\mathbb{R}}

\newcommand{\cU}{\mathcal{U}}
\newcommand{\cR}{\mathcal{R}}

\newcommand{\sign}{\mathrm{sign}}

\newcommand{\RR}{\mathbb{R}}
\newcommand{\stableset}{W_{\mathrm{loc}}^{\mathrm{cs}}}
\newcommand{\mstep}{\hat \rho}

\newcommand{\modelweak}{\eta}
\newcommand{\functionweak}{\rho}
\newcommand{\approxacc}{\beta}
\newcommand{\algstep}{\tau}
\newcommand{\cK}{\mathcal{K}}

\newcommand{\lf}{\operatornamewithlimits{liminf}}

\newcommand{\prox}{{\rm prox}}
\newcommand{\dom}{\text{dom}\,}

\newcommand{\argmin}{\operatornamewithlimits{argmin}}

\newcommand{\NN}{\mathbb{N}}

\newtheorem{thm}{Theorem}[section]
\newtheorem{definition}[thm]{Definition}

\newtheorem{lem}[thm]{Lemma}
\newtheorem{cor}[thm]{Corollary}

\theoremstyle{remark}

\newcommand{\cM}{\mathcal{M}}

\newcommand{\cX}{\mathcal{X}}
\newcommand{\cL}{\mathcal{L}}
\newcommand{\cP}{\mathcal{P}}

\usepackage{mathtools}

\usepackage[boxruled]{algorithm2e}

\begin{document}
	
	\title{Proximal methods avoid active strict saddles\\ of weakly convex functions}
	
	
	\author{Damek Davis\thanks{School of Operations Research and Information Engineering, Cornell University,
Ithaca, NY 14850, USA;
\texttt{people.orie.cornell.edu/dsd95/}.}\and 
Dmitriy Drusvyatskiy
\thanks{Department of Mathematics, University of Washington,
Seattle, WA 98195, USA;
\texttt{sites.math.washington.edu/~ddrusv/}. Research of Drusvyatskiy was supported by the NSF DMS
1651851 and CCF 1740551 awards.}
}

	\date{}
	\maketitle
\begin{abstract}
We introduce a geometrically transparent strict saddle property for nonsmooth functions. This property guarantees that simple proximal algorithms on weakly convex problems converge only to local minimizers, when randomly initialized. We argue that the strict saddle property may be a realistic assumption in applications, since it provably holds for generic semi-algebraic optimization problems. 
\end{abstract}

\bigskip
	

\section{Introduction}
Nonconvex optimization techniques
	are increasingly playing
a major role in modern signal processing, high dimensional statistics, and machine learning.
A driving theme, 
	fully supported by empirical evidence, 
	is 
that simple algorithms 
	often work well 
		in 
		highly nonconvex 
			and 
		even nonsmooth 
	settings.
Gradient descent, for example, 
	often 
		finds     
	points 
		with 
		small objective value, 
		despite existence 
			of 
			many highly suboptimal critical points. 
A growing body 
	of 
	literature 
		provides 
	one compelling explanation 
		for 
		this phenomenon. 
Namely, typical smooth objective functions 
	provably satisfy 
the {\em strict saddle property},
meaning 
each critical point 
	is 
either 
	a local minimizer 
			or 
	has 
	a direction 
		of 
		strictly negative curvature (e.g.,~\cite{ge2016matrix,sun2015nonconvex,bhojanapalli2016global,ge2017no,sun2018geometric}).
For such functions, 
	randomly initialized gradient-type methods
		provably converge 
	to local minimizers,
		escaping all strict saddle points~\cite{gradient_descent_jason, panageas2016gradient}.
Moreover, 
stochastically perturbed gradient methods 
	escape 
strict saddles efficiently,  
	indeed,
	in polynomial time \cite{du2017gradient,jin2017escape,ge2015escaping}.

Smoothness of the objective 
	plays 
a crucial role in the existing literature
		on 
		saddle avoidance.
Some extensions 
	to 
	constrained optimization do
		exist.  
	The papers \cite{criscitiello2019efficiently,sun2019escaping,ge2015escaping} 
		investigate 
	saddle point avoidance for the problem of minimizing a smooth functions over a smooth manifold.
The works \cite{hallakfinding,mokhtari2018escaping,nouiehed2018convergence}
	propose
algorithms
	for 
	minimizing a smooth objective
		over 
		a closed convex set.
At each step 
	of 
	these algorithms, 
	one
		must minimize 
	a nonconvex quadratic
		over 
		a certain convex set (an NP hard problem in general).
The work~\cite{noisysticky} 
	proposes 
a polynomial time first-order algorithm 
	for 
	minimizing a smooth objective 
		over 
	linear inequality constraints.\footnote{This work appeared concurrently with our manuscript.}
At each step 
	of 
	this algorithm, 
	one 
		identifies 
	the ``active linear constraints"
			and 
	then 
		attempts   
	to find a ``second-order stationary point"
		of 
		the loss 
			in 
			the restricted subspace.

Though impressive in scope,
	existing work 
		has yet 
	to answer the following question:
\begin{quote}
Do simple algorithms 
	on 
	typical 
		nonsmooth 
			and 
		nonconvex 
	optimization problems 
	converge 
only to local minimizers?
\end{quote}

\noindent 
This question as stated 
	is 
purposefully vague, 
	since  
		``simple algorithms'' and ``typical optimization problems'' 
			can be interpreted 
		in multiple ways. 
Let us 
	try to formalize 
both ideas.
First, 
	if one 
		believes 
	that gradient descent 
		is 
	a canonical first-order method 
		for 
		smooth minimization, 
		it 
			is 
		natural to focus 
			on 
			three concrete algorithms 
				for 
					nonsmooth 
						and 
					nonconvex 
				problems:
			the 
			proximal point \cite{MR0290213,MR0298899,MR0410483,MR0201952}, 
			proximal gradient \cite{nesterov2013gradient,beck}, 
				and 
			proximal linear \cite{nesterov2007modified,burke_com,prox_error,comp_DP,prox} methods.
This is the path we follow in the current work.

The latter issue, 
	identifying a typical optimization problem, 
		is 
	more subtle. 
To motivate our approach,
	let us 
		revisit 
	the following question: 
	why 
		is 
	the strict saddle property a reasonable assumption 
		for 
		smooth minimization? 
A first compelling reason 
	is 
that the property 
	holds 
in practice 
	for 
		specific problems 
		of 
		 interest. 
There 
	is, however, 
a more classical justification, 
	one rooted in stability to perturbations. 
Namely, 
	consider the task 
		of 
	minimizing a smooth function $f$ 
		on 
		$\R^d$. 
Then for 
	a full measure set 
		of 
		perturbations $v\in\R^d$, 
	the perturbed function $x\mapsto f(x)-\langle v,x\rangle$ 
		is guaranteed to satisfy 
	the strict saddle property---a direct consequence of Sard's theorem. 
Consequently, 
	in a precise mathematical sense, 
	the strict saddle property 
		holds {\em generically} 
	in smooth optimization. 
This justification 
	does not suggest 
one 
	can omit verification of 
	the strict saddle property 
		in 
		concrete circumstances, but 
it
		does suggest 
	that the strict saddle property 
		is 
	widely prevalent.

Seeking 
	to identify 
a similarly reasonable class of nonsmooth objectives
	on 
	which simple algorithms converge to local minimizers, the current paper accomplishes the following.
\begin{quote}
We 
 	formulate 
natural geometric conditions, 
	guaranteeing 
		the proximal point, 
		proximal gradient, 
	 		and 
		proximal linear algorithms 
		escape 
	all saddle points. 
Moreover, 
the proposed conditions 
	are 
generic: 
	they 
		hold 
	for almost all linear perturbations 
		of 
			weakly convex 
				and 
			semi-algebraic problems.
\end{quote}
The class of optimization problems we consider is broad. A function $f$
is called
\emph{$\rho$-weakly convex}
if the assignment $x\mapsto f(x)+\frac{\rho}{2}\|x\|^2$ 
is convex for some $\rho>0$.\footnote{Weakly convex functions also go by other names such as lower-$C^2$, uniformly prox-regularity, paraconvex, and semiconvex. We refer the reader to the seminal works on the topic \cite{fav_C2,prox_reg,Nurminskii1973,paraconvex,semiconcave}.}
Common examples include 
 pointwise maxima 
of 
smooth functions 
and 
all compositions 
of 
Lipschitz convex functions 
with 
smooth maps. For detailed contemporary examples, 
we 
refer 
the reader to \cite{davis2019stochastic,drusvyatskiy2017proximal,charisopoulos2019low,duchi2018stochastic,jin2019local}. The genericity guarantee applies to semi-algebraic functions\footnote{A function is called semi-algebraic if its graph
	decomposes 
	into a finite union 
	of 
	sets, each defined by finitely many polynomial inequalities.}, and more broadly, to those that are definable in an o-minimal structure---a  virtually exhaustive function class in applications.

\subsection{The role of curvature}

To motivate our core geometric conditions,  
	we revisit the role that curvature
		plays 
	in saddle-point avoidance. 
Setting the stage 
	for 
	the rest of the paper, 
	consider the task 
		of 
		minimizing a weakly convex function $f$ 
			on 
			$\R^d$.
First-order optimality conditions 
	show 
that any local minimizer $\bar x$
	of 
	$f$ 
		satisfies 
		the {\em criticality condition}:
$$df(\bar x)(v)\geq 0\quad \textrm{for all } v\in \R^d,$$
where $df(\bar x)(v)$ 
	denotes 
the directional derivative 
	of 
	$f$ at $\bar x$ 
		in 
		direction $v$ (see Definition~\ref{defn:var_constr}). 
Conversely, 
	sufficient conditions 
		for 
		local optimality 
			at 
			a critical point $\bar x$ 
		require 
	a close look 
		at 
		the second-order variations 
			of 
			$f$
			along particular directions, 
	namely those where the directional derivative is zero.
Mirroring the
		smooth setting, 
	one may naively  call 
	a critical point $\bar x$ a strict saddle 
			if 
				there exists 
			a direction $v$ such that 
				$df(\bar x)(v)=0$ 
					and 
				$f$ decreases quadratically along $v$. 
This definition, however,
	is 
insufficient for saddle avoidance: 
simple 
	examples 
		show 
	that typical algorithms 
		can converge
	to such saddle points
		from 
		a positive measure 
			of 
			initial conditions.

\begin{center}
Negative curvature alone 
	does 
	not guarantee escape from saddle points.
 \end{center}

To illustrate what can go wrong, consider the example
\begin{equation}\label{eqn:pathological}
\min_{(x,y)\in\R^2}~ f(x,y)=(|x|+|y|)^2-2x^2,
\end{equation}
the graph of which
		is 
	shown 
		in 
		Figure~\ref{gph_bad_func}. 
First, 
	observe
that the origin
		meets
	the conditions of the candidate ``strict saddle" definition.
Indeed, 
 $f$ 
	is 
	differentiable
		at 
		the origin
			and 
		the origin is a critical point.
Moreover, 
	$f$ 
		decreases quadratically 
	along all directions making a small angle 
		with 
		the $x$-axis. 
Next, we
	turn to algorithm dynamics.
Figure~\ref{sublevel_flow_nonmooth} 
	depicts 
the subgradient flow $-\dot{\gamma}(t)\in \partial f(\gamma(t))$. 
From the picture 
	we
		find
	a positive measure cone, 
		 surrounding the $y$-axis 
		 	and
		consisting 
			of 
			origin-attracted initial conditions.		
Moreover, 
	we 
		show
	in Appendix~\ref{sec:bad_example_detail},  
		that 
		a typical algorithm---the proximal point method---initialized 
	anywhere within this cone 
		also converges 
	to the origin,
	illustrating the inadequacy of the definition.
While this argument 
	shows
that negative curvature
	is
	insufficient 
		for 
		nonsmooth optimization, 
it 
	can be made 
even more definitive 
	by 
	smoothing the problem at hand.
Namely, 
	an alternative view 
		of
	the proximal point method 
		recognizes
	that the dynamics of the algorithm 
		coincide
	with the dynamics of gradient descent 
		on 
		a $C^1$ smooth approximation 
			of 
			$f$, called the \emph{Moreau envelope} (see Section~\ref{sec:moreau}).
The smooth envelope, whose graph and gradient flow are shown in Figures~\ref{gph_bad_func_mooth} and \ref{sublevel_flow_mooth},
	has 
the same cone 
	of 
	directions 
		of
		second-order negative curvature as $f$,
			but despite its
				smoothness 
					and 
				negative curvature,
		gradient descent cannot escape the origin. The problem persists under a variety of different choices of the step-size.
Note that there is  no contradiction with 
the saddle avoidance property 
	of 
	gradient descent 
		on 
		smooth functions,
		since 
			the envelope
				is  
			not $C^2$, 
			but merely $C^1$-smooth
			around the origin.
Although this example 
	appears 
damning at first, 
	it 
		is
	 highly unstable,
	since 
	small linear tilts 
		of
		the function 
			do not exhibit 
		the same pathological behavior around critical points---a direct consequence of the forthcoming results. 
\begin{figure*}[h!]
	\centering
	\begin{subfigure}{0.5\textwidth}
		\centering
		\includegraphics[width=\textwidth]{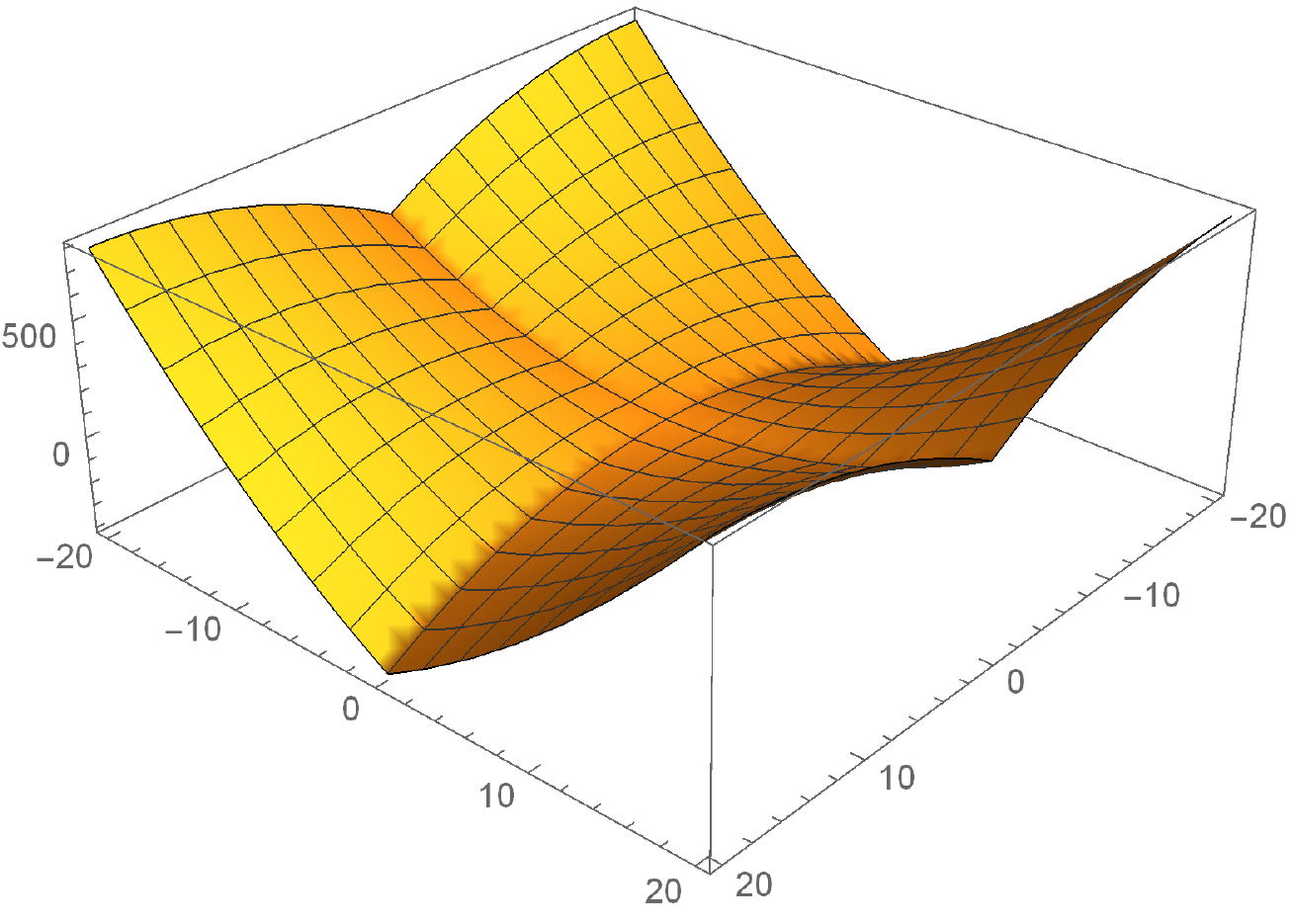}
		\caption{The function $\displaystyle f(x,y):=(|x|+|y|)^2-2x^2$}\label{gph_bad_func}
	\end{subfigure}%
	~ 
	\begin{subfigure}{0.5\textwidth}
		\centering
		\includegraphics[scale=0.47]{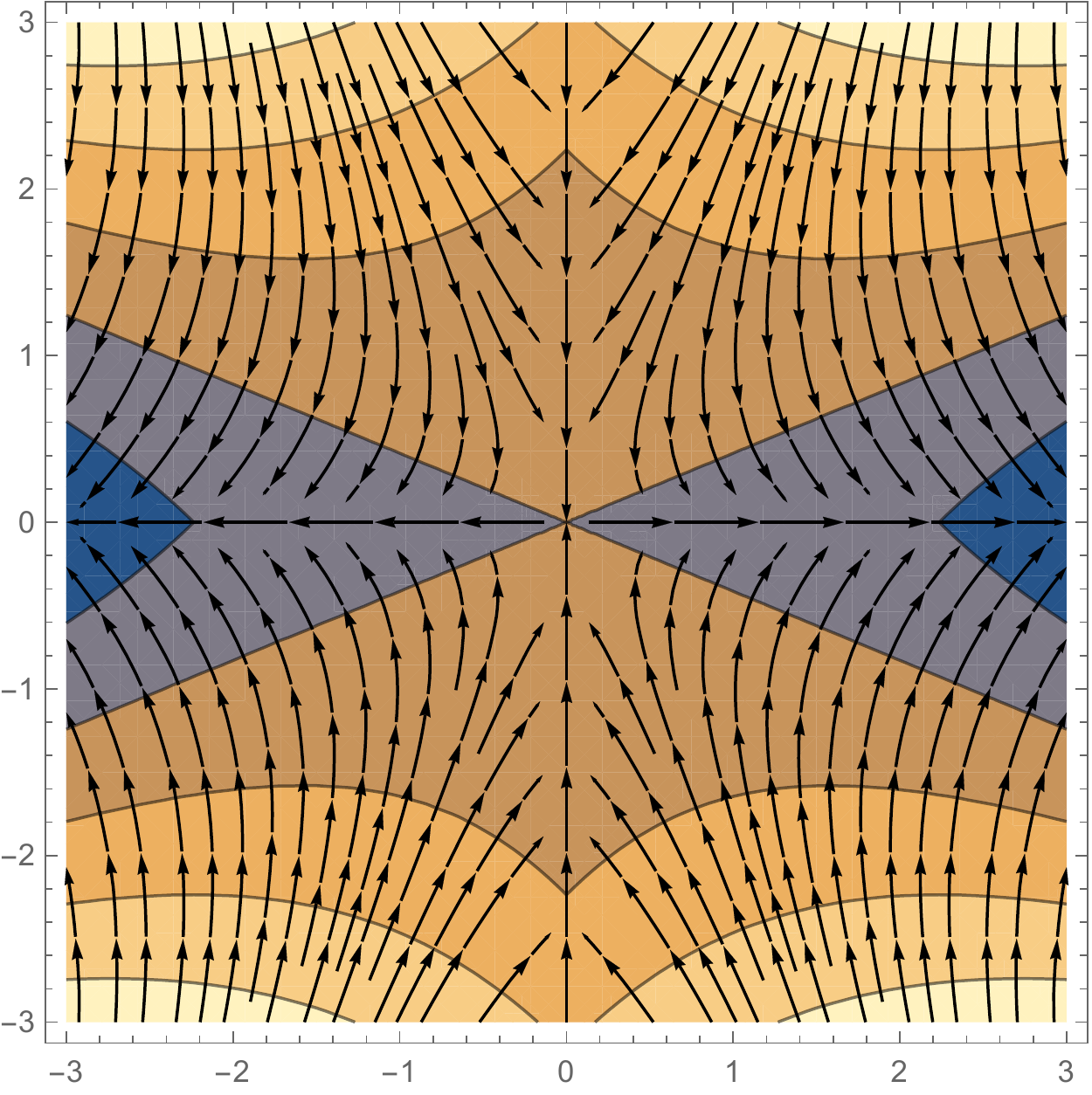}
		\caption{Subgradient flow $\dot{\gamma}\in -\partial f (\gamma)$}\label{sublevel_flow_nonmooth}
	\end{subfigure}
	\vspace{5pt}
	
		\begin{subfigure}{0.5\textwidth}
		\centering
		\includegraphics[width=\textwidth]{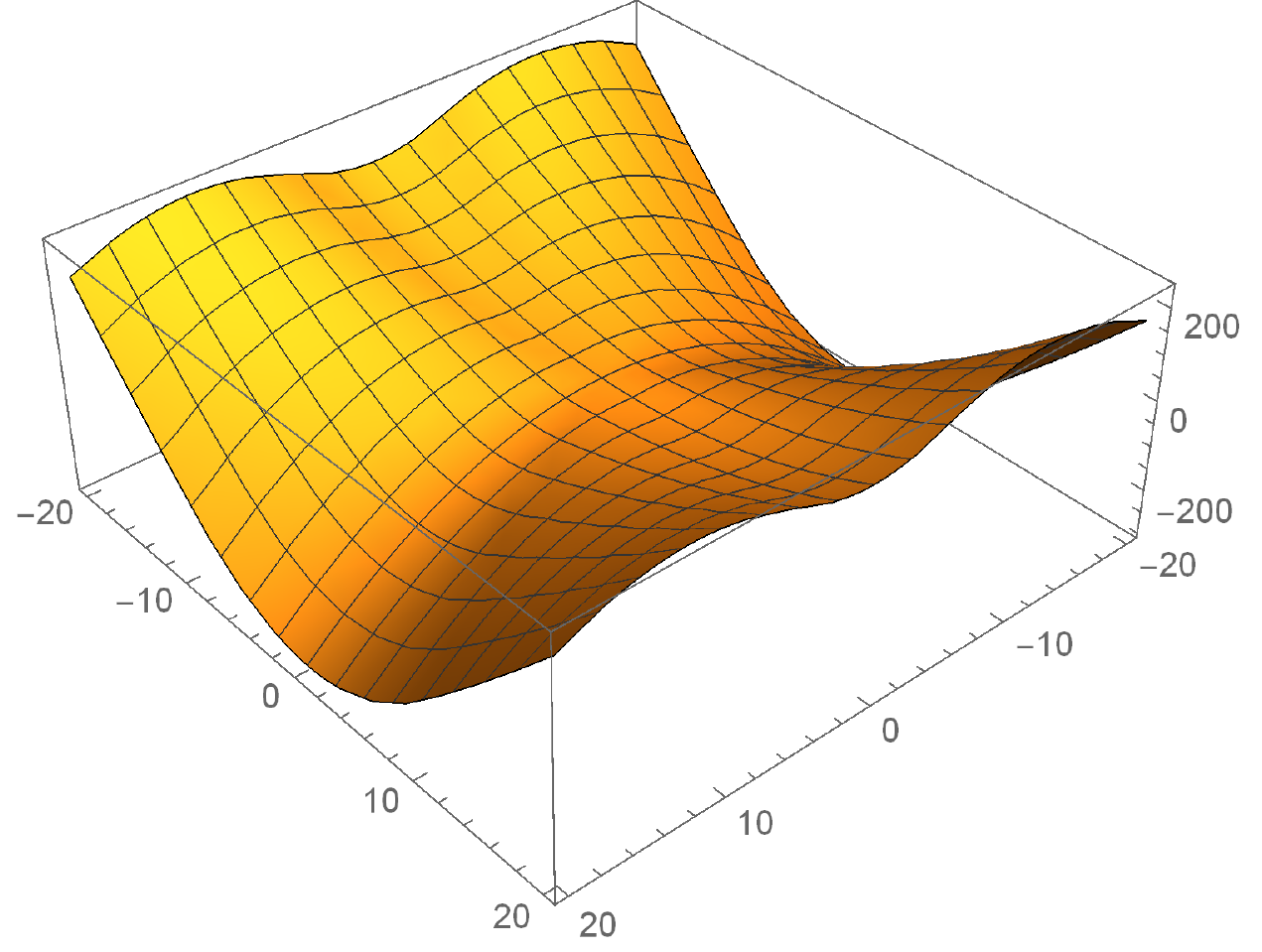}
		\caption{Envelope $\displaystyle \hat f(z):=\min_{z'\in\R^2} \left\{f(z')+3 \|z'-z\|^2\right\}$}\label{gph_bad_func_mooth}
	\end{subfigure}%
	~ 
	\begin{subfigure}{0.5\textwidth}
		\centering
		\includegraphics[scale=0.47]{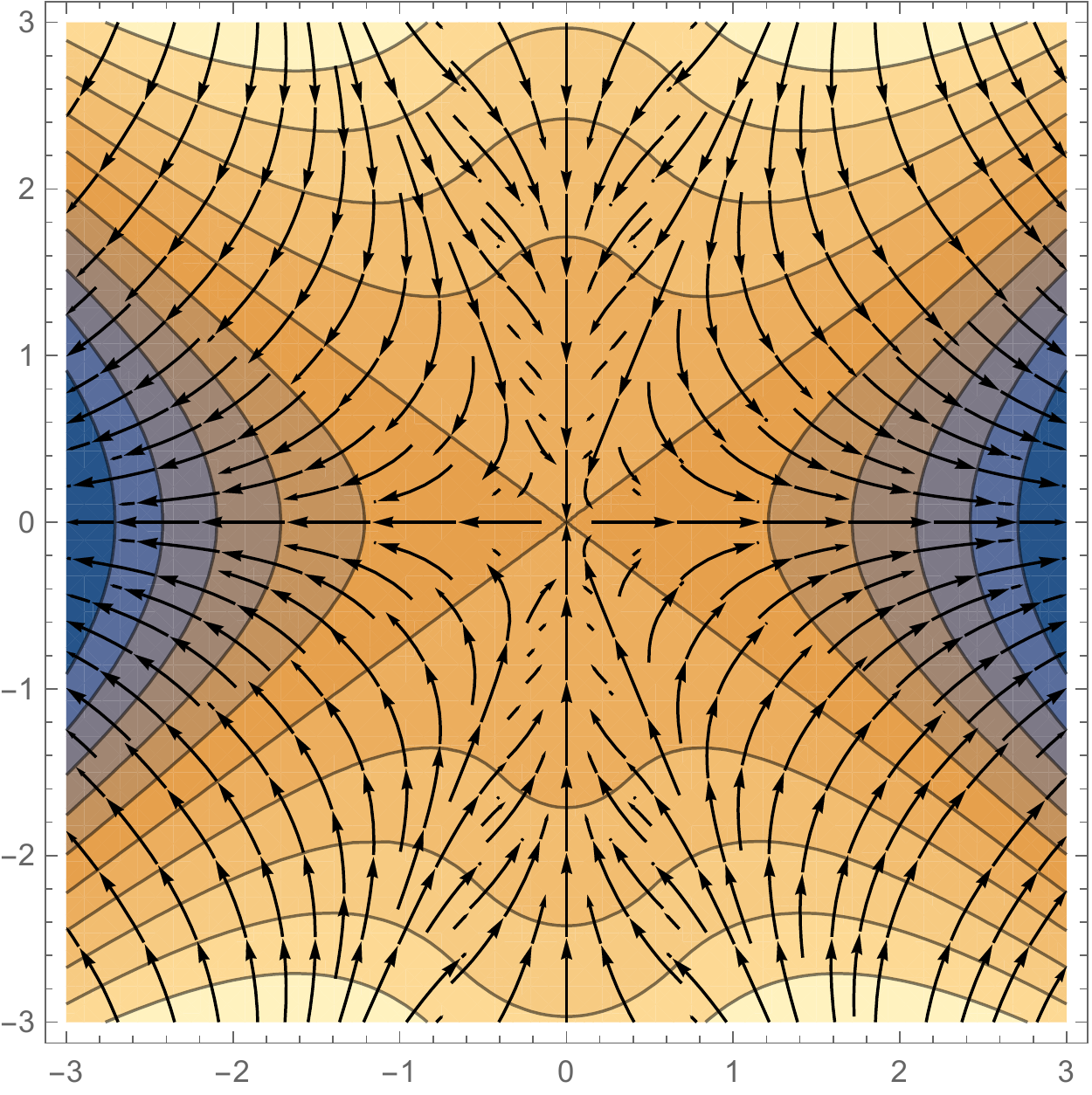}
		\caption{Gradient flow $\dot{\gamma}= -\nabla \hat f (\gamma)$}\label{sublevel_flow_mooth}
	\end{subfigure}
	\caption{The function $f$ in \eqref{eqn:pathological}, its Moreau envelope, and  their subgradient flows.}
\end{figure*}

\subsection{The role of the active manifold}
We 
	have seen 
that negative curvature alone
	is
insufficient for saddle avoidance. 
We
	argue 
here that 
	in 
	addition
		to negative curvature, 
	one
		must 
	make a structural assumption 
		on 
		the way nonsmoothness manifests.  
To 
	illustrate 
		and 
	contrast 
	with 
	example \eqref{eqn:pathological} above, 
	consider:
	\begin{equation}\label{eqn:p_smooth_ex}
\min_{(x,y)\in\R^2}~ g(x,y)=|x|-y^2.
\end{equation}
The graph 
	of $g$ 
		is 
	shown 
		in 
		Figure~\ref{subsec:partlysmooth_man}, 
	while its subgradient flow 
		appears 
	in 
	Figure~\ref{subsec:partlysmooth_grad}. 
Looking at the figure, 
we 
	see 
that the subgradient flow of $g$ 
	sharply contrasts
with that of the pathological example~\eqref{eqn:pathological}.
Indeed, 
while 
	both functions
		have 
	directions
		of 
		negative curvature, 
the set 
	of 
	origin-attracted initial conditions of the flow $-\partial g$
		is simply 
	the $x$-axis---a measure zero set.
This favorable behavior  
	of 
	$g$ 
	arises 
because its nonsmoothness manifests in a structured way: 
its unique critical point $\bar z$ (the origin) 
	lies 
on 
	a smooth manifold $\cM$ (the $y$-axis).
The function $g$ then
			varies 
		smoothly along $\cM$ 
				and
		sharply normal to $\cM$ meaning:
$$\inf \{\|v\|: v\in \partial g(z),~z\in U\setminus \cM\}>0,$$
where $U$ is some neighborhood of $\bar z$. 
Such ``active manifolds'' 
	have 
classical roots 
	in 
	optimization
		and
	serve 
as a far reaching
		generalization
			of 
			``active sets"
				in 
				nonlinear programming. 
Important references 
	include both
		the original works \cite{Al-Khayyal-Kyparisis91,Burke-More88, Burke90,Calamai-More87,Dunn87,Ferris91,Flam92}
	 	and 
		the more recent work 
			on identifiable surfaces \cite{Wright}, $UV$-decomposition \cite{LOS},
			partial smoothness \cite{lewis_active}, 
				and 
			cone reducibility \cite{Bon_Shap}. 
Here, 
	we most closely 
		follow
	the framework
		developed 
		in 
		\cite{ident}. 

\begin{figure*}[h!]
    \centering
    \begin{subfigure}{0.5\textwidth}
        \centering
        \includegraphics[width=\textwidth]{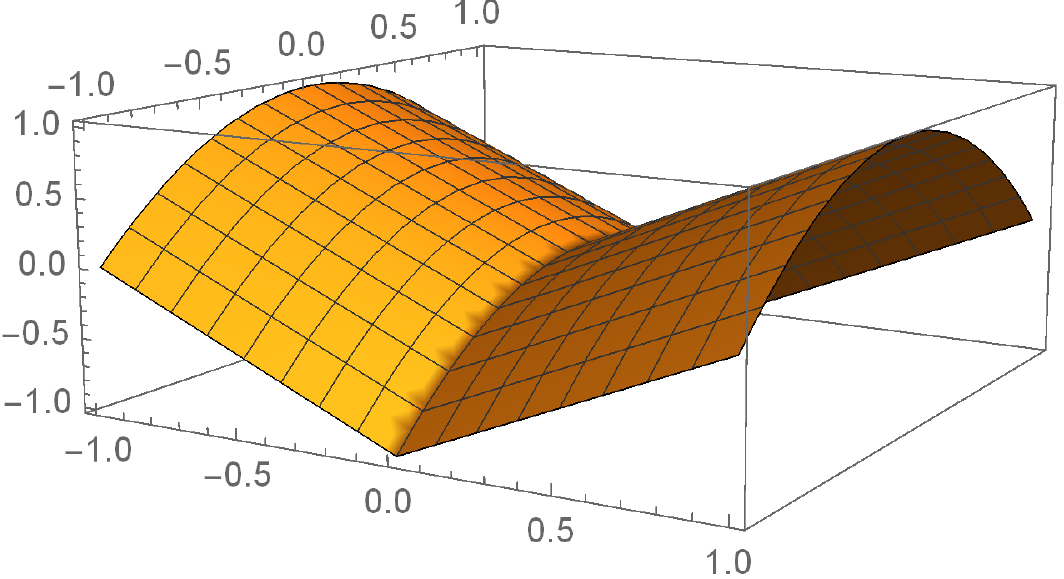}
        \caption{The function $g(x,y)$}\label{subsec:partlysmooth_man}
    \end{subfigure}%
    ~ 
    \begin{subfigure}{0.5\textwidth}
        \centering
        \includegraphics[scale=0.5]{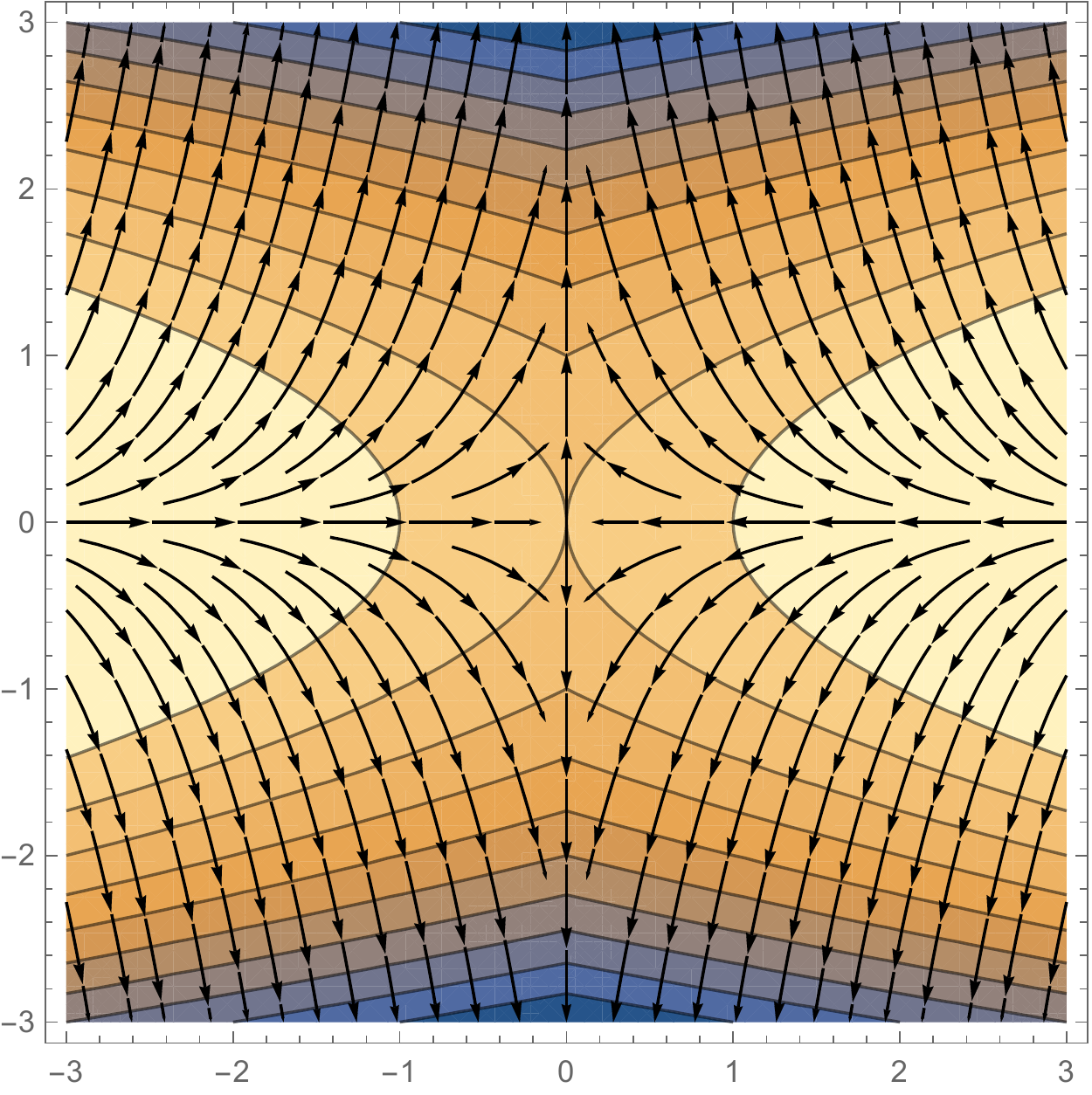}
        \caption{Subgradient flow $\dot{\gamma}\in -\partial g (\gamma)$}\label{subsec:partlysmooth_grad}
    \end{subfigure}
   \caption{The function $g(x,y)=|x|-y^2$ has an active manifold at the origin.}\label{fig:activemanifold}
\end{figure*}

\subsection{Escape from saddles by the center stable manifold theorem}\label{sec:intro_escape}
Formalizing the favorable behavior 
	of 
	 example \eqref{eqn:p_smooth_ex},
	we 
		will call 
	a critical point $\bar x$ 
		of 
		a function $g$ a {\em strict saddle} whenever 
			(i) $g$ admits an active manifold containing $\bar x$, 
				and 
			(ii) the function $g$ 
				decreases 
			quadratically along some direction $v$ 
			satisfying $dg(\bar x)(v)=0$. 
A function $g$ 
	is 
said to have the {\em  strict saddle property} 
	if each 
		of 
		its critical points 
			is 
		either 
			a local minimizer 
				or 
			a strict saddle.\footnote{Perhaps more appropriate would be the terms {\em active strict saddle} and the {\em active strict saddle property}. For brevity, we omit the word ``active."} 
Though it may seem that this definition 
	is stringent at first,
	the strict saddle property 
		is 
	in a precise mathematical sense generic. 
Namely, 
	it 
		follows 
	from \cite{drusvyatskiy2016generic} that 
		given any semi-algebraic weakly convex function $g$, 
		the perturbed function $g_v(x)=g(x)-\langle v,x\rangle$ 
			has 
		the strict saddle property for almost all $v\in\R^d$.\footnote{Weak convexity is not essential here, provided one modifies the definitions appropriately. Moreover, this guarantee holds more generally for functions definable in an o-minimal structure.}
In particular, 
	almost all linear perturbations 
		of 
	the function $f$ 
		in 
		\eqref{eqn:pathological} 
			do have 
		the  strict saddle property. That being said, it is important to note that under more nuanced perturbations, the strict saddle property may fail. For example, the classical NP-complete problem of checking copositivity of a matrix $A\in \R^{d\times d}$ amounts to verifying if the origin $\bar x=0$ locally minimizes the quadratic $x^TAx$ over the nonnegative orthant $\R^d_+$. It is straightforward to see that this constrained  problem  does not admit an active manifold at $\bar x$ for any matrix $A$.

With the definition 
	of 
a strict saddle  
		at hand, 
	we 
		can now outline 
	the main results
		of 
		the paper. 
As in the smooth setting, first explored in the seminal paper~\cite{gradient_descent_jason}, 	our arguments 
		will be based 
	on the center stable manifold theorem. 
Namely,
	we will interpret 
the three simple minimization algorithms 
	as 
	fixed point iterations 
$$x_{k+1}=S(x_{k})\qquad \textrm{ for some maps }S\colon\R^d\to\R^d.$$
Table~\ref{table:update_map} 
	lists 
the maps $S(\cdot)$ 
	for 
	the 
		proximal point, 
		proximal gradient, 
			and 
		proximal linear algorithms. 
In each case, 
the fixed points 
	of 
	$S(\cdot)$ 
		are 
	precisely the critical points 
		of 
		the minimization problem.

\begin{table}[h!]
\begin{center}
	\begin{tabular}{|l |l |l|} 
		\hline
		Algorithm& Objective  & Update function $f_x(y)$ \\ [0.5ex] 
		\hline\hline
		Prox-point  & $\displaystyle  r(x)$ & $\displaystyle r(y)+\tfrac{1}{2\mu}\|y-x\|^2$ \\ 
		\hline
		Prox-gradient  & $\displaystyle g(x)+r(x)$  & $\displaystyle g(x)+\langle \nabla g(x),y-x\rangle+r(y)+\tfrac{1}{2\mu}\|y-x\|^2$  \\
		\hline
		Prox-linear  & $\displaystyle h(F(x))+r(x)$ & $\displaystyle  h(F(x)+\nabla F(x)(y-x))+r(y)+\tfrac{1}{2\mu}\|y-x\|^2$  \\
		\hline
	\end{tabular}
\end{center}
\caption{The three algorithms with the update $S(x)=\argmin_y f_x(y)$; we assume $h$ is convex, $r$ is weakly convex, and both $g$ and $F$ are smooth.}\label{table:update_map}
\end{table}

To put our guarantees in context, 
	it 
		will be 
	useful to recall the center stable manifold theorem. 
To this end, 
	suppose that the iteration map $S(\cdot)$ 
		is 
	$C^1$-smooth 
		on 
		a neighborhood 
			of 
			some fixed point $\bar x$. 
Then $\bar x$ 
	is 
called an {\em unstable fixed point} 
	of 
		$S$ 
		if the Jacobian $\nabla S(\bar x)$ 
			has 
		at least one eigenvalue whose magnitude 	
			is 
		strictly greater than one. 
The center stable-manifold theorem \cite[Theorem III.7]{shub2013global}
	guarantees 
the following: 
	if $\bar x$ is an unstable fixed point 
		of 
		$S$ and the Jacobian $\nabla S(\bar x)$ is invertible, 
	then almost all initializers $x$ 
		in 
		a neighborhood $U$ 
			of 
			$\bar x$ 
		generate 
	iterates $\{S^k(x)\}_{k\geq 1}$ that eventually escape the neighborhood. 
More precisely, 
	the theorem 
		guarantees
	that the set 
		of 
		initial conditions
$$\left\{x\in U\colon S^k(x)\in U\textrm{ for all } k\geq 1\right\}$$
	has zero
Lebesgue measure. 
All that 
	is needed 
to globalize this guarantee 
	is 
to ensure that the preimage $S^{-1}(V)$ 
	of 
	any measure zero set $V$ 
		is 
	itself measure zero. 
Then for almost all initial conditions $x\in\R^d$, 
	the limit $\lim_{k\to\infty}S^k(x)$, when it exists, 
		is 
	not an unstable fixed point 
		of $S$. 
A straightforward way to ensure that the inverse $S^{-1}$ 
respects null sets 
	is 
by introducing the relaxation map:
\begin{equation}\label{eqn:tmap}
T(x):=(1-\alpha)x+\alpha S(x).
\end{equation}
Both 
	$T$ 
		and 
	$S$ 
	have 
		the same fixed points, and any fixed point $\bar x$ at which  $\nabla S(\bar x)$ has a real eigenvalue strictly greater than one is an unstable fixed point of $T$. Moreover, 
	if the map $S$ 
		is 
	Lipschitz, 
	then the inverse $T^{-1}$ 
		preserves 
	null-sets for sufficiently small $\alpha\in (0,1)$.

\subsection{The main results}
We 
	can now summarize 
our main results:
\begin{quote}
We
	show
that around each strict saddle
	of 
	the problem, 
each 
	of 
	the iterations maps $S(\cdot)$  
		in 
		Table~\ref{table:update_map} 
	is 
$C^1$ smooth.
	Moreover,  
	 if $\bar x$ is a strict saddle, then the Jacobian $\nabla S(\bar x)$ has a real eigenvalue strictly greater than one. 
\end{quote}
From this result, 
the center stable manifold theorem 
	guarantees
that iteration~\eqref{eqn:tmap} locally escapes strict saddles. 
Seeking to globalize the guarantees, 
	we 
		compute 
	the global Lipschitz constants 
		for 
			the 
				proximal point 
					and 
				proximal gradient 
				methods. 
We 
	deduce 
that, when randomly initialized, 
the relaxed iterations \eqref{eqn:tmap} 
	for 
	both the
		proximal point 
			and 
		proximal gradient 
		methods
			converge 
		to local minimizers 
			of 
			weakly convex functions, 
			provided they 
				have 
			the  strict saddle property. 
On the other hand, 
without placing further restrictions on the problem data, 
we 
	are unable 
to compute the global Lipschitz constant 
	of 
	the map $S(\cdot)$ corresponding 
		to 
		the proximal linear algorithm.
We 
	leave 
it as an intriguing open question to determine Lipschitz properties 
	of 
	the proximal linear update.

The outlined results 
	may seem 
surprising at first: 
	the optimization problem 
		is 
	nonsmooth 
		and 
	yet we 
		prove 
	the iteration maps $S(\cdot)$ 
		are 
	$C^1$-smooth around any strict saddle. 
The reason 
	is
transparent 
	and 
derives 
	from 
	the interplay 
		between 
	the active manifold 
		and 
	weak convexity. 
Take the proximal point method for example. 
The very definition 
	of
	the active manifold 
		guarantees 
	that the fixed point iteration $S(\cdot)$ 
		maps 
	an entire neighborhood $\cX$ around an strict saddle $\bar x$ 
		{\em into} 
	the active manifold $\cM$. 
Consequently, 
	for all $x\in \cX$, 
		the update $S(x)$ 
			can be realized 
		as a minimizer 
			of
			a smooth function 
				over 
				the active manifold:
\begin{equation}\label{eqn:reduction}
S(x)=\argmin_{y\in \cM} ~f(y)+\frac{1}{2\mu}\|y-x\|^2.
\end{equation}
Weak convexity, in turn, 
 	ensures 
that $S(\bar x)$ 
	satisfies 
a quadratic growth condition 
	for the  problem \eqref{eqn:reduction}, 
		which by classical perturbation theory 
			guarantees 
		that $S(\cdot)$ 
			is 
		$C^1$-smooth 
			on 
			a neighborhood of $\bar x$. 
It 
	only remains 
to argue that the negative curvature 
	of 
	the objective function 
		at 
		$\bar x$ 
			implies 
		that the Jacobian $\nabla S(\bar x)$ 
			has 
		at least one real eigenvalue greater than one. 
Though this computation 
	is 
straightforward 
	for 
	the proximal point method, 
	it 
		becomes 
	more interesting (and surprising)
		for 
		the 
			proximal gradient 
				and 
			proximal linear algorithms.

\bigskip
\textbf{Roadmap.} The outline of the paper is as follows. 
Section~\ref{sec:prelim}
	is
a self-contained presentation 
	of
	the necessary preliminaries 
		for 
		formalizing the ideas of the introduction.
Then in Sections~\ref{sec:prox_point},~\ref{sec:prox_grad}, and~\ref{sec:prox_linear}
	we 
		directly analyze 
	the iteration maps 
		for 
			the 
				proximal point, 
				proximal gradient, 
					and
				proximal linear 
			algorithms. Section~\ref{sec:makeitrain} establishes iterate convergence of the relaxed schemes \eqref{eqn:tmap} under the Kurdyka-{\L}ojasiewicz property.

\section{Preliminaries}\label{sec:prelim}
Throughout, we follow standard notation in convex and variational analysis, as set out for example in the monographs \cite{RW98,CLSW,con_ter,mord1}.
We consider a Euclidean space $\R^d$ endowed with an inner product $\langle\cdot,\cdot \rangle$ and the induced norm $\|x\|=\sqrt{\langle x, x\rangle}$. The unit sphere in $\R^d$ will be denoted by $\mathbb{S}^{d-1}$.
 For any function $f\colon\R^d\to\R\cup\{\infty\}$, the {\em domain} and {\em epigraph} are the sets
\begin{equation*}
\dom f=\{x\in \R^d: f(x)<\infty\},\qquad 
\epi f=\{(x,r)\in \R^d\times\R: r\geq f(x)\},
\end{equation*}
respectively.
The function $f$ is called {\em closed} if $\epi f$ is a closed set. 
 For any set $\cM\subset\R^d$, the indicator function $\delta_{\cM}$ evaluates to zero on $\cM$ and to $+\infty$ off it. For any function $f\colon\R^d\to \R\cup\{\infty\}$ and a set $\cM\subset\R^d$, we define the restriction $f_{\cM}:=f+\delta_{\cM}$. Throughout, the symbol $o(r)$ will denote any univariate function  satisfying $o(r)/r\to 0$ as $r\searrow 0$.

Consider a differentiable mapping $F(x)=(F_1(x),\ldots,F_m(x))$
	from 
	$\R^d$ to $\R^m$. 
Throughout, 
	the symbol $\nabla F(x)\in \R^{m\times d}$ 
		will denote 
	the Jacobian matrix,
		whose 
		$ij$'th entry 
			is given 
		by $\frac{d}{dx_j}F_i(x)$. 
Thus row $i$ 
	of 
	$\nabla F(x)$ 
		is the gradient 
			of 
			the coordinate function $F_i(x)$. 
In the particular case $m=1$, 
	we will treat $\nabla F(x)$ 
		either 
			as a column 
				or 
			as a row vector, 
		depending on context. 
For a $C^2$-smooth function $g\colon\R^d\times\R^n\to\R^m$, 
	we 
		partition 
	the Hessian as follows:
$$
\nabla^2 g(x, y) = \begin{bmatrix} \nabla_{xx}g(x,y) & \nabla_{xy}g(x,y) \\ \nabla_{yx}g(x,y) & \nabla_{yy} g(x,y)\end{bmatrix}
$$

\subsection{Subdifferentials and subderivatives}\label{sec:subdiff}
The following definition records 
		 the standard
			first 
				and 
			second-order 
		differential constructions, 
	which we will use in the paper.
After the definition, 
	we 
		will comment 
	on the role of each construction. For further details we refer the reader to \cite[Definitions 8.1, 8.3, 13.59]{RW98}.

\begin{definition}[Subdifferential and subderivatives]\label{defn:var_constr}
{\rm
Consider 
	a function $f\colon\R^d\to\R\cup\{\infty\}$ 
		and 
	a point $\bar{x}$ 
	with $f(\bar{x})$ finite. Then the {\em subdifferential} of $f$ at $\bar x$, denoted $\partial f(\bar x)$, consists of all vectors $v$ satisfying 
\begin{equation*} 
f(x)\geq f( \bar x)+\langle v,x- \bar x\rangle +o(\|x- \bar x\|)\qquad \textrm{ as }x\to  \bar x.	
\end{equation*}
The {\em subderivative} of $f$ at $\bar{x}$ in direction $\bar u\in\R^d$ is
$$df(\bar{x})(\bar u):=\lf_{\substack{t\searrow 0\\ u\to\bar{u}}} \frac{f(\bar{x}+tu)-f(\bar{x})}{t}.$$
The {\em critical cone of} $f$ {\em at} $\bar{x}$ {\em for} $\bar {v}\in\R^d$ is 
$$C_{f}(\bar{x},\bar {v}):=\{u\in \R^d: \langle \bar {v},u\rangle=df(\bar{x})(u)\}.$$
The {\em parabolic subderivative of} $f$ {\em at} $\bar{x}$ {\em for} $\bar u\in \dom df(\bar x)$ {\em with respect to} $\bar w$ is 
$$d^2f(\bar{x})(\bar {u}|\bar {w})=\lf_{\substack{t\searrow 0\\ w\to\bar{w}}} \frac{f(\bar{x}+t\bar {u}+\frac{1}{2}t^2w)-f(\bar{x})-df(\bar{x})(\bar {u})}{\frac{1}{2}t^2}.$$}
\end{definition}

We 
	now comment
	on 
	these definitions, in order. First, 
	a vector $v$ lies in the subdifferential $\partial f(\bar x)$ 
		precisely when  
			the affine function $x\mapsto f(\bar x)+\langle v,x-\bar x\rangle$ 
				minorizes 
			$f$ up to first-order near $\bar x$.
The definition reduces to familiar objects in classical circumstances. 
For example, 
	differentiability of $f$ at $\bar x$
		implies
	the set $\partial f(\bar x)$ is a singleton, 
		containing only the gradient $\nabla f(\bar x)$.
Convexity of $f$ too entails a simplification, wherein 
	 $\partial f(\bar x)$ 
		reduces 
	to the subdifferential 
		of 
		convex analysis. 

While the subdifferential 
	encodes 
the set 
	of 
	approximate affine minorants,
	the subderivative
		measures the maximal instantaneous rate 
		of 
		decrease 
			of 
			$f$ 
			in 
			direction $\bar{u}$. 
Like the subdifferential, the subderivative reduces to familiar objects in classical circumstances. 
For example, 
	if $f$ 
		is 
	locally Lipschitz at $\bar{x}$,
		then 
	one
		may set
	$u=\bar{u}$
		in 
		its defining expression. 
Simplifying further,
	if $f$ 
		is 
	differentiable 
		at $\bar x$, 
	we 
		recover 
	the directional derivative expression $df(\bar{x})(\bar{u})=\langle \nabla f(\bar x),\bar u\rangle$. 
Finally,  
	if $f$ 
		is 
	 convex, 
	then the subderivative 
		reduces 
	to the support function of the subdifferential
$$df(\bar{x})(\bar{u})=\sup\{\langle \bar{u},v\rangle: v\in\partial f(\bar x)\},$$
highlighting the dual roles of the subdifferential and subderivative constructions.

For smooth losses,
	necessary optimality conditions 
		entail 
	vanishing gradients, 
	while 
	sufficient optimality conditions 
		follow from
	second order growth properties of $f$. 
Similar characterizations 
	persist
in the nonsmooth setting.
In particular, the subderivative and the subdifferential 
		feature
	in first-order necessary optimality conditions, 
	where 
		the (dual) criticality condition $0\in\partial f(\bar x)$ 
			is equivalent to 
		the (primal) nonnegativity condition
\begin{equation}\label{eqn:first_order_crit_def}
df(\bar{x})(u)\geq 0\qquad \textrm{for all }u\in\R^d.
\end{equation}
A point $\bar x$ satisfying these first-order necessary conditions \eqref{eqn:first_order_crit_def} is thus called {\em critical} for $f$. 
Sufficient optimality conditions, on the other hand, 
	make use of second-order variations
		of
		$f$. 
Namely, 
	suppose that a point $\bar x$ 
		is 
	critical for $f$ 
		and 
	consider a direction $\bar u\in\R^d$. 
There are two possibilities to consider. 
On the one hand, 
	if $df(\bar{x})(\bar u)>0$, 
	then $f$ 
		must locally increase 
	in direction $\bar u$. 
On the other hand, 
	if $df(\bar{x})(\bar u)=0$, 
	then 
		we 
			must examine
	second order variations of $f$ to determine local optimality. 
Such directions of ambiguity for the subderivative
	make up 
the critical cone $C_f(\bar x, 0)$.
For these directions,
	we 
		must look to 
	the parabolic derivative $d^2f(\bar{x})(\bar{u}|\bar{w})$,
		a measurement 
			of 
		the second order variation 
			of 
			$f$ along a parabolic arc 
				with 
				tangent direction $\bar u$ 
					and 
				second-order variation $\bar w$. 
This construction too simplifies when $f$ is $C^2$ smooth at $\bar x$, 
	reducing to the familiar second order variation:
$$d^2f(\bar{x})(\bar{u}|\bar{w})=\langle\nabla^2 f(\bar{x})\bar{u},\bar{u}\rangle.$$
This relation 
	suggests 
second-order optimality conditions
	for
	nonsmooth problems.  
Although we will not appeal to such conditions directly in this work, we record them here for completeness.
If $\bar{x}$ 
		is 
		a local minimizer of $f$, 
		then 
			$df(\bar{x})(u)\geq 0 \textrm{ for all } u\in \R^n$,
				and 			moreover $\inf_{w\in \R^n} d^2f(\bar{x})(u|w)\geq 0$ for any nonzero $u\in C_f(\bar{x},0)$. 
	Complementing this necessary condition, 
		a large class of functions, those that are \emph{parabolically regular,}
			may also be endowed
		with a sufficient optimality condition. 
Namely, 
	if $df(\bar{x})(u)\geq 0 \textrm{ for all } u\in \R^n$ 
		and 
	  $\inf_{w\in \R^n} d^2f(\bar{x})(u|w)> 0$ for any nonzero $u\in C_f(\bar{x},0)$, 	
	 then $\bar{x}$ 
	  	is 
	a local minimizer 
		of 
		$f$. We refer the reader to 
\cite{Bon_Shap} or \cite[Theorem 13.66]{RW98} for details.

\subsection{Smooth minimization on a manifold}

The main results of this work 
	exploit 
local smooth features 
	of 
	nonsmooth optimization problems (c.f. Definition~\ref{defn:ident_man}). 
In the presence
	of 
	these features,
	the constructions 
		of 
		Definition~\ref{defn:var_constr} locally simplify.
Before moving
	to
	the general setting, 
we thus  
	interpret
	the various derivative constructions 
		in 
			the classical setting	of 
			minimizing a $C^2$-smooth function $f$ 
				on 
				a $C^2$-smooth manifold $\cM$.
To that end, we first recall the definition of a manifold.
\begin{definition}[Smooth manifold]
{\rm
A subset $\mathcal{M}\subset\R^n$ 
	is 
a $C^p$ {\em manifold of dimension $r$ around $\bar{x}\in\cM$}
	if there 
		is 
	an open neighborhood $U$ around $\bar{x}$ 
		and 
	a mapping $G$ from $\R^n$ to $\RR^{n-r}$ 
		such that following hold: 
		$G$ 
			is 
		$C^p$-smooth, 
		the derivative $\nabla G(\bar{x})$ has full rank,
			and equality holds:
		$$\mathcal{M}\cap U=\{x\in U: G(x)=0\}.$$ 
We call $G=0$ the {\em local defining equations} 
	for 
	$\cM$ around $\bar x$. 
The {\em tangent space} 
	to 
	$\mathcal{M}$ 
		at 
		$\bar{x}$ 
			is 
			$T_{\mathcal{M}}(\bar{x}):=\ker \nabla G(\bar{x})$
				and	
the {\em normal space} 
	to 
	$\mathcal{M}$ at $\bar{x}$ 
		is 
	 $N_{\mathcal{M}}(\bar{x}):=\range \nabla G(\bar{x})^*$.}
\end{definition}

Turning 
	to
	the classical setting, 
	consider the optimization problem
	\begin{equation}\label{eqn:smooth_min_over_smooth_man}
\min_{y\in\R^d}~ f(y)\quad\textrm{subject to } y\in \cM.
\end{equation}
Fix a point $\bar y\in\cM$
	and 
suppose that 
	both 
		the function $f$ is $C^2$-smooth around $\bar y$ and $\cM$ is a $C^2$-smooth manifold around $\bar y$.
Due to local smoothness, 
	the subdifferential
		admits
	the simple expression:
$$\partial f_{\cM}(\bar y)=\nabla f(\bar y)+N_{\mathcal{M}}(\bar{y}).$$
Recall that we use the shorthand $f_{\cM}:=f+\delta_{\cM}$.
From this expression, 
	we 
		see
	that a point $\bar y\in\cM$ 
		is 
	first-order critical for the problem \eqref{eqn:smooth_min_over_smooth_man} precisely when the inclusion holds:
\begin{equation}\label{eqn:first_order_nec}
0\in \nabla f(\bar y)+N_{\mathcal{M}}(\bar{y}).
\end{equation}
This inclusion can be equivalently stated in terms of the Lagrangian function. 
Namely, let $G=0$ 
	be 
the local defining equations 
	for 
	$\cM$ around $\bar y$
	and 
define the Lagrangian function 
$$\mathcal{L}(y,\lambda):=f(y)+\langle G(y),\lambda\rangle.$$
Then  \eqref{eqn:first_order_nec} 
	amounts to
existence 
	of 
	a (unique) multiplier vector $\bar \lambda\in\R^m$ satisfying $0=\nabla_y \cL(\bar y,\bar \lambda)$. 
Next, assuming $\bar y$ is critical,
	 second-order necessary conditions 
		read
\begin{equation}\label{eqn:nece_cond}
\langle \nabla^2_{yy} \cL(\bar y,\bar\lambda) u,u\rangle\geq 0 \quad\textrm{for all }u\in T_{\cM}(\bar y).
\end{equation}
Conversely, second-order sufficient conditions read
\begin{equation}\label{eqn:suff_cond}
 \langle \nabla^2_{yy} \cL(\bar y,\bar\lambda) u,u\rangle> 0 \quad\textrm{for all }0\neq u\in T_{\cM}(\bar y).
\end{equation}
It is well-known that the sufficient condition \eqref{eqn:suff_cond} implies more that just local minimality; namely, \eqref{eqn:suff_cond} holds if and only if there exists $\alpha>0$ such that
\begin{equation}\label{eqn:quadr_growth}
f(y)-f(\bar y)\geq \alpha \|y-\bar y\|^2, \qquad \textrm{for all }y\in \cM\textrm{ near }\bar y.
\end{equation}
Any point $\bar y$ satisfying \eqref{eqn:quadr_growth} is called a {\em strong local minimizer} of $f$ on $\cM$.

The Lagrangian conditions
	\eqref{eqn:nece_cond} 
		and 
	\eqref{eqn:suff_cond} 
	may be succinctly expressed  
		through 
		parabolic subderivatives 
			of 
			$f_{\cM}(y)$, 
	yielding a form
		independent
		of
		the choice 
			of 
		local defining equations $G=0$.
In particular,
	a quick computation 
		shows
	that for any $u\in T_{\cM}(\bar y)$, 
	the function $w\mapsto d^2  f(\bar{y})(u|w)$ 
		is 
	constant 
		on 
		its domain.\footnote{The domain of $d^2 f_{\cM}(\bar{y})(u|\cdot)$ consists of $w$ satisfying $(\langle \nabla^2 G_1(\bar y)u,u\rangle,\ldots, \langle \nabla^2 G_{n-r}(\bar y)u,u\rangle)=-\nabla G(\bar y)w$, where $G_i$ are the coordinate functions of $G$.} 
Dropping the dependence on $w$, 
	the equation then holds:
$$d^2 f_{\cM}(\bar{y})(u)=\langle \nabla^2_{yy} \cL(\bar y,\bar \lambda)u,u\rangle\qquad \textrm{for all }u\in T_{\cM}(\bar y).$$ 

The use of \eqref{eqn:suff_cond} 
 goes far beyond verifying local optimality; indeed, this condition plays
		a fundamental role
			in 
			certifying solution stability under small perturbations.
To illustrate, 
	consider 
	the value function 
		of 
		the parametric family
\begin{gather}\label{eqn:paramet_prob}
\varphi(x)=\inf_y~ \{f(x,y): y\in \cM\}, \tag{$\cP_{ x}$}
\end{gather}
where $f$ is $C^2$-smooth and $\cM\subset\R^d$ is a closed set. 
Let $\bar y$ 
	be 
	a minimizer 
		of 
		$\cP_{\bar x}$
			for 
			a fixed parameter $\bar x$, 
	and 
suppose that $\cM$
	 is a 
	 $C^2$-smooth manifold around $\bar y$. 
Let $G=0$ 
	be 
	the local defining equations 
		for 
		$\cM$ around $\bar y$  
			and 
define the parametric Lagrangian function 
$$\mathcal{L}(x,y,\lambda)=f(x,y)+\langle G(y),\lambda\rangle.$$
Since $\bar y$ solves $\cP_{\bar x}$, 
	there 
		is 
	a multiplier vector $\bar \lambda$ satisfying $0=\nabla_y \cL(\bar x,\bar y,\bar \lambda)$.

The following perturbation result 
	will form 
the core of our arguments. 
In short:
	both 
		the value function $\varphi(x)$
			and
		the minimizer of $\cP_x$ 
			vary 
		smoothly with $x$,
	provided two mild conditions hold (level-boundedness and quadratic growth). 
Moreover, the derivatives of both the value function and the solution maps can be computed explicitly. 
For details and a much more general perturbation result, see \cite[Theorem 3.1]{Shapiro1985}.

\begin{thm}[Perturbation analysis]\label{thm:stab}
	Suppose that the following two properties hold.
	\begin{enumerate}
		\item {\bf (Level-boundedness)} \label{thm:stab:itemlevel} There exists a number $\gamma>\varphi(\bar x)$ and a neighborhood $\cX$ of $\bar x$ such that 
		the set
		$$\bigcup_{x\in \cX}\{y\in\cM: f(x,y)\leq \gamma\}\qquad \textrm{is bounded}.$$
		\item {\bf (Quadratic growth)}\label{thm:stab:itemquad} The point $\bar y$ is a strong local minimizer and a unique global minimizer of $\mathcal{P}_{\bar x}$.
	\end{enumerate}
	Define the partial Hessian matrices 
	$$H_{xx}=\nabla^2_{xx}\cL(\bar x,\bar y,\bar\lambda),\qquad H_{xy}=\nabla^2_{xy}\cL(\bar x,\bar y,\bar \lambda),\qquad H_{yy}=\nabla^2_{yy}\cL(\bar x,\bar y,\bar \lambda),$$
	and the quantities
	\begin{align*}
	\eta(h)&=\min_{v\in T_{\cM}(\bar y)}~\langle H_{xx} h,h\rangle+2\langle H_{xy} v,h\rangle+\langle H_{yy} v,v\rangle,\\
	\Phi(h)&=\argmin_{v\in T_{\cM}(\bar y)}~\langle H_{xx} h,h\rangle+2\langle H_{xy} v,h\rangle+\langle H_{yy} v,v\rangle.
	\end{align*}
	Then for every $x$ near $\bar x$, the problem $\cP_{x}$ admits a unique solution $y(x)$, which varies $C^1$-smoothly and admits the first-order expansion $$\bar y(\bar x+h)=\bar y+\Phi(h)+o(\|h\|)\qquad \textrm{as }h\to 0.$$ 
	Moreover, the  function $\varphi$ is $C^2$-smooth around $\bar x$ and admits the second order expansion
	$$\varphi(\bar x+h)=\varphi(\bar x)+\langle \nabla_x f(\bar x,\bar y),h\rangle+\tfrac{1}{2}\eta(h)+o(\|h\|^2) \qquad \textrm{as }h\to 0.$$
\end{thm}

The two assumptions
	of 
	the theorem 
		play 
	different roles.  
The level-boundedness property 
	ensures 
that the solutions 
	of 
	the perturbed problems $\cP_x$ 
		lie 
	in a compact set around $\bar y$. 
The quadratic growth property in turn 
	ensures
smoothness 
	of both
	the solution map 
		and 
	the value function. 
In what follows, 
	we 
		will apply 
	this result several times. 
Both conditions will follow 
in all cases
	from 
	the next simple lemma.

\begin{lem}[Sufficient conditions for level boundedness]\label{lem:checkcond}
	Consider a closed function $\varphi\colon\R^d\times \R^n\to \R\cup\{\infty\}$ 
		and 
	fix a point $\bar x\in\R^d$. 
	Suppose 
		there exists $\alpha>0$ 
			such that 
			for all $x$ near $\bar x$, 
		the function $\varphi(x,\cdot)$ 
			is 
		$\alpha$-strongly convex 
			and 
		its minimizer $y(x)$ varies continuously. 
Then $y(x)$ 
	is 
a strong global minimizer 
	of 
	$\varphi(x,\cdot)$ 
	for all $x$ near $\bar x$. 
Moreover, 
	there 
		exists 
	a neighborhood $\cX$ 
		of 
		$\bar x$ 
			such that for any real $\gamma> \varphi(\bar x,y(\bar x))$, 
			the set
	$$\bigcup_{x\in \cX}\{y\in\R^n: \varphi(x,y)\leq \gamma\}\qquad \textrm{is bounded}.$$
\end{lem}
\begin{proof}
Strong convexity 
	ensures 
there 
	is 
a neighborhood $\cX$ 
	of 
	$\bar x$ such that 
		for any $x\in\cX$, 
the estimate holds: 
	\begin{equation}\label{eqn:str_conv_est}
	\varphi(x,y(x))+\frac{\alpha}{2}\|y-y(x)\|^2\leq \varphi(x,y) \qquad \forall y\in \R^n,
	\end{equation}
showing $y(x)$ 
	is 
	a strong global minimizer of $\varphi(x,\cdot)$. 
Shrinking $\cX$ if necessary
	we may assume 
that $y(\cdot)$ also varies continuously on $\cX$. 
Choose any $\delta>0$.
Then by shrinking $\cX$ again 
	and 
	by 
		leveraging 
	both 
		closedness of $\varphi$ 
			and
		continuity of $y$,
	we 
		may ensure
	that
	\begin{equation}\label{eqn:lsc}
	\|y(x)-y(\bar x)\|\leq \delta \qquad \textrm{and}\qquad \varphi(x,y(x))\geq \varphi(\bar x,y(\bar x))-\delta\qquad \textrm{for all }x\in \cX.
	\end{equation}
	The proof will now follow quickly from the bound \eqref{eqn:lsc}. Indeed, 	consider 
				any points 
					$x\in \cX$ 
						and 
					$y\in\R^d$ 
				satisfying $\varphi(x,y)\leq \gamma$. Then \eqref{eqn:str_conv_est} yields 
	$$\|y-y(x)\|\leq \sqrt{\frac{2(\gamma-\varphi(x,y(x)))}{\alpha}}.$$
	Applying \eqref{eqn:lsc} then gives the uniform bound
	$$\|y-y(\bar x)\|\leq \|y(x)-y(\bar x)\|+\sqrt{\frac{2(\gamma-\varphi(x,y(x)))}{\alpha}}\leq \delta+\sqrt{\frac{2(\gamma+\delta-\varphi(\bar x,y(\bar x)))}{\alpha}},$$
	completing the proof.
\end{proof}

\subsection{Weak convexity and the Moreau envelope}\label{sec:moreau}

In general, 
	the little-$o$ error term in the definition of $\partial f(\bar x)$ (Definition~\ref{defn:var_constr})
		may depend 
both 
	on the base point $\bar x$ 
		and 
	on the subgradient $v$. 
In this work, 
	we
		focus on
	a particular class of functions for which the error in approximation is uniform. 
Namely, 
	we 
		focus on
	the class of {\em $\rho$-weakly convex} functions $f\colon\R^d\to\R\cup\{\infty\}$, meaning those for which 
			the assignment $x\mapsto f(x)+\frac{\rho}{2}\|x\|^2$ 
				defines
			a convex function. Subgradients of a $\rho$-weakly convex function $f$ automatically yield a uniform lower bound:
				\begin{equation}\label{eqn:subgrad_ineq}
		f(y)\geq f(x)+\langle v,y-x\rangle-\frac{\rho}{2} \|y-x\|^2,\qquad\forall x,y\in\R^d,~ v\in \partial f(x).
		\end{equation}

A useful feature 
	of 
	weakly convex functions 
		is 
	that they 
		admit 
	a smooth approximation that 
		preserves 
	critical points. 
Setting the notation, 
	fix 
		a $\rho$-weakly convex function $f\colon\R^d\to\R\cup\{\infty\}$ 
		and 
		a parameter $\mu<\rho^{-1}$. 
Define 
	the {\em Moreau envelope} 
	and
	the {\em proximal point} map, respectively:
\begin{align*}
f_{\mu}(x)&=\inf_{y\in\R^d}~ \left\{f(y)+\frac{1}{2\mu}\|y-x\|^2\right\},\\
\prox_{\mu f}(x)&=\argmin_{y\in\R^d}~ \left\{f(y)+\frac{1}{2\mu}\|y-x\|^2\right\}.
\end{align*}
We 
	will use 
		a few basic properties 
			of 
			these two constructions, summarized below.

\begin{lem}[Moreau envelope and the proximal point map]\label{lem_prop_moreau}
Consider a $\rho$-weakly convex function $f\colon\R^d\to\R\cup\{\infty\}$ and  fix a parameter $\mu<\rho^{-1}$. Then the following are true.
\begin{enumerate}
\item\label{it1:moreau} The envelope $f_{\mu}$ 
	is 
	$C^1$-smooth with its gradient given by 
\begin{equation}\label{eqn:grad_moreau}
\nabla f_{\mu}(x)=\mu^{-1}(x-\prox_{\mu f}(x)).
\end{equation}
\item\label{it2:moreau} The envelope $f_{\mu}(\cdot)$ is $\mu^{-1}$-smooth and $\tfrac{\rho}{1-\mu\rho}$-weakly convex meaning:
\begin{equation}\label{eqn:smoothness_moreau}
-\frac{\rho}{2(1-\mu\rho)}\|x'-x\|^2\leq f_{\mu}(x')- f_{\mu}(x)-\langle \nabla f_{\mu}(x),x'-x\rangle\leq \frac{1}{2\mu}\|x'-x\|^2,
\end{equation}
for all $x,x'\in \R^d$.
\item \label{it3:moreau}
The proximal map $\prox_{\mu f}(\cdot)$ is $\frac{1}{1-\mu\rho}$-Lipschitz continuous and the gradient map $\nabla f_{\mu}$ is Lipschitz continuous with constant $\max\{\mu^{-1},\frac{\rho}{1-\mu\rho}\}$.
\item \label{it4:moreau} The critical points 
	of 
		$f$ 
			and 
		$f_{\mu}$ coincide. 
In particular, they are exactly the fixed points of the proximal map $\prox_{\mu f}$. 
\end{enumerate}
\end{lem}
\begin{proof}
Claim~\ref{it1:moreau} follows for example from \cite[Theorem 4.4]{prox_reg}. The left-hand-side of \eqref{eqn:smoothness_moreau} is proved in \cite[Theorem 5.2]{prox_reg}. To see the right-hand-side, observe
\begin{align*}
f_{\mu}(x')&\leq f(\prox_{\mu f}(x))+\frac{1}{2\mu}\|\prox_{\mu f}(x)-x'\|^2\\
&=f_{\mu}(x)+\frac{1}{2\mu}\left(\|\prox_{\mu f}(x)-x'\|^2-\|x-\prox_{\mu f}(x)\|^2\right)\\
&=f_{\mu}(x)+\langle \mu^{-1}(x-\prox_{\mu f}(x)),x'-x\rangle+\frac{1}{2\mu}\|x-x'\|^2.
\end{align*}
Thus Claim~\ref{it2:moreau} holds.
The result \cite[Theorem 4.4]{prox_reg} shows that $\prox_{\mu f}(\cdot)$ is Lipschitz continuous with constant $\frac{1}{1-\mu\rho}$. Lipschitz continuity of 
 $\nabla f_{\mu}(\cdot)$ with constant $\max\{\mu^{-1},\frac{\rho}{1-\mu\rho}\}$ follows from \eqref{eqn:smoothness_moreau} and Alexandrov's theorem \cite[Theorem 13.51]{RW98}. Thus claim~\ref{it3:moreau} holds. Claim \ref{it4:moreau} is  immediate from~\eqref{eqn:grad_moreau} and the observation that the function $y\mapsto f(y)+\frac{1}{2\mu}\|y-x\|^2$ is strongly convex for any $x$.
\end{proof}

\subsection{Active manifolds} 
The nonsmooth behavior 
	of
	sets 
		and 
	functions
	arising
in applications
	is
	typically far from pathological 
		and 
	instead manifests 
	in highly structured ways. 
Formalizing this perspective 
	we 
		will assume
	that nonsmoothness, in a certain localized sense, only occurs along an ``active manifold."
This notion, 
	introduced in \cite{lewis_active} under the name of partial smoothness
		and 
	rooted in the earlier works \cite{Wright,Al-Khayyal-Kyparisis91,Burke-More88,Burke90,Calamai-More87,Dunn87,Ferris91,Flam92},
	extends
the concept 
	of 
	\emph{active sets} 
		in 
		nonlinear programming 
	far beyond the classical setting.  
In this work,
	we 
		will take
	the related perspective developed in \cite{ident}, since it will be most expedient for our purpose.

Before giving the formal definition, 
	we 
		provide
	some intuition.  
Taking a geometric view, 
	we 
		will assume 
	that each critical point of a function $f$
		lies    
	on a smooth manifold $\cM$, and  
	that the objective 
		varies 
	smoothly along the manifold, 
		but
	sharply off of it.
For example consider Figure~\ref{subsec:partlysmooth_man}:
	there the function $f(x,y)=|x|-y^2$ 
		admits 
	the active manifold $\cM=\{0\}\times \R$ around its unique critical point (the origin).
From an algorithmic point of view,
	active manifolds 
		are
	the sets that typical algorithms 
		(e.g. proximal point, proximal gradient \cite{hare_alg}, 
			and 
		dual averaging \cite{lee2012manifold}) 
		identify 
	in finite time. 
Active manifolds also 
	play 
a central role 
	for 
	sensitively analysis, 
	providing a path 
		to 
		reduce such questions 
		to the smooth setting.
In particular,
	reasonable conditions 
		guarantee 
	that the active manifold 
		is 
	smoothly traced out 
		by 
		critical points 
			of 
			slight perturbations 
				of 
				the problem. 
We are now ready to state the formal definition.\footnote{What we call an {\em active manifold} here is called an {\em identifiable manifold} in \cite{ident}--the reference we most closely follow. The term active is more evocative in the context of the current work.}
\begin{definition}[Active manifold]\label{defn:ident_man}{\rm 
Consider a closed weakly convex function $f\colon\R^d\to\R\cup\{\infty\}$ 
	and 
fix a set $\mathcal{M} \subseteq \RR^d$ 
	containing 
	a critical point $\bar x$ 
		of 
		$f$. 
Then $\mathcal{M}$ 
	is called 
{\em an active} $C^p${\em-manifold around} $\bar x$  
	if 
		there exist 
	a neighborhood $U$ around $\bar x$ satisfying the following.
\begin{itemize}
\item {\bf (smoothness)}  
The set $\mathcal{M}\cap U$ 
	is 
a $C^p$-smooth manifold 
	and 
the restriction 
	of 
	$f$ 
	to $\mathcal{M}\cap U$ 
	is $C^p$-smooth.
\item {\bf (sharpness)} 
The lower bound holds:
$$\inf \{\|v\|: v\in \partial f(x),~x\in U\setminus \cM\}>0.$$
\end{itemize}}
\end{definition}

If $f$ 
	admits 
an active manifold 
	around 
	a critical point $\bar x$, 
	then it 
		must be 
	locally unique: 
	any two active manifolds at $\bar x$
		must coincide 
	on a neighborhood of $\bar x$ \cite[Proposition 2.4, Proposition 10.10]{darxiv}.\footnote{Note that due to the convention $\inf_{\emptyset}=+\infty$, the entire space $\mathcal{M}=\R^d$ is the active manifold
for a globally  $C^p$-smooth function $f$ around any of its critical points.} Moreover, the critical cone $C_f(\bar x,0)$ coincides with the tangent space $T_{\cM}(\bar x)$ \cite[Proposition 10.8]{darxiv}. With the definition of the active manifold in hand, 
	we 
		can now introduce 
	the  strict saddle property for nonsmooth functions.\footnote{Better terminology would be the terms {\em active strict saddle} and the {\em active strict saddle property}. To streamline the notation, we omit the word active, as it should be clearly understood from context.}

\begin{definition}[Strict saddles]\label{defn:strict_saddle_conv}{\rm 
Consider a weakly convex function $f\colon\R^d\to\R\cup\{\infty\}$. Then we say that a critical point $\bar x$ is a {\em strict saddle} of $f$ if there exists a $C^2$-active manifold $\cM$ 
	of $f$ at $\bar x$
		and 
the inequality $d^2 f_{\cM}(\bar x)(u)<0$ holds for some vector $u\in T_{\cM}(\bar x)$. If every critical point of $f$ is either a local minimizer or a strict saddle, then we say that $f$ satisfies the {\em  strict saddle property}.} 
\end{definition}

Looking at Figure~\ref{subsec:partlysmooth_man}, 
	we 
		see 
that the function $f(x,y)=|x|-y^2$ indeed has the strict saddle property: 
the restriction 
	of 
	$f$ to the axis $\cM=\{0\}\times \R$, 
	namely $f_{\cM}(0,t)=-t^2$, 
		certainly has 
	directions of negative curvature.  
Figure~\ref{subsec:partlysmooth_grad} 
	depicts 
the subgradient flow generated by this function. 
Notice that the set of initial conditions attracted to the origin has measure zero. 
This observation 
	suggests 
that typical algorithms 
	are 
also unlikely to stall at the strict saddle point, an observation made precise by the forthcoming results.

The curvature condition in the definition of the strict saddle 
	pertains only to negative curvature of the restriction of $f$ to $\cM$.
One 
	may instead ask 
whether existence 
	of 
	directions 
		of 
		negative curvature for $f$ alone suffice. 
The answer turns out to be yes.

\begin{thm}[{\cite[Corollary 4.15]{drusvyatskiy2016generic}}]
Consider a closed weakly convex function $f\colon\R^d\to\R\cup\{\infty\}$ that admits a $C^3$-active manifold $\cM$ around a critical point $\bar x$. Then it  holds:
$$d^2 f(\bar x)( u\mid w)\geq d^2 f_{\mathcal{M}}(\bar x)( u) \qquad \textrm{ for all } u\in C_{f}(\bar{x},0), w\in \R^d.$$
\end{thm}

A natural question 
	is 
whether we expect the strict saddle property to hold typically. 
One supporting piece 
	of 
	evidence 
		is 
	that the property holds under generic linear perturbations of semialgebraic problems.\footnote{A function is semi-algebraic if its graph can be written as a finite union of sets each cut out by finitely many polynomial inequalities} 
This is almost immediate from guarantees in \cite[Theorem 4.16]{drusvyatskiy2016generic}, though this conclusion is not explicitly stated in the theorem statement. We state this guarantee below and provide a quick proof in Section~\ref{sec:proof_generic} for completeness. 

\begin{thm}[Strict saddle property is generic]\label{thm:strict_saddle_hoora_gen}
Consider a closed, weakly convex, semi-algebraic function $f\colon\R^d\to\R\cup\{\infty\}$. Then for a full Lebesgue measure set of perturbations $v\in \R^d$, the perturbed function $x\mapsto f(x)-\langle v,x\rangle$ has the strict saddle property.
\end{thm}

\subsection{The Center Stable Manifold Theorem} 
In this work, 
	we 
		will show 
	that a variety of simple algorithms 
		escape
	strict saddle points. To prove results of this type, 
	we will
		interpret
	 algorithms 
		as
		fixed point iterations 
			of 
			a nonlinear map $T\colon \R^d\to\R^d$, 
			having certain favorable properties. As in the smooth setting of \cite{gradient_descent_jason}, the core  
	of 
	our arguments will be based on the center stable manifold theorem.

\begin{thm}[The Center Stable Manifold Theorem{~\cite[Theorem III.7]{shub2013global}}]\label{thm:center_stab}
Let the origin be a fixed point of the $C^1$ local diffeomorphism $T \colon U \rightarrow \RR^d$ where $U$ is a neighborhood of the origin in $\RR^d$. Let $E^s \oplus E^c \oplus E^u$ be the invariant splitting of $\RR^d$ into the generalized eigenspaces of the Jacobian $\nabla T(0)$ corresponding to eigenvalues of absolute value less than one, equal to one, and greater than one. Then there exists a local $T$ invariant $C^1$ embedded disk $\stableset$, tangent to $E^s \oplus E^c$ at $0$ and a neighborhood $B$ around zero such that $T(\stableset) \cap B \subseteq\stableset$.  In addition, if $T^k(x) \in B$ for all $k \geq 0$, then $x \in \stableset$.
\end{thm}

An immediate consequence 
	of 
	this theorem is the following: 
	if $\nabla T(0)$ 
		is 
	invertible 
			and 
		has 
	at least one eigenvalue of magnitude greater than one, 
	then 
		there exists 
	a neighborhood $B$ 
		of 
		the origin such that the set
$$\{x\in B:~ T^k(x)\in B~ \textrm{ for all }k\geq 0\},$$
has measure zero. 
This fact motivates the following key definition.

\begin{definition}[Unstable fixed points]
{\rm 
A fixed point $\bar x$ 
	of 
	a map $T\colon\R^d\to\R^d$ 
		is called 
	{\em unstable}
	if 
		$T$ 
			is 
		$C^1$-smooth 
			around 
			$\bar x$ 
				and 
		the Jacobian $\nabla T(\bar x)$ 
			has 
		an eigenvalue 
			of
			magnitude strictly greater than one.}
\end{definition}

To globalize the guarantees
	of
	the center stable manifold theorem, 
	we 
		will need 
	to impose global regularity properties on $T$. 
In this work, 
	we 
		will require
	the map $T$ 
		to be
	a {\em lipeomorphism},
	namely,  
	we 
		require 
	that $T$ 
		is 
	globally Lipschitz 
		and 
	its inverse $T^{-1}$ 
		is 
	a well-defined globally Lipschitz map. 
The following corollary 
	is 
now immediate. 
Its proof 
	closely follows 
the presentation 
	in
	\cite[Theorem 2]{Lee:2019:FMA:3349830.3349888}.

\begin{cor}\label{cor:main_stab_cor}
Let $T \colon \RR^d \rightarrow \RR^d$ 
	be 
a lipeomorphism 
	and 
let $\mathcal{U}_T$ 
	consist 
of 
all unstable fixed points $x$ 
	of 
	$T$ 
	at 
	which the Jacobian $\nabla T(x)$ 
		is 
	invertible. 
Then the set 
	of 
	initial conditions 
		attracted 
	by such fixed points
$$
W : = \left\{x \in \RR^d \colon \lim_{k \rightarrow \infty} T^k(x) \in \cU_T\right\}
$$ 
has zero Lebesgue measure. 
\end{cor}
\begin{proof}
For every $\bar x \in \cU_T$ 
	there exists 
a neighborhood $U$ 
	of 
	$\bar x$ such that $T \colon U \rightarrow \RR^d$ 
		is 
	a local diffeomorphism. 
Thus, the center stable manifold theorem 
	shows there exists 
an open neighborhood $B_{\bar x}$ 
	of 
	$\bar x$ so that $S_{\bar x} := \bigcap_{k = 0}^\infty T^{-k}(B_{\bar x})$ 
		is 
	contained in a measure zero set. 
In particular, 
	$S_{\bar x}$ itself 
		is 
	measure zero. 

Now 
	observe 
that $\cU_T \subseteq \bigcup_{\bar x \in \cU_T} B_{\bar x}$ 
	is 
an open cover 
	of 
	$\cU_T$. 
Since $\RR^d$ 
	is 
second countable, 
	this cover 
		has 
	a countable subcover $\cU_{T} \subseteq \bigcup_{i =1}^\infty B_{\bar x_i}$. 
Observe the inclusion $W \subseteq \bigcup_{i = 1}^\infty \bigcup_{j = 0}^\infty T^{-j}(S_{\bar x_i})$. Since $T$ 
	is 
a lipeomorphism, 
	the right hand side 
		is 
	a countable union 
		of 
		measure zero sets, 
	and therefore $W$ 
		has 
	measure zero. 
\end{proof}

To verify that a map $T$ is a lipeomorphism, 
	we will appeal 
to the following standard sufficient condition. 
We 
	provide 
a quick proof 
	for 
	completeness.

\begin{lem}\label{lem:ver_lipeo}
	Let $H \colon \RR^d \rightarrow \RR^d$ 
		be 
	a Lipschitz continuous map 
		with 
		constant $\lambda < 1$. 
Then $(I+H)$ 
	is 
invertible 
	and 
$(I+H)^{-1} \colon \RR^d \rightarrow \RR^d$ 
	is 
Lipschitz continuous 
	with 
	constant $(1-\lambda)^{-1}$. 
\end{lem}
\begin{proof}
To show that $(I+ H)$ 
	is 
invertible, 
	we 
		must show 
	that for every $u \in \RR^d$, 
	the equation $u = H(x) + x$ 
		has 
	a unique solution $x(u) \in \RR^d$. 
Equivalently, 
	we 
		must show 
	that for every $u\in\R^d$,	the mapping
	$$
	\zeta_u(x) := u - H(x)
	$$
	has a unique fixed point. This is immediate from Banach's fixed point theorem since $\zeta_u(\cdot)$ is strictly contractive.  
	
	To show that $(I+H)^{-1}$ is Lipschitz, 
		choose 
	arbitrary $u, v \in \RR^d$ 
		and
		define  
	$x := (I+H)^{-1}(u)$ and $y := (I+H)^{-1}(v)$.
We then compute	
\begin{align*}
	\|u - v\| &= \|(I+H)(x) - (I+H)(y)\| \geq  \|x-y\| - \|H(x) - H(y)\| \geq (1-\lambda)\|x - y\|,
	\end{align*}
where 
	we have used
		the reverse triangle inequality 
			and 
		Lipschitz continuity 
			of 
			$H$. 
Rearranging 
	completes
the proof. 
\end{proof}

While the iteration mappings $S$ 
		of 
		Section~\ref{sec:intro_escape}
			can be Lipschitz, 
	they 
		are 
	usually not invertible.
Thus to ensure 
	Lipschitz invertibility, 
	we 
		will consider
damped fixed point iterations, 
	as summarized in the following elementary lemma. 
We 
	provide 
a quick proof 
	for 
	completeness.
\begin{lem}[Damped fixed point iterations]\label{lem:relax_and_split}
Consider a map $S\colon\R^d\to\R^d$ 
	and 
fix a damping parameter $\alpha \in (0,1)$. 
Define the map 
$$T(x)=(1-\alpha)x+\alpha\cdot S(x).$$ 
Then the following 
	are 
true.
\begin{enumerate}
	\item\label{iit1} 
	The fixed points 
		of 
		$T$ 
			and 
		$S$ 
		coincide.
	\item\label{iit3} 
	If $S$ is differentiable at $\bar x$ and the Jacobian $\nabla S(\bar x)$ has a real eigenvalue strictly greater than one, 
	then $\bar x$ 
		is 
	an unstable fixed point 
		of 
		$T$.
	\item\label{iit4} 
	If the map $S$ 
		is 
	continuous 
		and 
	the iterates generated 
		by 
		the process $x_{k+1}=T(x_k)$ 
			converge 
	to some point $\bar x$, 
	then $\bar x$ 
		must be 
	a fixed point 
		of 
		$S$.
	\item\label{iit2} 
	If the map $I-S$ 
		is 
	$L$-Lipschitz, 
	then $T$ 
		is 
	a lipeomorpshim 
		for 
		any $\alpha\in (0,L^{-1})$.
\end{enumerate}
\end{lem}
\begin{proof}
Claims \ref{iit1} and \ref{iit3} follow directly from algebraic manipulations. Claim \ref{iit2} follows immediately from Lemma~\ref{lem:ver_lipeo} by writing $T=I+H$ with $H=\alpha(S-I)$. To see claim \ref{iit4}, suppose that $T$ is continuous and that $x_k$ converge to some point $\bar x$. Then we deduce $$T(\bar x)=T\left(\lim_{k\to\infty} x_k\right)=\lim_{k\to\infty} T(x_k)=\lim_{k\to\infty} x_{k+1}=\bar x.$$
Therefore $\bar x$ is a fixed point of $T$. Using claim \ref{iit1},
 we deduce that $\bar x$ is a fixed point of $S$.
\end{proof}

\section{The Proximal Point Method}\label{sec:prox_point}
We now 
	turn to
the saddle escape properties
	of
	the proximal-point method. 
Fixing the problem at hand, 
	we 
		consider 
$$\min_{x\in\R^d}~ f(x),$$
where $f\colon\R^d\to\R\cup\{\infty\}$
	 is 
a $\rho$-weakly convex function
	that
		is
	bounded from below. 
For a fixed $\mu<\rho^{-1}$, 
	the classical proximal-point method 
		is 
	precisely the fixed point iteration 
$$x_{t+1}=\prox_{\mu f}(x_t).$$ 
Key
to our analysis 
	is 
the equivalence 
	between 
		this algorithm
		and 
		 gradient descent on
			the Moreau envelope. 
This equivalence 
	follows
		from~\eqref{eqn:grad_moreau},
			which 
				quickly yields 
			the description
$$x_{k+1}=x_k- \mu\cdot\nabla f_{\mu}(x_k).$$
The saddle escape properties 
	of
	the proximal point method
	thus 
		flow
	from the strict saddle properties of the Moreau envelope. 
Indeed,
	 the following theorem 
shows that when $f$ 
	admits 
a $C^2$ active manifold 
	around 
	a critical point $\bar x$, 
	the envelope $f_{\mu}$ 
		is 
	automatically $C^2$-smooth near $\bar x$. 
Moreover, if $\bar x$ 
	is 
a strict saddle  
	of 
	$f$, 
	then it 
		is 
	also a strict saddle  
		of 
		$f_{\mu}$. 
Consequently,
	any strict saddle point 
		of $f$
		is 
	an unstable fixed point 
		of 
		the proximal map $\prox_{\mu f}(\cdot)$.

\begin{thm}[Saddle points of the Moreau envelope]\label{thm:saddle_moreau}
Let $f\colon\R^d\to\R\cup\{\infty\}$ 
	be 
		a closed 
			and 
		$\rho$-weakly convex function 
			and 
let $\bar x$ 
	be 
	any critical point 
		of 
		$f$. 
Suppose that $f$ 
	admits 
a $C^2$ active manifold $\cM$ 
	at 
	$\bar x$. 
Then
	for any $\mu<\rho^{-1}$, 
		the Moreau envelope $f_{\mu}$ 
			is 
				$C^2$-smooth 
					around 
					$\bar x$
				and
		its Hessian satisfies
\begin{equation}\label{eqn:mor_smooth}
\min_{h\in  \mathbb{S}^{d-1}\cap T_{\cM}(\bar x)} \langle\nabla^2 f_{\mu}(\bar x) h,h \rangle\leq \min_{h\in \mathbb{S}^{d-1}\cap T_{\cM}(\bar x)} d^2 f_{\cM}(\bar x)(h).
\end{equation}
Consequently, 
	if $\bar x$ 
		is 
	a strict saddle point
		of 
		$f$, 
	then $\bar x$ 
		is 
	both 
		a strict saddle point 
		of 
		$f_{\mu}$ 
			and  
		an unstable fixed point
			of 
			the proximal map $\prox_{\mu f}(\cdot)$.
Moreover,
		$\nabla \prox_{\mu f}(\bar x)$
			has 
		a real eigenvalue that is strictly greater than one. 
\end{thm}
\begin{proof}
It 
	is 
well known (for example from \cite{hare_alg}) 
	that 
	for all $x$ near $\bar x$, 
	the inclusion $\prox_{\mu f}(x)\in \cM$ holds. 
From this inclusion, 
	we 
		will be able to view 
	the proximal subproblem 
		through
	the lens of the perturbation result in Theorem~\ref{thm:stab}.
For the sake of completeness, however,
	let us first 
		quickly verify 
	the claim. 
Consider a sequence $x_i\to\bar x$ 
	and 
observe the inclusion $\nabla f_{\mu}(x_i)\in \partial f(\prox_{\mu f}(x_i))$. 
Since the gradient $\nabla f_{\mu}$ 
	is 
continuous, 
	we 
		deduce the limits
	$\prox_{\mu f}(x_i)\to \bar x$ 
			and 
	$\nabla f_{\mu}(x_i)\to 0$. 
Therefore by definition 
	of 
	the active manifold,
		we
			have $\prox_{\mu f}(x_i)\in \cM$ 
		for all sufficiency large indices $i$, proving the claim. 

Turning to the perturbation result, 
let $F\colon\R^d\to\R$ 
	be 
any $C^2$-smooth function 
	agreeing with $f$ 
		on 
		a neighborhood 
			of 
			$\bar x$ 
			in 
			$\cM$.\footnote{For example, let $F$ be a $C^2$ function defined on a neighborhood $U$ of $\bar x$ that agrees with $f$ on $U\cap \mathcal{M}$. Using a partition of unity (e.g. \cite[Lemma 2.26]{lee2013smooth}), one can extend $F$ from a slightly smaller neighborhood to be $C^2$ on all of $\R^d$.}
Applying the claim, 
	we
		find 
	that the equality
$$f_{\mu}(x)=\min_{y\in \cM}~ \left\{F(y)+\frac{1}{2\mu}\|y-x\|^2\right\},$$
holds for all $x$ near $\bar x$. 
Our goal
	is to apply the perturbation result (Theorem~\ref{thm:stab})
		with 
			$f(x,y):=F(y)+\frac{1}{2\mu}\|y-x\|^2$ 
				and 
			$\varphi(x):=f_{\mu}(x)$. 
To that end, 
we now 
	verify 
the assumptions 
	of 
	Theorem~\ref{thm:stab}. 
First we verify the quadratic growth condition:
since 
	we have chosen 
$\mu<\rho^{-1}$, 
	it 
		follows 
	that 
		for every $x\in\R^d$ 
		the function $y\mapsto f(x)+\frac{1}{2\mu}\|y-x\|^2$ 
			is 
		strongly convex 
			with 
			constant $\mu^{-1}-\rho$. 
Next we verify the level boundedness condition: 
since
	the minimizer $y(x):=\prox_{\mu f}(x)$ 
		of 
		this function 
			varies 
	continuously 
		and 
	satisfies $y(\bar x)=\bar x$, 
	the conditions of Lemma~\ref{lem:checkcond} 
		are 
	satisfied.
Therefore, 
	the assumptions of Theorem~\ref{thm:stab} are valid.

We now apply Theorem~\ref{thm:stab}. 
To that end, 
	let $G=0$ 
		be 
	the defining equation 
		of 
		$\cM$ 
		around $\bar x$ 
			and 
	define the parametric Lagrangian function 
$$\mathcal{L}(x,y,\lambda):=F(y)+\frac{1}{2\mu}\|y-x\|^2+\langle G(y),\lambda\rangle.$$
Since $\bar x$ 
	is 
critical 
	for 
	$f$, 
	the equality $\bar x=\prox_{\mu f}(\bar x)$ 
		holds. 
Consequently, 
	$y(\bar x) = \bar x$
		minimizes 
	the function $y\mapsto F(y)+\frac{1}{2\mu}\|y-\bar x\|^2$ 
		on 
		$\cM$. 
Therefore, 
	first-order optimality conditions
		guarantee
	there 
		exists 
	a multiplier vector $\bar \lambda$ 
		satisfying 
$$0=\nabla_y \cL(\bar x,\bar x,\bar \lambda)=\nabla F(\bar x)+\sum_{i\geq 1}\bar \lambda_i G_i(\bar x),$$
where $G_i(\cdot)$ 
	are 
the coordinate functions 
	of 
	$G(\cdot)$.
Appealing to Theorem~\ref{thm:stab}, 
	we 
		learn 
	both 
		that  
			$f_{\mu}$ 
				is 
			$C^2$-smooth 
				around 
				$\bar x$ 
			and 
		that its Hessian 
			satisfies 
\begin{equation}\label{hess_est_prox_point}
\langle\nabla^2 f_{\mu}(\bar x) h,h \rangle=\min_{u\in T_{\cM}(\bar x)}~\langle H_{xx} h,h\rangle+2\langle H_{xy} u,h\rangle+\langle H_{yy} u,u\rangle,
\end{equation}
where the Hessian matrices 
	are 
given by 
$$H_{xx}=\mu^{-1} I,\qquad H_{xy}=-\mu^{-1} I, \qquad H_{yy}=\nabla^2 F(\bar x)+\sum_{i=1}^m \bar{\lambda}_i\nabla^2 G_i(\bar x)+\mu^{-1} I.$$
Thus 
	rearranging~\eqref{hess_est_prox_point} 
		and 
	setting $D:=\nabla^2 F(\bar x)+\sum_{i=1}^m \bar{\lambda}_i\nabla^2 G_i(\bar x),$ 	we 
		have
$$\langle\nabla^2 f_{\mu}(\bar x) h,h \rangle=\min_{u\in T_{\cM}(\bar x)}\left\{\langle Du,u\rangle+\mu^{-1} \|h-u\|^2\right\}.$$
Therefore, 
	we 
		arrive 
	at the estimate
\begin{align*}
\min_{h\in {\mathbb{S}}^{d-1}\cap T_{\cM}(\bar x)}\langle\nabla^2 f_{\mu}(\bar x) h,h \rangle&=\min_{u\in T_{\cM}(\bar x)}\min_{h\in {\mathbb{S}}^{d-1}\cap T_{\cM}(\bar x)}\left\{\langle Du,u\rangle+\mu^{-1} \|h-u\|^2\right\}\\
&\leq \min_{h\in \mathbb{S}^{d-1}\cap T_{\cM}(y_0)}\langle Dh,h\rangle =\min_{h\in {\mathbb{S}^{d-1}}\cap T_{\cM}(\bar x)} d^2 f_{\cM}(\bar x)(h),
\end{align*}
thereby verifying \eqref{eqn:mor_smooth}. 
If $\bar x$ 
	is 
a strict saddle point 
	of 
	$f$, 
	then \eqref{eqn:mor_smooth} 
		implies 
	that $\nabla^2 f_{\mu}(\bar x)$ 
		has 
	a strictly negative eigenvalue. 
From the expression $\prox_{\mu f}=I-\mu\nabla f_{\mu}$, 
	we therefore
		deduce 
	that the Jacobian 
		of 
		$\prox_{\mu f}$ at $\bar x$
			has 
		at least one real eigenvalue that is strictly greater than one. 
Consequently,
$\bar x$ 
	is
an unstable fixed point 
	of 
	$\prox_{\mu f}$.
\end{proof}

Even if the proximal mapping
	has 
an unstable fixed-point, 
	it 
		often fails 
	to meet the conditions of the center stable manifold theorem (Theorem~\ref{thm:center_stab}).
	Indeed, the proximal mapping is generally not injective, even near critical points. 
To remedy this issue, 
	we 
		instead analyze 
	a slightly damped version 
		of 
		the proximal point method
$$x_{k+1}=(1-\alpha) x_k +\alpha\cdot\prox_{\mu f}(x_k),$$
where $\alpha\in (0,1)$ 
	is 
a fixed constant.
Reinterpreting this algorithm in terms of the Moreau envelope, 
	we
		arrive 
	at the recurrence
\begin{equation}\label{eqn:grad_moreau_damp}
x_{k+1}=x_k- (\alpha\mu)\cdot\nabla f_{\mu}(x_k).
\end{equation}
Thus the role
	of 
	damping is clear: 
	it 
		still induces
	gradient descent on the Moreau envelope,
	but 
		with 
		a stepsize slightly below the ``theoretically optimal'' step $\mu$. 
This 
	is 
entirely inline 
	with 
	the saddle point escape guarantees 
		for 
		gradient descent 
			in 
			smooth minimization \cite{gradient_descent_jason}.

\begin{thm}[Proximal point method: global escape]\label{thm:prox-point}
Let $f\colon\R^d\to\R\cup\{\infty\}$ 
	be 
a closed 
	and 
$\rho$-weakly convex function 
	satisfying the strict saddle property. 
Choose 
	a constant $\mu<\rho^{-1}$ 
		and 
	a damping parameter $\alpha \in (0,\min\{1,(\mu\rho)^{-1}-1\})$. 
With these choices,  
	consider the algorithm
\begin{equation}\label{eqn:prox_point_damp}
x_{k+1}=(1-\alpha) x_k +\alpha\cdot\prox_{\mu f}(x_k).
\end{equation}
Then for almost all initializers $x_0$, 
the following holds:
	if the limit 
		of 
	$\{x_k\}_{k\geq 0}$ exists,
	it
		must be 
	a local minimizer 
		of 
		$f$.
\end{thm}
\begin{proof}
Define the map $S:=\prox_{\mu f}(x_k)$.
 Lemma~\ref{lem_prop_moreau} guarantees that the map $I-S=\mu\nabla f_{\mu}$ is Lipschitz continuous with constant 
$\max\{1,\frac{\mu\rho}{1-\mu\rho}\}$. Taking into account the range of $\alpha$ and 
%
applying
	Lemma~\ref{lem:relax_and_split} 
		and
	Theorem~\ref{thm:saddle_moreau} 
		we 
			may deduce
		the following three properties: 
		(1) $T$ is a lipeomorphism, 
		(2)   the limit 
				of
				the sequence $x_k$, 
					if it exists, 
					must be 
			a critical point of $f$, 			
				and 
		(3) if a critical point
			of 
			$f$ 
				is
			not a local minimum,
			then it
				is
			an unstable fixed point of $T$. 
An application 
	of 
	Corollary~\ref{cor:main_stab_cor} 
		completes 
	the proof.
 \end{proof}

\section{The Proximal Gradient Method}\label{sec:prox_grad}
We now turn to the saddle escape properties 
	of 
	the proximal gradient method. 
Fixing the problem at hand, 
	we
		consider
\begin{equation}\label{eqn:add_comp}
\min_{x\in\R^d}~ f(x)=g(x)+r(x),
\end{equation}
where $g\colon\R^d\to\R$ 
	is 
a $C^2$-smooth function 
	with 
	$\beta$-Lipschitz gradient 
		and 
$r\colon\R^d\to\R\cup\{+\infty\}$ 
		is 
	a closed 
		and 
	$\rho$-weakly convex function. 
We assume throughout that $f$ 
	is bounded 
from below. 
For this problem, 
the proximal gradient method
	takes 
the form
$$x_{k+1}=\prox_{\mu r}\left(x_k-\mu\nabla g(x_k)\right).$$
Unlike the proximal point algorithm, 
	the proximal gradient algorithm
		may not 
	correspond to gradient descent on a smooth envelope of the problem.
Still, as the following theorem shows, 
	the iteration mapping 
	is 
$C^1$ smooth near $\bar x$ 
	whenever
$f$ 
	admits 
	a $C^2$ active manifold 
		around 
		a critical point $\bar x$.
Moreover, 
	if $\bar x$ 
		is 
	a strict saddle point 
		of 
		$f$, 
	then $\bar x$ 
		is 
	an unstable fixed point 
		of  
	the iteration mapping

\begin{thm}[Unstable fixed points of the prox-gradient map]\label{thm:saddle_prox_grad}
	Consider the optimization problem~\eqref{eqn:add_comp} 
		and 
	let $\bar x$ 
		be 
	any critical point 
		of 
		$f$. 
Suppose that $f$
	admits 
a $C^2$ active manifold 
	$\cM$ 
	at $\bar x$. 
Then
	for any $\mu\in(0,\rho^{-1})$, 
	the proximal-gradient map 
	$$S(x):=\prox_{\mu r}\left(x-\mu\nabla g(x)\right)$$
		is 
	$C^1$-smooth 
		on 
		a neighborhood 
			of 
			$\bar x$.  
Moreover, 
	if $\bar x$ 
		is 
	a strict saddle point 
		of 
		$f$, 
	then 
		$\nabla S(\bar x)$
			has 
		a real eigenvalue that is strictly greater than one.
\end{thm}
\begin{proof}
It is well-known (for example from \cite{hare_alg}) that for all $x$ near $\bar x$, the point $S(x)$ lies in $\cM$. 
From this inclusion, 
	we 
		will be able to view 
	the proximal subproblem 
		through
	the lens of the perturbation result in Theorem~\ref{thm:stab}.
For the sake of completeness, however, we provide a quick proof.
Indeed, 
	consider a sequence $x_i\to\bar x$ 
		and 
	set $y_i=S(x_i)$. 
Then by definition 
	of 
	the proximal gradient map, 
	we 
		have
	$0\in \nabla g(x_i)+\mu^{-1}(y_i-x_i)+\partial r(y_i),$
		and 
	therefore
	\begin{align*}
	\dist(0,\partial f (y_i))=\dist(-\nabla g(y_i),\partial r (y_i))&\leq \dist(-\nabla g(x_i),\partial r (y_i))+\beta\|y_i-x_i\|\\
	&\leq (\mu^{-1}+\beta)\|y_i-x_i\|.
	\end{align*}
	Since $S(\cdot)$ 
		is 
			continuous 
				and 
		$S(\bar x)=\bar x$, 
		we 
			deduce $y_i\to\bar x$ 
				and 
			therefore $\dist(0,\partial f (y_i))\to 0$. 
Therefore the points $y_i$ 
	lie 
in $\cM$ 
	for 
	all sufficiently large indices $i$, proving the claim. 
	
Turning to the perturbation result, 	
	let $R\colon\R^d\to\R$ 
		be 
	any $C^2$-smooth function 
		agreeing 
		with $r$ 
			on 
			a neighborhood 
				of 
				$\bar x$ 
				in $\cM$. 
Applying the claim, 
	we find that
		for 
	$x$ near $\bar x$, 
	the point $S(x)$
		uniquely 
	minimizes problem
	\begin{gather}
	\min_{y\in \cM}~ \left\{g(x)+\langle\nabla g(x),y-x\rangle+R(y)+\frac{1}{2\mu}\|y-x\|^2\right\}.\tag{$\mathcal{P}_{x}$}
	\end{gather}
Our goal
	is 
to apply the perturbation result (Theorem~\ref{thm:stab}) 
	with 
	$f(x,y):=g(x)+\langle\nabla g(x),y-x\rangle+R(y)+\frac{1}{2\mu}\|y-x\|^2$. 
 To that end, 
 	we 
		now verify 
	the assumptions 
		of
		 Theorem~\ref{thm:stab}. 
First we verify the quadratic growth condition: since we have chosen $\mu<\rho^{-1}$, it follows that for every $x\in\R^d$ the function $y\mapsto f(x,y)$ is strongly convex with the constant $\mu^{-1}-\rho$. 
Next 
	we 
		verify 
	the level-boundedness condition: 
	since the minimizer $S(x)$ clearly 
		varies continuously 
			and 
		satisfies $S(\bar x)=\bar x$, 
	the conditions of Lemma~\ref{lem:checkcond} 
		are 
	satisfied. 
Therefore, the assumptions 
	of 
	Theorem~\ref{thm:stab} 
		are 
	valid. 
	
We now apply Theorem~\ref{thm:stab}. 
To that end, 
let $G=0$ 
	be
the defining equation 
	of 
	$\cM$ 
	around 
	$\bar x$ 
		and 
define the parametric Lagrangian function 
	$$\mathcal{L}(x,y,\lambda)=g(x)+\langle\nabla g(x),y-x\rangle+R(y)+\frac{1}{2\mu}\|y-x\|^2+\sum_{i\geq 1}\lambda_i G_i(y),$$
	where $G_i(\cdot)$ are the coordinate functions of $G$. 
Clearly $y(\bar x) = \bar x$
	minimizes 
$f(\bar x,\cdot)$ 
	on 
	$\cM$.	
Therefore, 
	first-order optimality conditions 
		guarantee 
	there 
		exists 
	a multiplier vector $\bar \lambda$ satisfying 
$$0=\nabla_y \cL(\bar x,\bar x,\bar \lambda)=\nabla g(\bar x)+\nabla R(\bar x)+\sum_{i\geq 1}\bar \lambda_i G_i(\bar x).$$
Appealing to Theorem~\ref{thm:stab}, 
	we 
		learn 
	that the solution map $S(\cdot)$ 
		is 
	$C^1$-smooth 
		around 
		$\bar x$ 
			with 
	\begin{equation}\label{eqn:hss_prox_grad}
	\nabla S(\bar x)h=\argmin_{v\in T_{\cM}(\bar x)}~2\langle H_{xy} v,h\rangle+\langle H_{yy} v,v\rangle,
	\end{equation}
	where the Hessian matrices 
		are given 
	by 
	$$H_{xy}=\nabla^2 g(\bar x) -\mu^{-1} I, \qquad H_{yy}=\nabla^2 R(\bar x)+\mu^{-1} I+\sum_{i=1}^p \bar{\lambda}_i\nabla^2 G_i(\bar x).$$
We now
	simplify 
the expression~\eqref{eqn:hss_prox_grad}.
To that end,
	let $W$ 
		be 
	the orthogonal projection 
		onto 
		$T_{\cM}(\bar x)$ 
			and 
	define the linear maps 
		$\overline{H_{yy}}\colon T_{\cM}(\bar x)\to T_{\cM}(\bar x)$ 
			and 
		$\overline{H_{xy}}\colon T_{\cM}(\bar x)\to T_{\cM}(\bar x)$ 
			by 
				setting 
					$\overline{H_{yy}}=W H_{yy}W$ 
						and 
					$\overline{H_{xy}}=W H_{xy}W$, 
					respectively. 
Since $\bar x$ 
	is 
a strong local minimizer 
	of $\mathcal{P}_{\bar x}$, 
	the map $\overline{H_{yy}}$ 
		is 
	positive definite, 
		and hence 
	invertible. 
Solving \eqref{eqn:hss_prox_grad}
	then yields 
the expression   
	\begin{equation*}
	\nabla S(\bar x)h=-\overline H_{yy}^{-1}\overline{H_{xy}}^\top h \qquad \textrm{for all }h\in T_{\cM}(\bar x).
	\end{equation*}
Note that $\overline{H_{xy}}^\top$ is a symmetric matrix, so we drop the ``$\top$" throughout.
	
	Let us 
		now verify 
	that 
		if $\bar x$ 
			is 
		a strict saddle 
			of 
			$f$, 
		then $\nabla S(\bar x)$ 
			has a real eigenvalue that is greater than one.
	To this end, 
		observe that $\gamma\in\R$ 
			is 
		a real eigenvalue of $\nabla S(\bar x)$ 
			with 
			an associated eigenvector $v\in T_{\cM}(\bar x)$ 
				if and only if 
	$$\nabla S(\bar x)v=\gamma v\qquad \Longleftrightarrow\qquad -\overline H_{yy}^{-1}\overline{H_{xy}} v=\gamma v\qquad \Longleftrightarrow\qquad (\gamma\overline H_{yy} + \overline{H_{xy}}) v=0.$$
In particular 
	if the matrix $\gamma\overline H_{yy} + \overline{H_{xy}}$ 
		is 
	singular, 
	then $\gamma$ 
		is 
	an eigenvalue 
		of 
		$\nabla S(\bar x)$.
To prove  
	such a $\gamma$ exists, 
we 
	will examine
the following ray of symmetric matrices 
	$$\{\gamma\overline H_{yy} + \overline{H_{xy}}: \gamma\geq 1\}.$$
Beginning with the end point, 
	the strict saddle property
		shows 
	that 
	$$\overline H_{yy} + \overline{H_{xy}}=W\left(\nabla^2 g(\bar x)+\nabla^2 R(\bar y)+\sum_{i}\bar{\lambda}_i \nabla^2 G_i(\bar x)\right)W.$$
has a strictly negative eigenvalue. 
On the other hand, 
	the direction 
		of 
		the ray $\overline H_{yy}$ 
			is 
			a positive definite matrix. 
Therefore 
	by continuity 
		of 
		the minimal eigenvalue function, 
	there 
		exists 
	some $\gamma>1$ such that the matrix $\gamma\overline H_{yy} + \overline{H_{xy}}$ is singular, as claimed.
\end{proof}

Similar to the proximal point method, 
	the proximal gradient mapping 
		fails to meet
	the conditions 
		of 
		the center stable manifold theorem (Theorem~\ref{thm:center_stab}), 
	since it generally lacks invertibility. 
Therefore as before
	we will analyze 
a slightly damped version 
	of 
	the process, and prove the following theorem.

\begin{thm}[Proximal gradient method: global escape]\label{thm:prox-gradient}
Consider the optimization problem~\eqref{eqn:add_comp} 
	and
suppose that $f$ 
	has 
the strict saddle property.
Choose any constant $\mu\in(0,\rho^{-1})$
	and 
a damping parameter $\alpha\in (0,1)$ satisfying
$$\alpha\cdot \left(\mu\beta+(1+\mu\beta)\max\left\{1,\tfrac{\mu\rho}{1-\mu\rho}\right\}\right)<1.$$
Consider the algorithm
\begin{equation}\label{eqn:prox_grad}
x_{k+1}=(1-\alpha) x_k +\alpha\cdot \prox_{\mu r}\left(x_k-\mu\nabla g(x_k)\right).
\end{equation}
Then for almost all initializers $x_0$, 
the following holds:
	if the limit 
		of 
	$\{x_k\}_{k\geq 0}$ exists,
	it
		must be 
	a local minimizer 
		of 
		$f$.
\end{thm}
\begin{proof}
Define the maps $S=\prox_{\mu r}\left(I-\mu\nabla g\right)$.
We successively rewrite
\begin{align*}
I-S&=(I-\mu\nabla g)-\prox_{\mu r}\left(I-\mu\nabla g\right)+\mu\nabla g\\
&=\mu\cdot\nabla r_{\mu}\circ(I-\mu\nabla g)+\mu\nabla g.
\end{align*}
Lemma~\ref{lem_prop_moreau} implies that the map $I-S$ is Lipschitz continuous with constant $\mu\beta+(1+\mu\beta)\max\left\{1,\tfrac{\mu\rho}{1-\mu\rho}\right\}$. 
%
Taking into account the range of $\alpha$
	and applying
	Lemma~\ref{lem:relax_and_split} 
		and
	Theorem~\ref{thm:saddle_prox_grad} 
		we 
			may deduce
		the following three properties: 
		(1) 
		$T$ 
			is 
		a lipeomorphism, 
		(2)  
		the limit of the sequence $x_k$, if it exists, 
			must be 
		a critical point 
			for 
			$f$, 
		and 
		(3) if a critical point
			of 
			$f$ 
				is
			not a local minimum,
			then it
				is
			an unstable fixed point of $T$.  An application of Corollary~\ref{cor:main_stab_cor} then completes the proof. \end{proof}

\section{The Proximal Linear Method}\label{sec:prox_linear}
We now 
	turn to 
the saddle escape properties 
	of 
	the proximal linear method, 
	a generalization 
		of 
			the 
				proximal point 
				and 
				proximal gradient 
			methods. 
Setting the stage, 
consider 
	the composite optimization problem 
\begin{equation}\label{eqn:conv_comp}
\min_{x}~ f(x)=h(F(x))+r(x),
\end{equation} 
 where $F\colon\R^d\to\R^m$ 
 	is 
	a $C^2$-smooth map, 
	$h\colon\R^d\to\R$ 
		is 
	convex, 
		and 
	$r\colon\R^d\to\R\cup\{\infty\}$ 
		is 
	$\rho$-weakly convex.  
As 
	is 
standard 
	in 
	the literature, 
	we 
		will 
	assume that 
		there exists 
	a constant $\beta>0$ satisfying 
 \begin{equation}\label{eqn:quadrapprox}
 \left|h(F(y))-h(F(x)+\nabla F(x)(y-x))\right|\leq \frac{\beta}{2}\|y-x\|^2, \qquad \forall x, y \in \RR^d.
 \end{equation}
These assumptions 
	then 
		easily imply
	that $f$ 
		is 
	weakly convex 
		with 
		constant $\beta+\rho$. 

With the stage set, 
	we 
		now slightly refine
	the notion 
		of 
		a strict saddle, 
	adapting it to the compositional nature 
		of 
		the problem. 
This refinement 
	intuitively 
		asks 
	that the active manifold 
		for 
		$f$ 
			at 
			a critical point $\bar x$  
				is induced 
		by active manifolds 
			of 
				$h$ 
					and 
				$r$. 
Similar conditions 
	have appeared 
elsewhere, 
	for example, 
		in 
		\cite{ident,prox,lewis_active}. 
To describe the condition formally, 
	we 
		will also revise
	the definition 
		of 
		an active manifold, 
	allowing us to discuss active manifolds 
		of 
			$h(\cdot)$ 
				and 
			$r(\cdot)$ 
		at 
		noncritical points.
The revision 
	is 
intuitive,
requiring
	just a linear tilt 
		of 
		the functions: 
  \begin{itemize}
  	\item 
  	Consider 
		a set $\cR\subset\R^d$, 
		a point $x\in\cR$, 
			and 
		a subgradient $v\in \partial r(x)$. 
	We 
		will say 
	that $\cR$ 
		is 
		a {\em 
			$C^2$ active manifold 
				of 
				$r$ 
					at 
					$x$ 
						for 
						$v$} 
		if $\cR$
			is 
		a $C^2$ active manifold 
			of 
			the tilted function $r-\langle v,\cdot\rangle$ 
				at 
				$x$ 
					in 
					the sense 
						of 
						Definition~\ref{defn:ident_man}. 
\end{itemize}
We may 
	likewise define
the active manifold 
	of 
	$h$ 
		at 
		$z$ 
			for 
		$w\in \partial h(z)$, 
	based on a tilting of $h$ by $w$. 
Coupling these definitions,
	we arrive 
at the active manifold concept 
	for 
	the composite problem \eqref{eqn:conv_comp}.

 \begin{definition}[Composite active manifold]\label{defn:comp_act_man}{\rm
 Consider the compositional problem \eqref{eqn:conv_comp} 
 	and 
let $\bar x$ 
	be 
a critical point 
	of 
	$f$. 
Fix arbitrary vectors 
	$\bar w\in \partial h(F(\bar x))$ 
		and 
	$\bar v\in \partial r(\bar x)$
satisfying
\begin{equation}\label{eqn:subdif_incl}
 0\in \nabla F(\bar x)^*\bar w+\bar v.
 \end{equation}
Suppose the following hold.
\begin{enumerate}
 \item\label{it1:make} 
 	There 
		exist 
	$C^2$-smooth manifolds 
		$\mathcal{R}\subset\R^d$ 
			and  
		$\mathcal{H}\subset\R^m$ 
	containing 
		$\bar x$ 
			and 
		$F(\bar x)$, 
		respectively, 
			and 
	satisfying the transversality condition: 
\begin{equation}\label{eqn:transversal}
\nabla F(\bar x)\left[T_{\mathcal{R}}(\bar x)\right]+T_{\mathcal{H}}(F(\bar x))=\R^m.
\end{equation}
\item\label{it2:make} 
	$\mathcal{R}$ 
		is 
	an active manifold 
		of 
		$r$ 
			at 
			$\bar x$ 
				for 
				$\bar v$ 
					and 
	$\mathcal{H}$ 
		is 
	an active manifold 
		of 
		$h$ 
			at 
			$F(\bar x)$ 
				for 
				$\bar w$. 
\end{enumerate}
 Then 
 	we will call 
$\cM:=\mathcal{R}\cap F^{-1}(\mathcal{H})$ a {\em composite $C^2$ active manifold}
	for 
	the problem \eqref{eqn:conv_comp} 
		at 
		$\bar x$. 
If in addition the inequality $d^2 f_{\cM}(\bar x)(u)<0$ 
	holds 
for some vector $u\in T_{\cM}(\bar x)$, 
	then 
		we will call 
	$\bar x$ a {\em composite strict saddle point}.}
 \end{definition}
 
This definition
	has
several important subtleties.
First, 
the set $\cM:=\mathcal{R}\cap F^{-1}(\mathcal{H})$ 
	is 
indeed a $C^2$-smooth manifold 
	around 
	$\bar x$, 
	due to the classical transversality condition \eqref{eqn:transversal},
	a central fact 
		in 
		differential geometry \cite[Theorem 6.30]{lee2013smooth}. 
Next, the vectors
	$\bar v$
		and 
	$\bar w$ 
	do exist.
This follows since $\bar x$
	is
first-order critical for $f$:
 $$0\in \nabla F(\bar x)^*\partial h(F(\bar x))+\partial r(\bar x).$$
Beyond existence,
 	the vectors
	$\bar v\in \partial r(\bar x)$ 
		and 
	$\bar w\in \partial h(F(\bar x))$ 
		are in fact
	the unique elements satisfying \eqref{eqn:subdif_incl}, 
	a second consequence of transversality. 
To see this, 
	we 
		state 
	\eqref{eqn:transversal}
		in
		dual terms as
\begin{equation}\label{eqn:dual_transv}
(\nabla F(\bar x)^*)^{-1}N_{\cR}(\bar x)\cap N_{\cH}(F(\bar x))=\{0\}.
\end{equation}
Considering another pair $v\in \partial r(\bar x)$ and $ w\in \partial h(F(\bar x))$ satisfying \eqref{eqn:subdif_incl}, 
	we deduce 
$$0=\nabla F^*(\bar x)(\bar w-w)+(\bar v-v).$$
To conclude
	$v=\bar v$ 
		and 
	$w=\bar w$, 
	we 
		use 
	\eqref{eqn:dual_transv}
			and
		simply recall 
	that 
		$\spann~\partial h(F(\bar x))=N_{\cH}(F(\bar x))$ 
			and 
		$\spann~\partial r(\bar x)=N_{\cR}(\bar x)$, as shown in \cite[Proposition 10.12]{darxiv}.
Finally, collecting these facts together,
	it 
		follows 
	from the chain rule \cite[Proposition 5.1]{darxiv} 
		that 
		$\cM$ 
			is 
		an active manifold 
			of 
			$f$ 
				at 
				$\bar x$ 
					in 
					the sense 
						of 
						Definition~\ref{defn:strict_saddle_conv}.

A natural question 
	is 
whether we expect the composite strict saddle property to hold typically. 
One supporting piece, 
	of 
	evidence, analogous to Theorem~\ref{thm:strict_saddle_hoora_gen}, 
		is 
	that the property holds under generic linear perturbations of semialgebraic composite problems.
This result 
	quickly follows 
from \cite[Theorem 5.2]{drusvyatskiy2016generic}. 
We 
	provide 
a proof sketch 
	in Section~\ref{sec:proof_generic}.

 \begin{thm}[Strict saddle property is generic]\label{thm:strict_saddle_hoora_gen2}
Consider the composite problem \eqref{eqn:conv_comp}, 
	where 
		$h$, 
		$r$, 
			and 
		$F$ 
			are 
	in addition semi-algebraic. 
Then 
	for 
	a full Lebesgue measure set 
		of 
		perturbations $(y,v)\in \R^m\times \R^d$, 
	the problem $$\min_x~ h(F(x)+y) + r(x) -\langle v,x\rangle$$ 
		has 
	the composite strict saddle property.
\end{thm}

Turning to our central task, 
	we 
		aim 
	to analyze the saddle escape properties
		of 
		the proximal linear method:
$$x_{k+1}=\argmin_{y}~h(F(x_k)+\nabla F(x_k)(y-x_k))+r(y)+\frac{1}{2\mu}\|y-x_k\|^2.$$
To analyze this method, 
we 
	prove
the following theorem,  
	showing
that any strict saddle point 
	of 
	the composite problem \eqref{eqn:conv_comp} 
		is 
	an unstable fixed point 
		of 
		proximal linear update.

\begin{thm}[Unstable fixed points of the proximal linear map]\label{thm:saddle_prox_lin}
	Consider the composite problem~\eqref{eqn:conv_comp} 
		and 
	let $\bar x$ 
		be 
	any critical point 
		of 
		$f$. 
	Suppose the problem 
		admits 
	a composite $C^2$ active manifold $\cM$ 
		at 
		$\bar x$. 
 Then for any $\mu\in(0,\rho^{-1})$, 
 	the proximal linear map 
 \begin{equation}\label{eqn:prox_lin_map}
	S(x):=\argmin_{y}~h(F(x)+\nabla F(x)(y-x))+r(y)+\frac{1}{2\mu}\|y-x\|^2.
	\end{equation}
		is 
	$C^1$-smooth 
		on 
		a neighborhood 
			of $\bar x$. 
			Moreover, if $\bar x$ is a composite strict saddle point, then the Jacobian $\nabla S(\bar x)$ has a real eigenvalue strictly greater than one.
\end{thm}
In most ways, 
the proof 
	mirrors 
that of Theorem~\ref{thm:prox-point}. 
There 
	is, however, 
an important complication: 
	we
		must move beyond
	the perturbation result of Theorem~\ref{thm:stab}
			and instead
		analyze 
	a parametric family 
		of 
		optimization problems 
		where both 
			the objective 
				and 
			{\em the constraints} 
				depend 
				on 
				a perturbation parameter. 
Therefore, 
	we will rely 
on the following generalization 
	of
	Theorem~\ref{thm:stab}.
For 
	details 
		and 
	a much more general perturbation result, 
	see \cite[Theorem 4.2]{Shapiro1985}.
\begin{thm}[Perturbation analysis]\label{thm:stab2}
	Consider the family of optimization problems 
	\begin{gather}
	\min_y~ f(x,y) \qquad \textrm{subject to}\qquad G(x,y)=0 \tag{$\mathcal{Q}_x$}
	\end{gather}
	Fix 
		a point $\bar x$ 
			and 
		a minimizer $\bar y$ 
			of  
			$\mathcal{Q}_{\bar x}$, 
		and 
	suppose the following hold.
	\begin{enumerate}
	\item {\bf (Non-degeneracy)} 
		The function $f(\cdot,\cdot)$ 
			and 
		the map $G(\cdot,\cdot)$ 
				are 
		$C^2$-smooth near $(\bar x,\bar y)$, 
			and 
		the Jacobian $\nabla_y G(\bar x,\bar y)$ 
			is 
		surjective.
		\item {\bf (Level-boundedness)} 
	There 
		exists 
	a neighborhood $\cX$ 
		of 
		$\bar x$ 
			and
	a number $\gamma$ 
		greater 
			than 
		the minimal value 
			of 
			$\mathcal{Q}_{\bar x}$  
	such that the set
		$$\bigcup_{x\in \cX}\{y\in Y(x): f(x,y)\leq \gamma\}\qquad \textrm{is bounded},$$
	where $Y(x):=\{y: G(x,y)=0\}$ 
		denotes 
	the set 
		of 
		feasible points 
			for 
			$\mathcal{Q}_x$.
		\item {\bf (Quadratic growth)} 
		The point $\bar y$ 
			is 
				a strong local minimizer 
					and 
				a unique global minimizer 
					of 
					$\mathcal{Q}_{\bar x}$.
	\end{enumerate}
	Define the parametric Lagrangian function 
	$$\mathcal{L}(x,y,\lambda)=f(x,y)+\langle G(x,y),\lambda\rangle.$$
	Fix
		 the multiplier vector $\bar \lambda$ 
		 	satisfying $0=\nabla_y \mathcal{L}(\bar x,\bar y,\bar \lambda)$ 
				and 
	define 
		the Hessian matrices 
	$$H_{xx}=\nabla^2_{xx}\cL(\bar x,\bar y,\bar\lambda),\qquad H_{xy}=\nabla^2_{xy}\cL(\bar x,\bar y,\bar \lambda),\qquad H_{yy}=\nabla^2_{yy}\cL(\bar x,\bar y,\bar \lambda).$$
	Then 
		for every $x$ near $\bar x$, 
		the problem $\mathcal{Q}_{x}$ 
			admits 
		a unique solution $y(x)$,
		which varies $C^1$-smoothly.
		Moreover, its directional derivative 
				in direction $h$ 
					given by 
\begin{equation}\label{eqn:shap_manvary}
 \begin{aligned}	
	\nabla y(\bar x)h\quad =\quad&\argmin_v ~ 2\langle H_{xy} v,h\rangle+\langle H_{yy} v,v\rangle\\
	&~~~~{\rm s.t.}~~~~\nabla_x G(\bar x,\bar y)h+ \nabla_y G(\bar x,\bar y)v=0.
\end{aligned}
\end{equation}
\end{thm}

With these tools in hand, we now prove Theorem~\ref{thm:saddle_prox_lin}.

\begin{proof}[Proof of Theorem~\ref{thm:saddle_prox_lin}]
Let 
	$\bar v$, 
	$\bar w$, 
	$\cH$, 
	$\cR$, 
		and 
	$\cM$ 
		be 
the 
	vectors 
	and 
	manifolds 
		specified 
in Definition~\ref{defn:comp_act_man}. 
It 
	is known 
from \cite[Theorem 4.11]{prox} that 
	for all $x$ near $\bar x$, 
	the inclusions hold:
$$S(x)\in \cM\qquad \textrm{and}\qquad F(x)+\nabla F(x)(S(x)-x)\in \mathcal{H}.$$
From this inclusion, 
	we 
		will be able to view 
	the proximal subproblem 
		through
	the lens of the perturbation result in Theorem~\ref{thm:stab2}.
For the sake of completeness, however, we provide a quick proof.
Indeed, 
	consider a sequence $x_i\to\bar x$ 
		and 
	define $z_i=F(x_i)+\nabla F(x_i)(S(x_i)-x_i)$. 
Then 
	appealing to 
the optimality conditions 
	of 
	the proximal linear subproblem, 
	we 
		deduce 
	that there 
		exist 
		vectors $v_i\in \partial r(x_i)$ 
			and 
		$w_i\in \partial h(z_i)$ 
	satisfying $\frac{1}{\mu}(x_i-S(x_i))=\nabla F(x_i)^*w_i+v_i$. 
Since 
	$S(\cdot)$ 
		is 
	continuous 
			and 
	$h$ 
		is 
	Lipschitz, 
the vectors 
	$w_i$ 
		and 
	$v_i$ 
		are
bounded.
Passing 
	to 
	a subsequence, 
	we 
		may assume 
	that 
		$w_i$ 
			and 
		$v_i$ 
			converge to 
	some 
		$w\in  \partial h(F(\bar x))$ 
			and 
		$v\in \partial r(\bar x)$, respectively, 
			and moreover, 
	that 
$$0\in \nabla F(\bar x)^*w+v.$$
We therefore 
	deduce 
		$w=\bar w$ 
			and 
		$v=\bar v$. 
Taking into account that 
	$\cR$ 
		is 
	a $C^2$-active manifold 
		at 
		$\bar x$ 
			for 
			$\bar v$ 
				and 
	$\cH$ 
		is 
	a $C^2$-active manifold 
		at 
			$F(\bar x)$ 
				for 
				$\bar w$, 
	we 
		deduce 
			$S(x_i)\in \cR$ 
				and  
			$z_i\in \mathcal{H}$ 
		for all large indices $i$, proving the claim.

Turning to the perturbation result, 
	let $\hat h\colon\R^m\to\R$ 
		be 
	any $C^2$-smooth function agreeing 
		with 
		$h$ 
			on 
			a neighborhood 
				of 
				$F(\bar x)$ 
					in 
					$\mathcal{H}$, 
					and 
	let $\hat r\colon\R^d\to\R$ 
		be 
	any $C^2$-smooth function agreeing 
		with 
		$r$ 
			on 
			a neighborhood 
				of 
				$\bar x$ 
					in 
					$\mathcal{R}$. 
Applying the claim, we find that for $x$ near $\bar x$, we may write
\begin{equation}\label{eqn:parametric_prox_lin}
\begin{aligned}
S(x)\quad=\quad&\argmin_{y} ~\hat h\left(F(x)+\nabla F(x)(y-x)\right)+\hat r(y)+\frac{1}{2\mu}\|y-x\|^2\\
&~~~~\textrm{s.t.}~~~F(x)+\nabla F(x)(y-x)\in \mathcal{H}\quad \textrm{and}\quad y\in \mathcal{R}
 \end{aligned}.
 \end{equation}
 Our goal
	is 
to apply the perturbation result (Theorem~\ref{thm:stab2})
	to 
	the parametric family \eqref{eqn:parametric_prox_lin}. 
To this end, 
	let $\omega=0$ 
		be 
	the local defining equations 
		of 
		$\cH$ 
			around 
		$F(\bar x)$ 
				and 
	let $\eta=0$ 
		be 
	the local defining equation 
		of 
		$\cR$ 
			around 
		$\bar x$. 
We 
	can now place 
\eqref{eqn:parametric_prox_lin} 
	in 
	the setting 
		of 
		Theorem~\ref{thm:stab2}
		by  setting
 $$f(x,y)=\hat h\left(F(x)+\nabla F(x)(y-x)\right)+\hat r(y)+\frac{1}{2\mu}\|y-x\|^2$$
 and 
$$G(x,y):=(G^{\mathcal{H}}(x,y),G^{\mathcal{R}}(x,y)):=(\omega(F(x)+\nabla F(x)(y-x)),\eta(y)).$$
For these functions, 
 	we 
		now verify 
	the assumptions 
		of
		 Theorem~\ref{thm:stab}. 
First, the nondegeneracy property 
	follows 
from the transversality condition~\eqref{eqn:transversal}. 
Second, we verify the quadratic growth condition: since we have chosen $\mu<\rho^{-1}$, it follows that for every $x\in\R^d$ the function $y\mapsto f(x,y)$ is strongly convex with the constant $\mu^{-1}-\rho$. 
Finally,
	we 
		verify 
	the level-boundedness condition: 
	since the minimizer $S(x)$ clearly 
		varies continuously 
			and 
		satisfies $S(\bar x)=\bar x$, 
	the conditions of Lemma~\ref{lem:checkcond} 
		are 
	satisfied. 
Therefore, the assumptions 
	of 
	Theorem~\ref{thm:stab} 
		are 
	valid. 
In particular, 
 we 
		learn 
	that the solution map $S(\cdot)$ 
		is 
	$C^1$-smooth 
		around 
		$\bar x$.

Computing the Jacobian 
	of 
	the solution mapping
		will occupy
the remainder of the proof. 
To that end, 
	define 
the parametric Lagrangian 
$$\cL(x,y,\lambda)=f(x,y)+\langle G(x,y),\lambda\rangle.$$
Localizing,
	the identification properties 
		then entail
	that $y=\bar x$ 
		is 
	a minimizer 
		of 
		the problem \eqref{eqn:parametric_prox_lin} 
			corresponding to 
			$x=\bar x$. 
We conclude 
	there 
		exists 
	a Lagrange multiplier vector 
$\bar \lambda=(\bar \lambda^{\cH},\bar \lambda^{\cR})$ 
	satisfying 
$0=\nabla_y \cL(\bar x,\bar x,\bar \lambda)$, 
a fact 
	we will return to 
after a few calculations.

We now compute the first order variations of $f$ and $G$. 
To simplify notation, 
	we 
		adopt 
	two conventions.
First
	we 
		align 
	the notation
		of 
		gradients 
			and 
		Jacobians, 
	viewing every gradient as a row vector. 
Second, 
	we 
		let 
	the symbol $\nabla^2 F[x;v]$ 
		denote 
	the $m\times d$ matrix 
		whose 
		$i$th row equals $v^\top\nabla^2 F_i(x)$.
Then defining the map
$$\zeta(x,y)=F(x)+\nabla F(x)(y-x),$$
a quick computation shows
 $$\nabla_y \zeta(x,y)= \nabla F(x)\qquad \textrm{ and }\qquad  \nabla_x \zeta(x,y)=\nabla^2 F[x,y-x].$$
 Therefore 
 	using 
the chain rule, 
	we 
		compute 
	the first-order variations
 \begin{align*}
 \nabla_x G^{\mathcal{H}}(x,y)&=\nabla \omega \left(\zeta(x,y)\right)\cdot\nabla^2 F[x,y-x]\\
 \nabla_y G^{\mathcal{H}}(x,y)&=\nabla \omega (\zeta(x,y))\cdot \nabla F(x)\\
 \nabla_x G^{\cR}(x,y)&=0\\
 \nabla_y G^{\cR}(x,y)&=\nabla \eta(y)\\
\nabla_x f(x,y) &=\nabla \hat h \left(\zeta(x,y)\right)\cdot \nabla^2 F[x,y-x]+\mu^{-1}(x-y)^{\top}\\
\nabla_y f(x,y) &=\nabla \hat h \left(\zeta(x,y)\right)\cdot \nabla F(x)+\nabla \hat r(y)+\mu^{-1}(y-x)^{\top}.
 \end{align*}
From these variations 
	we 
		deduce 
	$ \nabla_x G(\bar x,\bar x)=0$ 
		and 
 	therefore the constraint 
		in 
		\eqref{eqn:shap_manvary} simply 
			amounts 
		to the inclusion
 \begin{equation}
 \begin{aligned}
 v\in \ker \nabla_y G(\bar x,\bar x)&=\Big(\ker \nabla \eta(\bar x)\Big)\cap \Big(\ker (\nabla \omega (F(\bar x))\cdot \nabla F(\bar x))\Big)\\
 &=T_{\cR}(\bar x)\cap \nabla F(\bar x)^{-1}T_{\cH}(F(\bar x))=T_{\cM}(\bar x).
 \end{aligned}
 \end{equation}
In particular, formula~\eqref{eqn:shap_manvary} reduces to 
\begin{equation}\label{eqn:hss_prox_linear}
	\nabla S(\bar x)h=\argmin_{v\in T_{\cM}(\bar x)}~2\langle H_{xy} v,h\rangle+\langle H_{yy} v,v\rangle,
\end{equation}
To find an explicit solution, 
	we
		mirror 
	the analysis 
		of 
		the proximal gradient method. 
We 
	let 
$W$ 
	be 
the orthogonal projection 
	onto 
	$T_{\cM}(\bar x)$ 
			and 
	define 
the linear maps 
	$\overline{H_{yy}}\colon T_{\cM}(\bar x)\to T_{\cM}(\bar x)$ 
		and 
	$\overline{H_{xy}}\colon T_{\cM}(\bar x)\to T_{\cM}(\bar x)$ by setting $\overline{H_{yy}}=W H_{yy}W$ and $\overline{H_{xy}}=W H_{xy}W$, respectively. 
	Since $\bar x$ 
		is 
	a strong local minimizer 
		of 
		$\eqref{eqn:shap_manvary}$, 
		the map $\overline{H_{yy}}$ 
			is positive 
				definite
					and 
				invertible. 
Solving \eqref{eqn:shap_manvary} then yields the expression   
\begin{equation*}
\nabla S(\bar x)h=-\overline H_{yy}^{-1}\overline{H_{xy}}^\top h \qquad \textrm{for all }h\in T_{\cM}(\bar x).
\end{equation*}

Let us 
		now verify 
	that 
		if $\bar x$ 
			is 
		a composite strict saddle 
			of 
			$f$, 
		then $\nabla S(\bar x)$ 
			has a real eigenvalue that is greater than one.
To this end, observe that $\gamma\in\R$ is an eigenvalue of $\nabla S(\bar x)$ with an associated eigenvector $v\in T_{\cM}(\bar x)$ if and only if 
$$\nabla S(\bar x)v=\gamma v\qquad \Longleftrightarrow\qquad -\overline H_{yy}^{-1}\overline{H_{xy}}^\top v=\gamma v\qquad \Longleftrightarrow\qquad (\gamma\overline H_{yy} + \overline{H_{xy}}^\top) v=0.$$
In particular 
	if the matrix $\gamma\overline H_{yy} + \overline{H_{xy}}^\top$ 
		is 
	singular, 
	then $\gamma$ 
		is 
	an eigenvalue 
		of 
		$\nabla S(\bar x)$.
To prove  
	such a $\gamma\geq 1$ exists, 
we 
	will show 
	that $\overline H_{xy}$
		is 
		self-adjoint, 
and then 
	we 	
	will examine
the following ray of symmetric matrices 
	$$\{\gamma\overline H_{yy} + \overline{H_{xy}}^\top: \gamma\geq 1\}.$$
Beginning with the end point, 
	we 
		will show
	that the matrix $\overline H_{yy} +\overline{H_{xy}}^\top$ has a strictly negative eigenvalue.
On the other hand, 
	we already 
		know 
	the direction 
		of 
		the ray $\overline H_{yy}$ 
			is 
			a positive definite matrix. 
Therefore 
	by continuity 
		of 
		the minimal eigenvalue function, 
	there 
		will exist
	some $\gamma>1$ such that the matrix $\gamma\overline H_{yy} + \overline{H_{xy}}$ is singular, as claimed.

To this end, we now compute the second-order variations.
\begin{align*}
\nabla_{xy} G_i^{\mathcal{H}}(x,y)v&=\nabla^2 F[x;v]^\top\nabla \omega_i(\zeta(x,y))^\top+\nabla^2 F[x;y-x]^\top\nabla^2 \omega_i(\zeta(x,y)) \nabla F(x)v\\
\nabla_{yy} G_i^{\mathcal{H}}(x,y)v&= \nabla F(x)^{\top}\nabla^2\omega_i(\zeta(x,y))\nabla F(x)v\\
\nabla_{xy} f(x,y)v&=\nabla^2 F[x;v]^\top\nabla \hat h(\zeta(x,y))^\top+\nabla^2 F[x;y-x]\nabla^2\hat h(\zeta(x,y))\nabla F(x)v-\mu^{-1}v\\
\nabla_{yy} f(x,y)v&=\nabla F(x)^{\top}\nabla^2 \hat h(\zeta(x,y))\nabla F(x)v+\nabla^2\hat r(y)v+\mu^{-1}v.
\end{align*}
A quick computation then 
	shows
	that 
		$\nabla_{xy} f(\bar x,\bar x)$ 
			and 
		$\nabla_{xy} G_i^{\mathcal{H}}(\bar x,\bar x)$ 
			are 
		self-adjoint operators.
%
%
%
%
%
Consequently, 
	we 
		obtain
		$H_{xy} =H_{xy}^\top$ 
			and 
		the expression
\begin{align*}
(H_{yy}+H_{xy}^{\top})v=\nabla &F(\bar x)^{\top}\nabla^2 \hat h(F(\bar x))\nabla F(\bar x)v+\nabla^2\hat r(\bar x)v+\nabla^2 F[\bar x;v]^\top\nabla \hat h(F(\bar x))^\top\\
&+\sum_{i\geq 1}\bar{\lambda}^{\cH}_i\Big(\nabla F(\bar x)^{\top}\nabla^2\omega_i(F(\bar x))\nabla F(\bar x)v+\nabla^2 F[\bar x;v]^\top\nabla \omega_i(F(\bar x))^\top\Big)\\
&+\sum_{i\geq 1}\bar{\lambda}_i^{\cR}  \nabla^2 \eta_i(y)v.
\end{align*}
To prove that $H_{yy}+H_{xy}^{\top}$ has a strictly negative eigenvalue, 
	we 
		will show
	that it 
		coincides 
	with the Hessian 
		of 
		the Lagrangian 
			of 
			the problem:
 $$\min_x ~\hat h(F(x))+\hat r(x)\qquad \textrm{ subject to } \qquad \omega(F(x))=0, \eta(x)=0.$$
Indeed, define the Lagrangian function 
$$\mathcal{L}_0(x,\lambda)=\hat h(F(x))+\hat r(x)+\sum_{i\geq 1}\lambda_i^{\cH}\omega(F(x))+\sum_{i\geq 1}\lambda_i^{\cR}\eta(x).$$ 
A quick computation 
	shows 
\begin{align*}
\nabla^2 (\hat h\circ F)(x)v&=\nabla F(x)^\top\nabla^2\hat h(F(x))\nabla F(x)v+\nabla^2 F[x,v]^\top\nabla \hat h (F(x))^\top\\
\nabla^2 (\omega_i\circ F)(x)v&=\nabla F(x)^\top\nabla^2\omega_i(F(x))\nabla F(x)v+\nabla^2 F[x,v]^\top\nabla \omega_i (F(x))^\top
\end{align*}
and therefore the equality
$$\nabla^2 \mathcal{L}_0(\bar x,\bar \lambda)=H_{yy}+H_{xy}^{\top}.$$
The composite strict saddle property guarantees that the matrix $\nabla^2 \mathcal{L}_0(\bar x,\bar \lambda)$ has a strictly negative eigenvalue, completing the proof.
\end{proof}

In line with the previous sections, 
	one 
		could 
	ask 
		whether 
		a 
			damped 
				and
			randomly initialized 
		proximal linear method almost surely 
			escapes 
		all composite strict saddle points. 
An immediate obstacle 
	is
	that the global Lipschitz constant 
		of 
		the proximal linear map $S(\cdot)$ 
			defined 
		in \eqref{eqn:prox_lin_map} seems unclear,
			and
	therefore we 
		are unable 
	to find an appropriate damping parameter. 
Instead 
	we 
		will settle 
	for a local escape guarantee 
		supplied by 
	the center stable manifold theorem. 
We 
	leave 
it as an intriguing open question 
	to obtain 
global escape guarantees 
	for 
	the damped proximal linear algorithm.

A first difficulty 
	in
applying the center stable manifold theorem 
	is 
that the Jacobian $\nabla S(\bar x)$ 
	at 
	the saddle point $\bar x$ 
		may not be
	invertible. 
Consequently, 
	we 
		will damp 
	the proximal linear method, 
		forcing the update 
			to be 
		a local diffeomorphism.
To compute 
	an 
	appropriate threshold 
		for
		the damping parameter 
	we 
		will need 
	to estimate the operator norm 
		of 
		$\nabla S(\bar x)$. 
This 
	is 
the content 
	of 
	the following lemma.
\begin{lem}[The slope at the critical points]\label{lem:loc_lip_const}
Consider the composite optimization problem \eqref{eqn:conv_comp} 
	and 
choose any $\mu\in (0,(\rho+2\beta)^{-1})$. 
Then 
	for all points $x\in\R^d$ 
	and 
	all critical points $\bar x\in\R^d$, 
the proximal linear map $S(\cdot)$ 
	defined in \eqref{eqn:prox_lin_map}
satisfies
$$\|S(x)-\bar x\|\leq\left(1  + \sqrt{\frac{2\beta\mu}{1-\mu\beta-\mu\rho}}\right)\cdot \max\left\{1,\frac{\mu\rho+\mu\beta}{1-\mu\rho-2\mu\beta}\right\}\cdot \|x-\bar x\|.$$
\end{lem}
\begin{proof}
To simplify notation, 
	define 
	the map $$\zeta(x,y)=F(x)+\nabla F(x)(y-x).$$ 
Set  $\gamma:=\mu^{-1}-\beta$, fix an arbitrary point $x\in \R^d$, 
	and 
define
\begin{align*}
x^+:=S(x)\qquad \textrm{and}\qquad
\hat x:=\prox_{f/\gamma}(x).
\end{align*}
Using strong convexity of the prox-linear and proximal subproblems and the estimate \eqref{eqn:quadrapprox}, we successively compute 
\begin{align*}
h(\hat x)+r(\hat x)+\frac{\gamma}{2}\|\hat x-x\|^2&\leq h(x^+)+r(x^+)+\frac{\gamma}{2}\|x^+-x\|^2-\frac{\gamma-\rho-\beta}{2}\|x^+-\hat x\|^2\\
&\leq h(\zeta(x,x^+))+r(x^+)+\frac{\gamma+\beta}{2}\|x^+-x\|^2-\frac{\gamma-\rho-\beta}{2}\|x^+-\hat x\|^2\\
&\leq h(\zeta(x,\hat x))+r(\hat x)+\frac{\gamma+\beta}{2}\|\hat x-x\|^2-(\gamma-\rho)\|x^+-\hat x\|^2\\
&\leq h(\hat x)+r(\hat x)+\frac{\gamma+2\beta}{2}\|\hat x-x\|^2-(\gamma-\rho)\|x^+-\hat x\|^2.
\end{align*}
Rearranging yields the estimate
\begin{align*}
(\gamma-\rho)\|x^+-\hat x\|^2\leq 2\beta\|\hat x-x\|^2&=2\beta\gamma^{-2}\|\nabla f_{1/\gamma}(x)\|^2.
\end{align*}
Therefore, using Lipschitz continuity of the gradient $\nabla f_{1/\gamma}$ (Lemma~\ref{lem_prop_moreau})  and the triangle inequality yields
\begin{align*}
\|x^+- \bar x\|&\leq \left(\gamma^{-1} + \sqrt{\frac{2\beta\gamma^{-2}}{\gamma-\rho}}\right)\cdot \max\left\{\gamma,\frac{\rho+\beta}{1-\gamma^{-1}(\rho+\beta)}\right\}\cdot \|x-\bar x\|\\
&= \left(1  + \sqrt{\frac{2\beta\mu}{1-\mu\beta-\mu\rho}}\right)\cdot \max\left\{1,\frac{\mu\rho+\mu\beta}{1-\mu\rho-2\mu\beta}\right\}\cdot \|x-\bar x\|,
\end{align*}
as claimed.
\end{proof}

We are now ready to deduce that the damped proximal linear method almost locally escapes any composite strict saddle point.

\begin{thm}[Proximal linear method: local escape]
Consider the composite problem~\eqref{eqn:conv_comp} 
	and 
let $\bar x$ 
	be 
any composite strict saddle point. 
Choose any constant $\mu\in(0,(\rho+2\beta)^{-1})$
	and 
a damping parameter $\alpha\in (0,1)$ satisfying
$$\alpha\cdot \left(1+\left(\left(1  + \sqrt{\frac{2\beta\mu}{1-\mu\beta-\mu\rho}}\right)\cdot \max\left\{1,\frac{\mu\rho+\mu\beta}{1-\mu\rho-2\mu\beta}\right\}\right)\right)<1.$$
Define the damped proximal linear update 
$$T(x)=(1-\alpha)x+\alpha S(x),$$
where $S(\cdot)$ is the proximal linear map defined in \eqref{eqn:prox_lin_map}. 
Then there 
	exists 
a neighborhood $U$ 
	of
	$\bar x$ 
	such that the set of initial conditions
$$\{x\in U: S^k(x)\in U\textrm{ for all }k\geq 0\}$$
has zero Lebesgue measure.
\end{thm}
\begin{proof}
First, 
	using 
		Theorem~\ref{thm:saddle_prox_lin} 
			and 
		Lemma~\ref{lem:relax_and_split}, 
	we 
		deduce 
	that $\bar x$ 
		is 
	an unstable fixed point 
		of 
		$\bar x$. 
Let us next 
	verify
 that $T$
 	is 
a local diffeomorphism 
	around 
	$\bar x$. 
To see this, 
	observe 
$$\nabla T(\bar x)=I-\alpha (I-\nabla S(\bar x)).$$
Using Theorem~\ref{lem:loc_lip_const}, 
	we 
		deduce 
		$\alpha\|I-\nabla S(\bar x)\|_{{\rm op}}<1$
			and 
		therefore $T$ is invertible. 
An application 
	of 
	the center stable manifold theorem (Theorem~\ref{thm:center_stab}) 
		completes 
	the proof.
\end{proof}

\section{Convergence of relaxed descent methods}\label{sec:makeitrain}
Thus far, all of our escape theorems made an assumption that the iterate sequence generated by the algorithms converges. In this section, we verify this assumption for the damped proximal point, proximal gradient, and proximal linear methods.
Taking a general view, 
	we
		see 
	that the iterative methods of this paper can be understood within a broad family of damped model-based algorithms for minimizing a function $f$.
These algorithms
	construct
iterates $x_0, x_1 \ldots$ 
	by repeatedly 
		minimizing 
	a local model $f_x(\cdot)$ of the function and moving in the direction of its minimizer.
More specifically, in the section we suppose that there exist constant $\functionweak,\modelweak,\approxacc>0$ such that the 
the following properties hold:
\begin{enumerate}[label=$\mathrm{(A\arabic*)}$]
\item \label{ass:func_weak}The function $f \colon \RR^d \rightarrow \RR \cup \{\infty \}$ is closed and $\functionweak$-weakly convex.
\item \label{ass:model_approx} For all $x \in \RR^d$ there exists a closed $\modelweak$-weakly convex function $f_x \colon \RR^d \rightarrow \RR\cup\{\infty\}$ satisfying 
$$
|f(y) - f_x(y)| \leq \frac{\approxacc}{2}\|y - x\|^2\qquad \textrm{for all }y\in\R^d.
$$
\end{enumerate}
Under these assumptions
	we 
		will study 
	how the following algorithm behaves:
	given iterates $x_0, \ldots, x_t$ define
\begin{align*}
\begin{aligned}
y_{t} &= \argmin_{y \in \RR^d}\left\{ f_{x_t}(y) + \frac{\algstep}{2}\|y - x_t\|^2\right\} \\
x_{t+1} &= (1-\alpha) x_t + \alpha y_t,
\end{aligned}\tag{$\mathcal{MBA}$} \label{eq:MBA}
\end{align*}
where $\algstep > 0$ and $\alpha > 0$ are fixed constants, determined below. 

To analyze this algorithm, 
	we 
		rely
	on 
		the seminal paper~\cite{attouch2013convergence}.
There, 
	the authors 
		identified
	three conditions, 
guaranteeing global convergence 
	of 
		a sequence $\{z_t\}$ of ``algorithm iterates"
	to a critical point 
		of 
		a  closed function $g \colon \RR^d \rightarrow \RR \cup\{\infty \}$.
Namely, 
	they assume 
there exist $a, b > 0$ such that the following holds: 
\begin{enumerate}[label=$\mathrm{(B\arabic*)}$]
\item {\bf (Sufficient Decrease.)} \label{eq:att1} For each $t \in \NN$, we have
$$
g(z_{t+1}) + a\|z_{t+1} - z_t\|^2 \leq g(z_t)
$$
\item {\bf (Relative Error Conditions.)}\label{eq:att2} For each $t \in \NN$ there exists $w_{t+1} \in \partial g(z_{t+1})$ such that 
$$
\|w_{t+1} \| \leq b \|z_{t+1} - z_t\|
$$
\item {\bf (Continuity Condition.)} \label{eq:att3}There exists a subsequence $\{z_{t_j}\}$ and $\tilde z$ such that 
$$
z_{t_j} \rightarrow \tilde z \text{ and } g(z_{t_j}) \rightarrow g(\tilde z), \qquad \text{ as } j \rightarrow \infty.
$$
\end{enumerate}

The above assumptions alone
	may not guarantee
that $z_t$ converges to a critical point of $g$. 
Instead, the authors of~\cite{attouch2013convergence}
	restrict their focus to the broad class of functions satisfying the \emph{Kurdyka-\L{}ojasiewicz property}. 
\begin{definition}[K{\L} Function]{\rm 
Let $g \colon \RR^d \rightarrow \RR \cup \{\infty\}$ be a closed function.
We say that $g$ has the \emph{Kurdyka-\L{}ojasiewicz (KL) property} at a point $\bar x$, where $\partial g(\bar x)$ is nonempty,  if there exists $\varepsilon \in (0, +\infty]$, a neighborhood $U$ of $\bar x$, and a continuous convex function $\varphi \colon [0, \varepsilon) \rightarrow \RR_+$ satisfying
\begin{enumerate}
	\item $\varphi(0) = 0$, 
	\item $\varphi$ is $C^1$ on $(0, \varepsilon)$ with $\varphi' > 0$, and
	\item the K{\L} inequality 
	$$
	\dist(0, \partial g(x)) \geq \frac{1}{\varphi'(g(x) - g(\bar x))},
	$$
	holds for all $x \in U$ satisfying  $g(\bar x)<g(x)<g(\bar x)+\varepsilon$.
\end{enumerate} 
If $g$ satisfies the K{\L} property at each point $x$, with $\partial g(x)\neq\emptyset$, then $g$ is called a \emph{K{\L} function.}}
\end{definition}	

The class of K{\L} functions 
	is 
	broad,
	containing 
		 all closed semialgebraic functions and more broadly any functions definable in an o-minimal structure, 
	as shown in the pioneering work~\cite{bolte2007clarke}.
	Under these assumptions 
	we have the following theorem from~\cite[Theorem 2.9]{attouch2013convergence}.

\begin{thm}\label{thm:attouch}
Let $ g \colon \RR^d \rightarrow \RR\cup\{\infty\}$ be a closed function. Consider a sequence $x_t$ that satisfies \ref{eq:att1}, \ref{eq:att2}, and \ref{eq:att3}. 
If $g$ satisfies the K{\L} property at some cluster point $\tilde x$, then $\tilde x$ is a critical point of $g$, the entire sequence $x_k$ converges to $\tilde x$, and the sequence $x_t$ has finite length 
$$
\sum_{t=0}^\infty \|x_{t+1} - x_t\| < + \infty.
$$
\end{thm}

In the remainder of this section,
	we 
		will verify 
	assumption~\ref{eq:att1}, \ref{eq:att2}, and \ref{eq:att3}
		for
		the sequence $\{z_t\} = \{x_t\}$
			and
		the Moreau envelope
		$
			g := f_{1/\mstep},
		$
where $\mstep$ will be chosen in a moment. 
Since the critical points 
	of 
		$f$ 
			and 
		$f_{1/\mstep}$ 
	agree, 
	the result will imply convergence to critical points of $f$. 
To do so,
	we 
		employ 
	one final assumption. 
\begin{enumerate}[label=$\mathrm{(A\arabic*)}$, start=3]
\item \label{ass:moreau} For every $\mstep > 0$, the Moreau envelope $f_{1/\mstep}$ is a K{\L} function.
\end{enumerate}
Although assumption~\ref{ass:moreau}
	may appear 
hard to verify, 
	it 
		holds  
	whenever 
		$f$ is semialgebraic
	since in this case $f_{1/\mstep}$ 
		 is 
	also semialgebraic. More generally, the analogous statement holds if $f$ is definable in an o-minimal structure.
The following is the main result of this section.
\begin{thm}[Convergence of relaxed model-based methods]\label{thm:model_descent}
Suppose that $\alpha\in (0, 1]$, that 
$
\algstep > \max\{ \modelweak, 2\functionweak, \frac{4\approxacc + \functionweak + \modelweak}{2}\},
$ and that assumptions~\ref{ass:func_weak} and~\ref{ass:model_approx} hold. Then for all $T \geq 0$, we have
$$
\min_{t=0, \ldots, T}  \|\nabla f_{1/\mstep}(x_t)\| \leq \sqrt{\frac{f_{1/\mstep}(x_0) - \inf f}{\frac{\alpha (2\mstep- \functionweak - \modelweak - \approxacc)}{2\mstep(\mstep + \algstep - \functionweak - \modelweak ) }(T+1)}}.
$$
where $\mstep =  (1/2)\tau + (1/4)(\functionweak + \modelweak)$.
Moreover, if~\ref{ass:moreau} also holds and the sequence $\{x_t\}$ has a cluster point $\bar x$, then $\bar x$ is critical for $f$ and the entire sequence $\{x_t\}$ converges to $\bar x$. Moreover, the sequence $\{x_t\}$ has finite length.
$$
\sum_{t=0}^\infty \|x_{t+1} - x_t\| < +\infty.
$$ 
\end{thm}

This result is new and may be of independent interest. In particular, the conclusion of the theorem extends the 
 convergence guarantees for the proximal linear method developed in \cite{pauwels2016value} to all relaxed model-based algorithms.

\subsection{Proof of Theorem~\ref{thm:model_descent}}

We are free to choose the parameter $\hat \rho$ defining the Moreau envelope. To this end, we will need
	the existence of  a parameter $\mstep$, satisfying the following inequalities. 
\begin{lem}\label{lem:deepinequalities}
Under the assumptions of Theorem~\ref{thm:model_descent}, it holds that $\mstep > \functionweak$ and
\begin{enumerate}[noitemsep]
\item \label{item:relaxparams1} $\algstep - \mstep - \approxacc > 0$,
\item \label{item:relaxparams2}$2\mstep- \functionweak - \modelweak - \approxacc > 0$,
\item $\mstep + \algstep - \functionweak - \modelweak > 0$,
\item $1 - \frac{2\mstep- \functionweak - \modelweak - \approxacc}{\mstep + \algstep - \functionweak - \modelweak  } > 0$.
\end{enumerate}
\end{lem}
\begin{proof}
Note 
	that $\mstep > \algstep/2 > \functionweak >  0$ 
	and  
	that $\mstep = \algstep - \approxacc - \varepsilon/2$ for $\varepsilon =  (2\algstep - 4\approxacc - \functionweak - \modelweak)/2 > 0 $. 
To prove the first inequality, notice that 
 $\tau - \mstep  - \approxacc = \varepsilon/2 > 0$.
To prove the second inequality, notice that 
$$
2\mstep - \functionweak - \modelweak - \approxacc = 2\algstep - 4\approxacc - \functionweak - \modelweak  - \varepsilon = \varepsilon > 0.
$$
To prove the third inequality, observe
$$
\mstep + \algstep - \functionweak - \modelweak >  \mstep + (\mstep + \approxacc) - \functionweak - \modelweak \geq 2\mstep - \functionweak - \modelweak -\approxacc > 0, $$
where the first and second inequalities follow from items~\ref{item:relaxparams1} and~\ref{item:relaxparams2}, respectively. To prove the fourth inequality, we compute
$$
1 - \frac{2\mstep- \functionweak - \modelweak - \approxacc}{\mstep + \algstep - \functionweak - \modelweak  } =  \frac{\approxacc + \algstep - \mstep}{\mstep + \algstep - \functionweak - \modelweak  } = \frac{2\approxacc + \varepsilon/2}{\mstep + \algstep - \functionweak - \modelweak  }> 0,
$$
as desired.
\end{proof}
			
Throughout the rest of this section, we fix a constant $\mstep$ satisfying the conditions of Lemma~\ref{lem:deepinequalities}. 
Critical to our proof 
	is
the following lemma,
	comparing the proximal point 
$$
\hat x_t := \prox_{f/\mstep}(x_t)
$$
	to the ``approximately proximal point'' $y_t$. 
A closely 
	related
estimate
	appeared
		in~\cite[Lemma 4.2]{davis2019stochastic},
	driving the convergence analysis of that paper.
\begin{lem}\label{lem:criticalLem}
It holds that 
$$
 \|\hat x_{t} - y_t\|^2 \leq  \|\hat x_t - x_t\|^2 - \frac{2\mstep- \functionweak - \modelweak - \approxacc}{\mstep + \algstep - \functionweak - \modelweak  } \|\hat x_t - x_t\|^2  - \frac{\algstep - \mstep - \approxacc }{\mstep + \algstep - \functionweak - \modelweak }\|x_{t} - y_t\|^2.
$$
\end{lem}
\begin{proof}
Since the function $y \mapsto f(y) + \frac{\mstep}{2} \|y - x_t\|^2$ is $(\mstep - \functionweak)$-strongly convex and $\hat x_t$ is its minimizer, we have 
\begin{align*}
\frac{\mstep - \functionweak }{2} \|\hat x_{t} - y_t\|^2 &\leq \left(f(y_{t}) + \frac{\mstep}{2} \|y_{t}-x_t\|^2\right)- \left(f(\hat x_t) + \frac{\mstep}{2} \|\hat x_t - x_t\|^2\right).
\end{align*}
Consequently, using the double-sided model property \ref{ass:model_approx}, we find 
\begin{equation}\label{eqn:basicyo}
\frac{\mstep - \functionweak }{2} \|\hat x_{t} - y_t\|^2 \leq f_{x_t}(y_t) - f_{x_t}(\hat x_t) + \frac{\mstep + \approxacc}{2} \|x_t - y_{t}\|^2 - \frac{\mstep - \approxacc}{2} \|\hat x_t - x_t\|^2.
\end{equation}
Since the function $y \mapsto f_{x_t}(y) + \frac{\algstep}{2} \| y - x_t\|^2$ is $(\algstep - \modelweak)$-strongly convex and $y_t$ is its minimizer, we have
$$
 f_{x_t}(y_t) - f_{x_t}(\hat x_t)  \leq \frac{\algstep}{2}\|\hat x_t - x_t\|^2 - \frac{\algstep}{2} \|  y_t - x_t\|^2 - \frac{\algstep - \modelweak}{2} \|y_t - \hat x_t\|^2.
$$
Combining this estimate with \eqref{eqn:basicyo}, we compute 
\begin{align*}
\frac{\mstep - \rho }{2} \|\hat x_{t} - y_t\|^2 &\leq \frac{\algstep}{2}\|\hat x_t - x_t\|^2 - \frac{\algstep}{2} \|  y_t-x_{t}\|^2 - \frac{\algstep - \modelweak}{2} \|y_{t} - \hat x_t\|^2\\
&\hspace{20pt}  + \frac{\mstep + \approxacc}{2} \|x_t - y_t\|^2 - \frac{\hat \rho - \approxacc}{2} \|\hat x_t - x_t\|^2\\
&= \frac{\approxacc + \algstep - \mstep}{2}\|\hat x_t - x_t\|^2 +  \frac{\mstep + \approxacc - \algstep}{2}\|x_{t} - y_t\|^2 - \frac{\algstep - \modelweak}{2} \|y_t - \hat x_t\|^2.
\end{align*}
Rearranging, we conclude 
$$
\frac{\mstep + \algstep - \functionweak - \modelweak }{2} \|\hat x_{t} - y_t\|^2  \leq \frac{\approxacc + \algstep - \mstep}{2}\|\hat x_t - x_t\|^2 +  \frac{\mstep + \approxacc - \algstep}{2}\|x_{t} - y_t\|^2.
$$
Dividing both sides by $\frac{\mstep + \algstep - \functionweak - \modelweak }{2}$, we achieve the result: 
\begin{align*}
 \|\hat x_{t} - y_t\|^2  &\leq \frac{\approxacc + \algstep - \mstep}{\mstep + \algstep - \functionweak - \modelweak  }\|\hat x_t - x_t\|^2 +  \frac{\mstep + \approxacc - \algstep}{\mstep + \algstep - \functionweak - \modelweak }\|x_{t} - y_t\|^2\\
 &= \|\hat x_t - x_t\|^2 - \left(1 - \frac{\approxacc + \algstep - \mstep}{\mstep + \algstep - \functionweak - \modelweak } \right)\|\hat x_t - x_t\|^2  + \frac{\mstep + \approxacc - \algstep}{\mstep + \algstep - \functionweak - \modelweak  }\|x_{t} - y_t\|^2\\
 &= \|\hat x_t - x_t\|^2 - \frac{2\mstep- \functionweak - \modelweak - \approxacc}{\mstep + \algstep - \functionweak - \modelweak  } \|\hat x_t - x_t\|^2  - \frac{\algstep - \mstep - \approxacc }{\mstep + \algstep - \functionweak - \modelweak }\|x_{t} - y_t\|^2.
\end{align*}
This completes the proof of the lemma.
\end{proof}

The following lemma verifies the Assumption~\ref{eq:att1}.
\begin{lem}[Sufficient Decrease]
We have
$$
f_{1/\mstep} (x_{t+1}) \leq f_{1/\mstep}(x_t) - \frac{\mstep(\algstep - \mstep - \approxacc )}{2\alpha(\mstep + \algstep - \functionweak - \modelweak )}\|x_{t+1} - x_t\|^2 - \frac{\alpha (2\mstep- \functionweak - \modelweak - \approxacc)}{2\mstep(\mstep + \algstep - \functionweak - \modelweak ) } \| \nabla f_{1/\mstep}(x_t)\|^2.
$$
In particular, $f_{1/\mstep}$ and $\{x_t\}$ satisfy~\ref{eq:att1}. Moreover, for all $T \geq 0$, we have 
$$
\min_{t=0, \ldots, T}  \|\nabla f_{1/\mstep}(x_t)\|^2 \leq \frac{1}{T+1}\sum_{t=0}^T \|\nabla f_{1/\mstep}(x_t)\|^2 \leq \frac{f_{1/\mstep}(x_0) - \inf f}{\frac{\alpha (2\mstep- \functionweak - \modelweak - \approxacc)}{2\mstep(\mstep + \algstep - \functionweak - \modelweak ) }(T+1)}
$$
\end{lem}
\begin{proof}
We successively compute  
\begin{align}
f_{1/\mstep} (x_{t+1}) &= f(\hat x_{t+1}) + \frac{\mstep}{2} \|\hat x_{t+1} - x_{t+1}\|^2\notag\\
&\leq  f( \hat x_{t}) + \frac{\mstep}{2} \|\hat x_{t} - x_{t+1}\|^2 \notag\\
&=  f( \hat x_{t}) + \frac{\mstep}{2} \|(1-\alpha)(\hat x_{t} - x_{t}) + \alpha(\hat x_{t} - y_t) \|^2 \notag\\
&\leq f( \hat x_{t})  +  \frac{\mstep (1-\alpha)}{2} \|\hat x_{t} - x_t\|^2 +  \frac{\mstep\alpha}{2} \|\hat x_{t} - y_t\|^2\notag\\
&\leq   f( \hat x_{t}) + \frac{\mstep}{2} \|\hat x_{t} - x_t\|^2 \notag\\
&\qquad\quad- \frac{\mstep\alpha}{2} \left(  \frac{2\mstep- \functionweak - \modelweak - \approxacc}{\mstep + \algstep - \functionweak - \modelweak  } \|\hat x_t - x_t\|^2  +\frac{\algstep - \mstep - \approxacc }{\mstep + \algstep - \functionweak - \modelweak }\|x_{t} - y_t\|^2 \right) \label{eqn:put_in_eq}\\
&\leq  f_{1/\mstep}(x_t) - \frac{\mstep\alpha(\algstep - \mstep - \approxacc )}{2(\mstep + \algstep - \functionweak - \modelweak )}\|x_{t} - y_t\|^2 - \frac{\alpha (2\mstep- \functionweak - \modelweak - \approxacc)}{2\mstep(\mstep + \algstep - \functionweak - \modelweak ) } \| \nabla f_{1/\mstep}(x_t)\|^2,\notag
\end{align}
where \eqref{eqn:put_in_eq} 
	follows 
from Lemma~\ref{lem:criticalLem}, 
		and 
the final inequality 
	follows 
since $\mstep (x_t - \hat x_t) = \nabla f_{1/\mstep}(x_t)$. 
To get the descent inequality, it remains to write $x_t - y_t = (x_{t+1} - x_t)/\alpha$.
Finally, the bound on the average gradient norm follows by induction.
\end{proof}

The following Lemma verifies the Assumption~\ref{eq:att2}.
\begin{lem}[Relative Error]
It holds 
$$
\|\nabla f_{1/\mstep}(x_{t+1})\|  \leq \left( \max\left\{\hat \rho, \frac{\rho}{1-\rho/\hat \rho}\right\} + \frac{\mstep}{\alpha}\frac{1}{1 - \sqrt{\left(1 - \frac{2\mstep- \functionweak - \modelweak - \approxacc}{\mstep + \algstep - \functionweak - \modelweak  }\right)}}\right)\|x_{t+1} - x_t \|.
$$
In particular, $f_{1/\mstep}$ and $\{x_t\}$ satisfy~\ref{eq:att2}.
\end{lem}
\begin{proof}
We have 
$$
\|\nabla f_{1/\mstep}(x_{t+1})\| \leq \|\nabla f_{1/\mstep}(x_{t})\| + \max\left\{\hat \rho, \frac{\rho}{1-\rho/\hat \rho}\right\} \|x_{t+1} - x_t\|.
$$
Thus, we want to bound 
$$
 \|\nabla f_{1/\mstep}(x_{t})\| = \mstep \|\hat x_t - x_t\|
$$
by a multiple of $\|x_{t+1} - x_t\|$. This follows by Lemma~\ref{lem:criticalLem}: 
$$
\|\hat x_t - x_t\| \leq \| \hat x_t - y_t\|+ \|y_t - x_t\| \leq \sqrt{\left(1 - \frac{2\mstep- \functionweak - \modelweak - \approxacc}{\mstep + \algstep - \functionweak - \modelweak  }\right)}\|x_t - \hat x_t\| + \|y_t - x_t\|
$$
Rearranging and using the definition $x_t - y_t = (x_{t+1} - x_t)/\alpha$, it holds
$$
\|\hat x_t - x_t \| \leq \frac{1}{1 - \sqrt{\left(1 - \frac{2\mstep- \functionweak - \modelweak - \approxacc}{\mstep + \algstep - \functionweak - \modelweak  }\right)}} \|y_t - x_t \| = \frac{1}{\alpha}\frac{1}{1 - \sqrt{\left(1 - \frac{2\mstep- \functionweak - \modelweak - \approxacc}{\mstep + \algstep - \functionweak - \modelweak  }\right)}} \|x_{t+1} - x_t \|.
$$
The proof is complete.
as desired.\end{proof}

Finally, we can dispense with Assumption~\ref{eq:att3}, which is a simple consequence of the continuity of $f_{\mstep}$. 
\begin{lem}[Continuity Condition]
	The function $f_{\mstep}$ and the sequence $\{x_t\}$ satisfy~\ref{eq:att3}.
\end{lem}

\section*{Acknowledgements}
We thank John Duchi for his insightful comments on an early version of the manuscript. We also thank the anonymous referees for numerous suggestions that have improved the readability of the paper.

			\bibliographystyle{plain}
	\bibliography{bibliography}

\def\cfac#1{\ifmmode\setbox7\hbox{$\accent"5E#1$}\else
  \setbox7\hbox{\accent"5E#1}\penalty 10000\relax\fi\raise 1\ht7
  \hbox{\lower1.15ex\hbox to 1\wd7{\hss\accent"13\hss}}\penalty 10000
  \hskip-1\wd7\penalty 10000\box7}
\begin{thebibliography}{10}

\bibitem{Al-Khayyal-Kyparisis91}
F.~Al-Khayyal and J.~Kyparisis.
\newblock Finite convergence of algorithms for nonlinear programs and
  variational inequalities.
\newblock {\em J. Optim. Theory Appl.}, 70(2):319--332, 1991.

\bibitem{semiconcave}
P.~Albano and P.~Cannarsa.
\newblock Singularities of semiconcave functions in {B}anach spaces.
\newblock In {\em Stochastic analysis, control, optimization and applications},
  Systems Control Found. Appl., pages 171--190. Birkh\"{a}user Boston, Boston,
  MA, 1999.

\bibitem{attouch2013convergence}
H.~Attouch, J.~Bolte, and B.F. Svaiter.
\newblock Convergence of descent methods for semi-algebraic and tame problems:
  proximal algorithms, forward--backward splitting, and regularized
  gauss--seidel methods.
\newblock {\em Mathematical Programming}, 137(1-2):91--129, 2013.

\bibitem{noisysticky}
D.~Avdiukhin, c.~Jin, and G.~Yaroslavtsev.
\newblock Escaping saddle points with inequality constraints via noisy sticky
  projected gradient descent.
\newblock {\em Optimization for Machine Learning Workshop}, 2019.

\bibitem{beck}
A.~Beck and M.~Teboulle.
\newblock A fast iterative shrinkage-thresholding algorithm for linear inverse
  problems.
\newblock {\em SIAM J. Imaging Sci.}, 2(1):183--202, 2009.

\bibitem{bhojanapalli2016global}
S.~Bhojanapalli, B.~Neyshabur, and N.~Srebro.
\newblock Global optimality of local search for low rank matrix recovery.
\newblock In {\em Advances in Neural Information Processing Systems}, pages
  3873--3881, 2016.

\bibitem{bolte2007clarke}
J.~Bolte, A.~Daniilidis, A.~Lewis, and M.~Shiota.
\newblock Clarke subgradients of stratifiable functions.
\newblock {\em SIAM Journal on Optimization}, 18(2):556--572, 2007.

\bibitem{Bon_Shap}
J.F. Bonnans and A.~Shapiro.
\newblock {\em Perturbation Analysis of Optimization Problems}.
\newblock Springer, New York, 2000.

\bibitem{burke_com}
J.V. Burke.
\newblock Descent methods for composite nondifferentiable optimization
  problems.
\newblock {\em Math. Programming}, 33(3):260--279, 1985.

\bibitem{Burke90}
J.V. Burke.
\newblock On the identification of active constraints. {II}.\ {T}he nonconvex
  case.
\newblock {\em SIAM J. Numer. Anal.}, 27(4):1081--1103, 1990.

\bibitem{Burke-More88}
J.V. Burke and J.J. Mor{\'e}.
\newblock On the identification of active constraints.
\newblock {\em SIAM J. Numer. Anal.}, 25(5):1197--1211, 1988.

\bibitem{Calamai-More87}
P.H. Calamai and J.J. Mor{\'e}.
\newblock Projected gradient methods for linearly constrained problems.
\newblock {\em Math. Prog.}, 39(1):93--116, 1987.

\bibitem{charisopoulos2019low}
V.~Charisopoulos, Y.~Chen, D.~Davis, M.~D{\'\i}az, L.~Ding, and
  D.~Drusvyatskiy.
\newblock Low-rank matrix recovery with composite optimization: good
  conditioning and rapid convergence.
\newblock {\em arXiv preprint arXiv:1904.10020}, 2019.

\bibitem{CLSW}
F.H. Clarke, Yu. Ledyaev, R.I. Stern, and P.R. Wolenski.
\newblock {\em Nonsmooth Analysis and Control Theory}.
\newblock Texts in Math. 178, Springer, New York, 1998.

\bibitem{criscitiello2019efficiently}
C.~Criscitiello and N.~Boumal.
\newblock Efficiently escaping saddle points on manifolds.
\newblock {\em arXiv preprint arXiv:1906.04321}, 2019.

\bibitem{davis2019stochastic}
D.~Davis and D.~Drusvyatskiy.
\newblock Stochastic model-based minimization of weakly convex functions.
\newblock {\em SIAM Journal on Optimization}, 29(1):207--239, 2019.

\bibitem{drusvyatskiy2017proximal}
D.~Drusvyatskiy.
\newblock The proximal point method revisited.
\newblock {\em SIAG/OPT Views and News}, 26(2), 2018.

\bibitem{drusvyatskiy2016generic}
D.~Drusvyatskiy, A.D. Ioffe, and A.S. Lewis.
\newblock Generic minimizing behavior in semialgebraic optimization.
\newblock {\em SIAM Journal on Optimization}, 26(1):513--534, 2016.

\bibitem{darxiv}
D.~Drusvyatskiy and A.S. Lewis.
\newblock Optimality, identifiability, and sensitivity.
\newblock {\em arXiv:1207.6628}, 2012.

\bibitem{ident}
D.~Drusvyatskiy and A.S. Lewis.
\newblock Optimality, identifiablity, and sensitivity.
\newblock {\em Math. Program.}, 147(1-2, Ser. A):467--498, 2014.

\bibitem{prox_error}
D.~Drusvyatskiy and A.S. Lewis.
\newblock Error bounds, quadratic growth, and linear convergence of proximal
  methods.
\newblock {\em To appear in Math. Oper. Res., arXiv:1602.06661}, 2016.

\bibitem{comp_DP}
D.~Drusvyatskiy and C.~Paquette.
\newblock Efficiency of minimizing compositions of convex functions and smooth
  maps.
\newblock {\em Mathematical Programming}, 178(1-2):503--558, 2019.

\bibitem{du2017gradient}
S.S. Du, C.~Jin, J.D. Lee, M.I. Jordan, A.~Singh, and B.~Poczos.
\newblock Gradient descent can take exponential time to escape saddle points.
\newblock In {\em Advances in neural information processing systems}, pages
  1067--1077, 2017.

\bibitem{duchi2018stochastic}
J.C. Duchi and F.~Ruan.
\newblock Stochastic methods for composite and weakly convex optimization
  problems.
\newblock {\em SIAM Journal on Optimization}, 28(4):3229--3259, 2018.

\bibitem{Dunn87}
J.C. Dunn.
\newblock On the convergence of projected gradient processes to singular
  critical points.
\newblock {\em J. Optim. Theory Appl.}, 55(2):203--216, 1987.

\bibitem{Ferris91}
M.C. Ferris.
\newblock Finite termination of the proximal point algorithm.
\newblock {\em Math. Program.}, 50(3, (Ser. A)):359--366, 1991.

\bibitem{Flam92}
S.D. Fl{\aa}m.
\newblock On finite convergence and constraint identification of subgradient
  projection methods.
\newblock {\em Math. Program.}, 57:427--437, 1992.

\bibitem{ge2015escaping}
R.~Ge, F.~Huang, C.~Jin, and Y.~Yuan.
\newblock Escaping from saddle points—online stochastic gradient for tensor
  decomposition.
\newblock In {\em Conference on Learning Theory}, pages 797--842, 2015.

\bibitem{ge2017no}
R.~Ge, C.~Jin, and Y.~Zheng.
\newblock No spurious local minima in nonconvex low rank problems: A unified
  geometric analysis.
\newblock In {\em Proceedings of the 34th International Conference on Machine
  Learning-Volume 70}, pages 1233--1242. JMLR. org, 2017.

\bibitem{ge2016matrix}
R.~Ge, J.D. Lee, and T.~Ma.
\newblock Matrix completion has no spurious local minimum.
\newblock In {\em Advances in Neural Information Processing Systems}, pages
  2973--2981, 2016.

\bibitem{hallakfinding}
N.~Hallak and M.~Teboulle.
\newblock Finding second-order stationary points in constrained minimization: A
  feasible direction approach.
\newblock {\em www.optimization-online.org}.

\bibitem{hare_alg}
W.L. Hare and A.S. Lewis.
\newblock Identifying active manifolds.
\newblock {\em Algorithmic Oper. Res.}, 2(2):75--82, 2007.

\bibitem{jin2019local}
C.~Jin, P.~Netrapalli, and M.I. Jordan.
\newblock What is local optimality in nonconvex-nonconcave minimax
  optimization?
\newblock {\em arXiv preprint arXiv:1902.00618}, 2019.

\bibitem{jin2017escape}
R.~Jin, C.and~Ge, P.~Netrapalli, S.M. Kakade, and M.I. Jordan.
\newblock How to escape saddle points efficiently.
\newblock In {\em Proceedings of the 34th International Conference on Machine
  Learning-Volume 70}, pages 1724--1732. JMLR. org, 2017.

\bibitem{Lee:2019:FMA:3349830.3349888}
J.D. Lee, I.~Panageas, G.~Piliouras, M.~Simchowitz, M.I. Jordan, and B.~Recht.
\newblock First-order methods almost always avoid strict saddle points.
\newblock {\em Math. Program.}, 176(1-2):311--337, July 2019.

\bibitem{gradient_descent_jason}
J.D. Lee, M.~Simchowitz, M.I. Jordan, and B.~Recht.
\newblock Gradient descent only converges to minimizers.
\newblock In {\em Conference on learning theory}, pages 1246--1257, 2016a.

\bibitem{lee2013smooth}
J.M. Lee.
\newblock {\em Introduction to smooth manifolds}, volume 218 of {\em Graduate
  Texts in Mathematics}.
\newblock Springer, New York, second edition, 2013.

\bibitem{lee2012manifold}
Sangkyun Lee and Stephen~J Wright.
\newblock Manifold identification in dual averaging for regularized stochastic
  online learning.
\newblock {\em Journal of Machine Learning Research}, 13(Jun):1705--1744, 2012.

\bibitem{LOS}
C.~Lemar\'{e}chal, F.~Oustry, and C.~Sagastiz\'{a}bal.
\newblock The {U}-lagrangian of a convex function.
\newblock {\em Trans. Amer. Math. Soc.}, 352:711--729, 1996.

\bibitem{lewis_active}
A.S. Lewis.
\newblock Active sets, nonsmoothness, and sensitivity.
\newblock {\em SIAM J. Optim.}, 13(3):702--725 (electronic) (2003), 2002.

\bibitem{prox}
A.S. Lewis and S.J. Wright.
\newblock A proximal method for composite minimization.
\newblock {\em Math. Program.}, pages 1--46, 2015.

\bibitem{lewis2013partial}
A.S. Lewis and S.~Zhang.
\newblock Partial smoothness, tilt stability, and generalized {H}essians.
\newblock {\em SIAM Journal on Optimization}, 23(1):74--94, 2013.

\bibitem{MR0298899}
B.~Martinet.
\newblock R\'{e}gularisation d'in\'{e}quations variationnelles par
  approximations successives.
\newblock {\em Rev. Fran\c{c}aise Informat. Rech. Op\'{e}rationnelle},
  4(S\'{e}r. {\rm R}-3):154--158, 1970.

\bibitem{MR0290213}
B.~Martinet.
\newblock D\'{e}termination approch\'{e}e d'un point fixe d'une application
  pseudo-contractante. {C}as de l'application prox.
\newblock {\em C. R. Acad. Sci. Paris S\'{e}r. A-B}, 274:A163--A165, 1972.

\bibitem{mokhtari2018escaping}
A.~Mokhtari, A.~Ozdaglar, and A.~Jadbabaie.
\newblock Escaping saddle points in constrained optimization.
\newblock In {\em Advances in Neural Information Processing Systems}, pages
  3629--3639, 2018.

\bibitem{mord1}
B.S. Mordukhovich.
\newblock {\em Variational analysis and generalized differentiation. {I}},
  volume 330 of {\em Grundlehren der Mathematischen Wissenschaften [Fundamental
  Principles of Mathematical Sciences]}.
\newblock Springer-Verlag, Berlin, 2006.
\newblock Basic theory.

\bibitem{MR0201952}
J.-J. Moreau.
\newblock Proximit\'{e} et dualit\'{e} dans un espace hilbertien.
\newblock {\em Bull. Soc. Math. France}, 93:273--299, 1965.

\bibitem{nesterov2007modified}
Yu. Nesterov.
\newblock Modified gauss--newton scheme with worst case guarantees for global
  performance.
\newblock {\em Optimisation Methods and Software}, 22(3):469--483, 2007.

\bibitem{nesterov2013gradient}
Yu. Nesterov.
\newblock Gradient methods for minimizing composite functions.
\newblock {\em Mathematical Programming}, 140(1):125--161, 2013.

\bibitem{nouiehed2018convergence}
M.~Nouiehed, J.D. Lee, and M.~Razaviyayn.
\newblock Convergence to second-order stationarity for constrained non-convex
  optimization.
\newblock {\em arXiv preprint arXiv:1810.02024}, 2018.

\bibitem{Nurminskii1973}
E.A. Nurminskii.
\newblock The quasigradient method for the solving of the nonlinear programming
  problems.
\newblock {\em Cybernetics}, 9(1):145--150, Jan 1973.

\bibitem{panageas2016gradient}
I.~Panageas and G.~Piliouras.
\newblock Gradient descent only converges to minimizers: Non-isolated critical
  points and invariant regions.
\newblock {\em arXiv preprint arXiv:1605.00405}, 2016.

\bibitem{pauwels2016value}
E.~Pauwels.
\newblock The value function approach to convergence analysis in composite
  optimization.
\newblock {\em Operations Research Letters}, 44(6):790--795, 2016.

\bibitem{prox_reg}
R.A. Poliquin and R.T. Rockafellar.
\newblock Prox-regular functions in variational analysis.
\newblock {\em Trans. Amer. Math. Soc.}, 348:1805--1838, 1996.

\bibitem{con_ter}
R.T. Rockafellar.
\newblock {\em Convex analysis}.
\newblock Princeton Mathematical Series, No. 28. Princeton University Press,
  Princeton, N.J., 1970.

\bibitem{MR0410483}
R.T. Rockafellar.
\newblock Monotone operators and the proximal point algorithm.
\newblock {\em SIAM J. Control Optimization}, 14(5):877--898, 1976.

\bibitem{fav_C2}
R.T. Rockafellar.
\newblock Favorable classes of {L}ipschitz-continuous functions in subgradient
  optimization.
\newblock In {\em Progress in nondifferentiable optimization}, volume~8 of {\em
  IIASA Collaborative Proc. Ser. CP-82}, pages 125--143. Int. Inst. Appl. Sys.
  Anal., Laxenburg, 1982.

\bibitem{RW98}
R.T. Rockafellar and R.J-B. Wets.
\newblock {\em Variational {A}nalysis}.
\newblock Grundlehren der mathematischen Wissenschaften, Vol 317, Springer,
  Berlin, 1998.

\bibitem{paraconvex}
S.~Rolewicz.
\newblock On paraconvex multifunctions.
\newblock In {\em Third {S}ymposium on {O}perations {R}esearch ({U}niv.
  {M}annheim, {M}annheim, 1978), {S}ection {I}}, volume~31 of {\em Operations
  Res. Verfahren}, pages 539--546. Hain, K\"{o}nigstein/Ts., 1979.

\bibitem{Shapiro1985}
A.~Shapiro.
\newblock Second order sensitivity analysis and asymptotic theory of
  parametrized nonlinear programs.
\newblock {\em Mathematical Programming}, 33(3):280--299, Dec 1985.

\bibitem{shub2013global}
M.~Shub.
\newblock {\em Global stability of dynamical systems}.
\newblock Springer Science \& Business Media, 2013.

\bibitem{sun2015nonconvex}
J.~Sun, Q.~Qu, and J.~Wright.
\newblock When are nonconvex problems not scary?
\newblock {\em arXiv preprint arXiv:1510.06096}, 2015.

\bibitem{sun2018geometric}
J.~Sun, Q.~Qu, and J.~Wright.
\newblock A geometric analysis of phase retrieval.
\newblock {\em Foundations of Computational Mathematics}, 18(5):1131--1198,
  2018.

\bibitem{sun2019escaping}
Y.~Sun, N.~Flammarion, and M.~Fazel.
\newblock Escaping from saddle points on riemannian manifolds.
\newblock {\em arXiv preprint arXiv:1906.07355}, 2019.

\bibitem{Wright}
S.J. Wright.
\newblock Identifiable surfaces in constrained optimization.
\newblock {\em SIAM J. Control Optim.}, 31:1063--1079, July 1993.

\end{thebibliography}
	
\appendix
\section{Proofs of Theorem~\ref{thm:strict_saddle_hoora_gen} and~\ref{thm:strict_saddle_hoora_gen2}}\label{sec:proof_generic}
In this section, we prove Theorem~\ref{thm:strict_saddle_hoora_gen}. We should note that Theorem~\ref{thm:strict_saddle_hoora_gen}, appropriately restated, holds much more broadly beyond the weakly convex function class. To simplify the notational overhead, however, we impose the weak convexity assumption, throughout. 

 We will require some basic notation from variational analysis; for details, we refer the reader to \cite{RW98}. 
A set-valued map $F\colon\R^d\rightrightarrows \R^m$ assigns to each point $x\in \R^d$ a set $F(x)$ in $\R^m$. The graph of $F$ is defined by
$$\gph F:=\{(x,v):v\in F(x)\}.$$
 A map $F\colon\R^d\rightrightarrows \R^m$ is called {\em metrically regular at} $(\bar x,\bar v)\in \gph F$ if there exists
 a constant $\kappa>0$ such that the estimate holds:
 $$\dist(x,F^{-1}(v))\leq \kappa \dist(v,F(x))$$ 
 for all $x$ near $\bar x$ and all $v$ near $\bar v$.
 If the graph $\gph F$ is a $C^1$-smooth manifold around $(\bar x,\bar v)$, then metric regularity at $(\bar x,\bar v)$ is equivalent to the condition
\cite[Theorem 9.43(d)]{RW98}:\footnote{We should note that metric regularity of $F$ at $(\bar x,\bar v)$ is equivalent to  \eqref{eqn:met_reg_cod2} for an arbitrary set-valued map $F$ with closed graph, provided we interpret  $N_{\gph  F}(\bar x,\bar v)$ as the limiting normal cone \cite[Definition 6.3]{RW98}. }
\begin{equation}\label{eqn:met_reg_cod2}
(0,u)\in  N_{\gph  F}(\bar x,\bar v)\quad \Longrightarrow \quad u=0.
\end{equation}

We begin with the following lemma.
\begin{lem}[Subdifferential metric regularity in smooth minimization]\label{lem:basic_lem}
Consider the optimization problem 
$$\min_{x\in\R^d} f(x)\quad \textrm{subject to}\quad x\in \cM,$$
where $f\colon\R^d\to\R$ is a $C^2$-smooth function and $\cM$ is a $C^2$-smooth manifold.
Let $\bar x\in\cM$ satisfy the criticality condition $0\in \partial f_{\cM}(\bar x)$ and suppose that the subdifferential map $\partial f_{\cM}\colon\R^d\rightrightarrows\R^d$ is  metrically regular at $(\bar x,0)$. Then the guarantee holds:
\begin{equation}\label{eqn:guarant_parab_noteq}
\inf_{u\in \mathbb{S}^{d-1}\cap T_{\cM}(\bar x)} d^2 f_{\cM}(\bar x)(u)\neq 0.
\end{equation}
\end{lem}	
\begin{proof}
First, appealing to \eqref{eqn:met_reg_cod2}, we conclude that the implication holds:
\begin{equation}\label{eqn:met_reg_cod}
(0,u)\in  N_{\gph \partial f_{\cM}}(\bar x,0)\quad \Longrightarrow \quad u=0.
\end{equation}
Let us now interpret the condition \eqref{eqn:met_reg_cod} in Lagrangian terms. To this end,
let $G=0$ be the local defining equations for $\cM$ around $\bar x$. Define the Lagrangian function 
$$\mathcal{L}(x,\lambda)=f(x)+\langle G(x),\lambda \rangle,$$
and let $\bar \lambda$ be the unique Lagrange multiplier vector satisfying $\nabla_x \mathcal{L}(\bar x,\bar \lambda)=0$. 
According to \cite[Corollary 2.9]{lewis2013partial}, we have the following expression:
\begin{equation}\label{eqn:second_eq_metr}
(0,u)\in  N_{\gph \partial f_{\cM}}(\bar x,0)\quad \Longleftrightarrow \quad u\in T_{\cM}(\bar x)\quad \textrm{and}\quad L u \in N_{\cM}(\bar x),
\end{equation}
where $L:=\nabla^2_{xx}\cL(\bar x,\bar \lambda)$ denotes the Hessian of the Lagrangian. Combining \eqref{eqn:met_reg_cod} and  \eqref{eqn:second_eq_metr}, we  deduce that the only vector 
$u\in T_{\cM}(\bar x)$ satisfying $L u\in N_{\cM}(\bar x)$ is the zero vector $u=0$.

Now for the sake of contradiction, suppose that \eqref{eqn:guarant_parab_noteq} fails. Then the quadratic form  $Q(u)=\langle L u,u\rangle$  is nonnegative on $T_{\cM}(\bar x)$ and there exists $0\neq \bar u\in T_{\cM}(\bar x)$ satisfying $Q(\bar u)=0$.
We deduce that $\bar u$ minimizes $Q(\cdot)$ on $T_{\cM}(\bar x)$, and therefore the inclusion $L\bar u\in N_{\cM}(\bar x)$ holds, a clear contradiction.
\end{proof}

The following corollary for active manifolds will now quickly follow.

\begin{cor}[Subdifferential metric regularity and active manifolds]\label{cor:cor_for_partsmooth}
Consider a closed and weakly convex function $f\colon\R^d\to\R\cup\{\infty\}$. Suppose that $f$ admits a $C^2$-smooth active manifold around a critical point $\bar x$ and that the subdifferential map $\partial f\colon\R^d\rightrightarrows\R^d$ is  metrically regular at $(\bar x,0)$. Then $\bar x$ is either a strong local minimizer of $f$ or satisfies the curvature condition  $d^2 f_{\cM}(\bar x)(u)<0$ for some $u\in T_{\cM}(\bar x)$. 
\end{cor}
\begin{proof}
The result \cite[Proposition 10.2]{darxiv} implies that $\gph \partial f$ coincides with $\gph \partial f_{\cM}$ on a neighborhood of $(\bar x,0)$. Therefore the subdifferential map $\partial f_{\cM}\colon\R^d\rightrightarrows\R^d$ is   metrically regular at $(\bar x,0)$. Using Lemma~\ref{lem:basic_lem}, we obtain the guarantee:
$$
\inf_{u\in \mathbb{S}^{d-1}\cap T_{\cM}(\bar x)} d^2 f_{\cM}(\bar x)(u)\neq 0.
$$
If the infimum is strictly negative, the proof is complete. Otherwise the infimum is strictly positive. In this case,  $\bar x$ is a strong local minimizer of $f_{\cM}$, and therefore by \cite[Proposition 7.2]{ident} a strong local minimizer of $f$. 
\end{proof}

We are now ready for the proofs of Theorem~\ref{thm:strict_saddle_hoora_gen} and Theorems~\ref{thm:strict_saddle_hoora_gen2}.

\begin{proof}[Proof of Theorem~\ref{thm:strict_saddle_hoora_gen}]
The result \cite[Corollary 4.8]{drusvyatskiy2016generic} shows that for almost all $v\in\R^d$, the function $g(x):=f(x)-\langle v,x\rangle$ has at most finitely many critical points. Moreover each such  critical point $\bar x$ lies on some $C^2$ active manifold $\cM$ of $g$ and the subdifferential map $\partial g\colon\R^d\rightrightarrows\R^d$ is metrically regular at $(\bar x,0)$. Applying  Corollary~\ref{cor:cor_for_partsmooth} to $g$ for such generic vectors $v$, we deduce that every critical point $\bar x$ of $g$ is either a strong local minimizer or a strict saddle of $g$. The proof is complete.
\end{proof}

\begin{proof}[Proof of Theorem~\ref{thm:strict_saddle_hoora_gen2}]
The proof is identical to that of Theorem~\ref{thm:strict_saddle_hoora_gen} with \cite[Theorem 5.2]{drusvyatskiy2016generic} playing the role of \cite[Corollary 4.8]{drusvyatskiy2016generic}.
\end{proof}

\section{Pathological Example}\label{sec:bad_example_detail}

\begin{thm}\label{thm:path}
	Consider the following function
	$$
	f(x, y) = \frac{1}{2}(|x| + |y|)^2 - \frac{\rho}{2} x^2
	$$
	Assume that $\lambda > \rho$. Define a mapping $T \colon \RR^d \rightarrow \RR$ by the following formula.   
	$$
	S(x, y) = \begin{cases}
	0 & \text{if $(x, y) = 0$;}\\
	\left(0, \frac{\lambda}{1+\lambda} y\right) &  \text{if $|x| \leq \frac{1}{1+\lambda} |y|$;} \\
	\left(\frac{\lambda}{1+\lambda - \rho}x, 0\right) & \text{if $|y| \leq  \frac{1}{1+\lambda -\rho}|x|$,}\\
	\end{cases}
	$$
	and if $\frac{1}{ (1+\lambda - \rho)} |x| < |y| < (1+\lambda) |x|$, we have 
	$$
	S(x,y) = \begin{cases}
	\frac{\lambda }{(1+\lambda)(1+\lambda - \rho)- 1} 
	\begin{bmatrix}
	(1+\lambda) & -1 \\
	-1 & (1+\lambda - \rho)
	\end{bmatrix}\begin{bmatrix}
	x \\ y
	\end{bmatrix}
	& \text{if $\sign(x) = \sign(y)$;}\\
	\frac{\lambda }{(1+\lambda)(1+\lambda - \rho)-1} 
	\begin{bmatrix}
	(1+\lambda) & 1 \\
	1 & (1+\lambda - \rho) 
	\end{bmatrix}\begin{bmatrix}
	x \\ y
	\end{bmatrix} & \text{if $\sign(x) \neq \sign(y)$.}
	\end{cases}
	$$
	Then $\prox_{(1/\lambda) f}(x, y) = S(x,y)$.
\end{thm}
\begin{proof}
	Let us denote the components of $S(x,y)$ by  $(x_+, y_+) = S(x,y)$. By first order optimality conditions, we have $\prox_{(1/\lambda) f}(x, y) = (x_+, y_+) $ if and only if 
	$$
	\lambda (x - (1-(1/\lambda)\rho)x_+, y - y_+) \in \begin{cases}
	\{x_+ + \sign(x_+)|y_+|\}\times \{\sign(y_+)|x_+| + y_+\} & \text{if $x_+ \neq 0$ and $y_+ \neq 0$;}\\
	([-1, 1]y_+)\times \{y_+\} & \text{if $x_+ = 0$ and $y_+ \neq  0$;}\\
	\{x_+\}\times ([-1, 1]x_+) & \text{if $x_+ \neq  0$ and $y_+ =  0$;}\\
	\{0\}\times \{0\} & \text{if $x_+ = 0$ and $y_+ =  0$.}\\
	\end{cases}
	$$
	Let us show that $(x_+, y_+)$ indeed satisfies this inclusion. 
	\begin{enumerate}
		\item If $(x,y) = 0$, then $x_+ = y_+ = 0$, and the pair satisfies the inclusion.
		\item If $|x| \leq \frac{1}{1 + \lambda}|y|$ and $y \neq 0$, then $x_+ = 0$, $y_+ = \frac{\lambda}{1+\lambda}y$, and
		$$
		\lambda (x - (1-(1/\lambda)\rho)x_+, y - y_+) = \lambda\left(x, \frac{1}{1 + \lambda}y\right) \in ([-1, 1]y_+) \times \{y_+\}.
		$$
		Thus the pair satisfies the inclusion.
		\item If $|y| \leq  \frac{1}{1+\lambda -\rho}|x|$ and $x \neq 0$, then $x_+ = \frac{\lambda}{(1+\lambda - \rho)}x$, $y_+ = 0$, and
		$$
		\lambda (x - (1-(1/\lambda)\rho)x_+, y - y_+) = \lambda \left(x - \frac{\lambda-\rho}{(1+\lambda - \rho)}x, y\right) \in \{x_+\}\times ([-1, 1]x_+).
		$$
	\end{enumerate}
	For the remaining two cases, let us assume that $\frac{1}{ (1+\lambda - \rho)} |x| < |y| < (1+\lambda) |x|$. 
	\begin{enumerate}[resume]
		\item If $\sign(x) = \sign(y)$, let $s = \sign(x)$ and note that  
		\begin{align*}
		\begin{bmatrix}
		x_+ \\y_+
		\end{bmatrix}
		&= \frac{\lambda }{(1+\lambda)(1+\lambda - \rho)- 1} 
		\begin{bmatrix}
		(1+\lambda) & -1 \\
		-1 & (1+\lambda - \rho)
		\end{bmatrix}\begin{bmatrix}
		x\\ y
		\end{bmatrix}\\
		&= \frac{s\lambda }{(1+\lambda)(1+\lambda - \rho)- 1} 
		\begin{bmatrix}
		(1+\lambda)|x|  -|y| \\
		-|x| + (1+\lambda - \rho)|y|
		\end{bmatrix}
		\end{align*}
		From this equation 
		we 
		learn 
		$\sign(x_+) = \sign(y_+) = s$. 
		Inverting the matrix 
		we 
		also learn
		\begin{align*}
		\lambda \begin{bmatrix}
		x \\
		y
		\end{bmatrix} =\begin{bmatrix}
		(1+\lambda - \rho)  & 1 \\
		1 & (1+\lambda) 
		\end{bmatrix}
		\begin{bmatrix}
		x_+ \\
		y_+
		\end{bmatrix}
		&=
		\begin{bmatrix}
		x_+ + \lambda(1 - \rho/\lambda)x_+ + y_+ \\
		x_+ + y_+ + \lambda y_+
		\end{bmatrix} \\
		&= \begin{bmatrix}
		x_+ + \sign(x_+) |y_+| +  \lambda(1 - \rho/\lambda)x_+  \\
		\sign(y_+) |x_+| + y_+ + \lambda y_+
		\end{bmatrix}.
		\end{align*}
		Thus the pair satisfies the inclusion.
		\item If $\sign(x) \neq \sign(y)$, let $s = \sign(x)$ and note that  
		\begin{align*}
		\begin{bmatrix}
		x_+ \\y_+
		\end{bmatrix}
		&= \frac{\lambda }{(1+\lambda)(1+\lambda - \rho)- 1} 
		\begin{bmatrix}
		(1+\lambda) & 1 \\
		1 & (1+\lambda - \rho)
		\end{bmatrix}\begin{bmatrix}
		x\\ y
		\end{bmatrix}\\
		&= \frac{s\lambda }{(1+\lambda)(1+\lambda - \rho)- 1} 
		\begin{bmatrix}
		(1+\lambda)|x|  -|y| \\
		|x| - (1+\lambda - \rho)|y|
		\end{bmatrix}
		\end{align*}
		From this equation 
		we 
		learn 
		$\sign(x_+) \neq \sign(y_+) $. 
		Inverting the matrix 
		we 
		also learn
		\begin{align*}
		\lambda \begin{bmatrix}
		x \\
		y
		\end{bmatrix} =\begin{bmatrix}
		(1+\lambda - \rho)  & -1 \\
		-1 & (1+\lambda) 
		\end{bmatrix}
		\begin{bmatrix}
		x_+ \\
		y_+
		\end{bmatrix}
		&=
		\begin{bmatrix}
		x_+ + \lambda(1 - \rho/\lambda)x_+ - y_+ \\
		-x_+ + y_+ + \lambda y_+
		\end{bmatrix} \\
		&= \begin{bmatrix}
		x_+ + \sign(x_+) |y_+| +  \lambda(1 - \rho/\lambda)x_+  \\
		\sign(y_+) |x_+| + y_+ + \lambda y_+
		\end{bmatrix}.
		\end{align*}
		Thus the pair satisfies the inclusion.
	\end{enumerate}
	Therefore, the proof is complete. 
\end{proof}

\begin{cor}[Convergence to Saddles]
Assume the setting 
	of 
	Theorem~\ref{thm:path}. 
Let $\alpha \in (0, 1]$
	and
define the operator $T : = (1-\alpha) I + \alpha S$ on $\RR^2$.
Then the cone 
$
\cK = \{(x,y) \colon |x| \leq (1+\lambda)^{-1}y\}
$
satisfies $T\cK \subseteq \cK$. 
Moreover, for any $(x, y) \in \cK$, it holds that $T^k(x,y) = ((1-\alpha)^k x, (1 - \alpha(1 - \lambda(1+\lambda)^{-1}))^ky)$
	linearly converges
to the origin as $k$ tends to infinity.
\end{cor}
\begin{proof}
Since $\cK$ 
	is 
convex, 
	it 
		suffices 
	to show that $S \cK \subseteq \cK$.  
This 
	follows 
from Theorem~\ref{thm:path}. 
\end{proof}

\end{document}